\documentclass[preprint]{imsart}
\usepackage{amsfonts,amsmath,amssymb,amsthm,bbm}
\usepackage[font=footnotesize,labelfont=bf]{caption}
\usepackage{subcaption}
\usepackage{float,mathtools}
\usepackage[shortlabels]{enumitem}
\usepackage{graphicx,color}
\usepackage{algorithmicx}
\usepackage[ruled]{algorithm}
\usepackage{algpseudocode}
\usepackage[numbers]{natbib}
\usepackage{chngcntr}
\usepackage{apptools}
\usepackage{diagbox} 
\usepackage{hyperref}

\usepackage{nameref,zref-xr}
\zxrsetup{toltxlabel}
\zexternaldocument*{supplementary_material}

\newenvironment{proof*}{\paragraph{Proof}}{\hfill$\blacksquare$}

\newtheorem{remark}{Remark}
\newtheorem{theorem}{Theorem}[section]

\newtheorem{proposition}[theorem]{Proposition}

\newtheorem{definition}[theorem]{Definition}
\newtheorem{example}[theorem]{Example}
\newtheorem{exe}{Assumption}

\newcommand{\norm}[1]{\left\lVert#1\right\rVert}

\newcommand{\bk}{\color{black}}

\begin{document}

\begin{frontmatter}
\title{Subsampling Sparse Graphons Under Minimal Assumptions}
\runtitle{}

\begin{aug}
\author{\fnms{Robert} \snm{Lunde}\ead[label=e1]{rlunde@utexas.edu}}
\and
\author{\fnms{Purnamrita Sarkar}\ead[label=e2]{purna.sarkar@austin.utexas.edu}}
\affiliation{The University of Texas at Austin}

\runauthor{R.Lunde and P.Sarkar}
\runtitle{Subsampling Sparse Graphons}



\address{Department of Statistics and Data Science\\
University of Texas at Austin}

\end{aug}

\begin{abstract}
We establish a general theory for subsampling network data generated by the sparse graphon model.  In contrast to previous work for networks, we demonstrate validity under minimal assumptions; the main requirement is weak convergence of the functional of interest. We study the properties of two procedures: vertex subsampling and $p$-subsampling. For the first, we prove validity under the mild condition that the number of subsampled vertices is $o(n)$. For the second, we establish validity under analogous conditions on the expected subsample size.  For both procedures, we also establish conditions under which uniform validity holds. Furthermore, under appropriate sparsity conditions, we derive limiting distributions for the nonzero eigenvalues of the adjacency matrix of a low rank sparse graphon.  Our weak convergence result immediately yields the validity of subsampling for the nonzero eigenvalues under suitable assumptions. 


\end{abstract}



\end{frontmatter}

\section{Introduction}
\label{Introduction}
The analysis of network data has quickly become one of the most active research areas in statistics.  Many results~\citep{abbe-community-detection-sbm-developments} are now known about canonical network models such as the Stochastic Block Model and its many variants \citep{holland-sbm,airoldi-mmsb,karrer-newman-dcsbm}, the (Generalized) Random Dot Product Model \citep{young-schneiderman-rdpg,generalized-rdpg}, and the Latent Space Model \citep{hoff-raftery-handcock-latent-space-model}, among others.  For the recent developments on minimax rates of nonparametric estimation in the sparse graphon model (defined in Section~\ref{graphons-sparse-graphons}) \citep{gao-minimax-graphon-estimation,klopp-oracle-inequalities}, see~\citet{gao-ma-minimax-network-analysis}. \bk  However, the problem of statistical inference for common network statistics, particularly in the nonparametric setting, has been less studied.  One exception is count statistics, which has been extensively studied by \citet{Bhattacharyya-subsample-count-features}, where the validity of two subsampling schemes for normalized count functionals of sparse graphons is established.

Subsampling is a general methodology that has been shown to exhibit first-order correctness under minimal assumptions for a wide range of data generating processes; for an overview, see \citet{Politis-Romano-Wolf-subsampling}.  While count statistics are an important class of functionals in network analysis, we will show that a more general theory is possible for sparse graphons.  Before discussing our results in more detail, we will introduce the network model that we consider below.       
      
\subsection{Graphons and Sparse Graphons}
\label{graphons-sparse-graphons}
Let $G:=\{G_n\}_{n \in \mathbb{N}}$ denote a sequence of random undirected graphs, and let $\{V(G_n)\}_{n \in \mathbb{N}}$ and $\{E(G_n) \}_{n \in \mathbb{N}}$ denote the corresponding sequence of vertices and edges, respectively.  We will assume that $|V(G_n)| = n$ and that the adjacency matrix corresponding to $G_n$, denoted $A^{(n)}$, is generated by the following model:
\begin{align}
\label{bernoulli-model}
A_{ij}^{(n)} = A_{ji}^{(n)} = \mathbbm{1}( \eta_{ij} \leq h_n(\xi_i, \xi_j))  \stackrel{d}{=} \mathrm{Bernoulli}(h_n(\xi_i, \xi_j))
\end{align}
where $h_n:[0,1]^2 \mapsto [0,1]$ is a symmetric measurable function and $\xi_i \sim \mathrm{Uniform}[0,1]$ for $1 \leq i \leq n$ and  $ \eta_{ij} \sim \mathrm{Uniform}[0,1]$ for $1 \leq i < j \leq n$. We will assume that $A_{ii}^{(n)} = 0$.  For notational convenience, we will drop dependence on $n$ when appropriate.  Formally, we require that $\{\xi_i\}_{i=1}^\infty$ and $\{\eta_{ij} \}_{(i,j) \in \mathbb{N} \times \mathbb{N}}$ are defined on the common probability space $(\Omega, \mathcal{F}, P)$.

When $h_n$ is fixed for all $n$, this model corresponds to the notion of a graphon. Graphons are natural models for graphs that exhibit vertex exchangeability.  The theorems of \citet{aldous-representation-array} and \citet{hoover-exchangeability} imply that any binary jointly exchangeable infinite array (here the adjacency matrix) may be represented as a mixture of processes for which the data generating process is given above, with $h_n$ fixed for all $n$.  For modeling purposes, it is common to fix one component of the mixture, and assume that a size $n$ graph is a partial observation from an infinite array.  Alternatively, graphons arise as limits of convergent graph sequences, where ``convergence" may be defined by one of several equivalent notions (see \citet{Lovacz-Graph-Limits}).

Graphons are known to generate either empty or dense graphs; in the latter case, the expected number of edges is $O(n^2)$.  However, many real-world networks are known to have expected number of edges given by $o(n^2)$; therefore having $h_n$ fixed as $n \rightarrow \infty$ is often inappropriate. Instead, following \citet{Bickel-Chen-on-modularity}, we will consider the following parametrization.  Let:
\begin{align*}
\rho_n = P(A_{ij} = 1) = \int_0^1 \int_0^1 h_n(u,v) \ du \ dv 
\end{align*}

It follows that we may express $h_n(u,v)$ as:
\begin{align*}
h_n(u,v) = \rho_n w_n(u,v)
\end{align*}
where $w_n(u,v): [0,1]^2 \mapsto \mathbb{R}$ is the conditional density of $(\xi_i, \xi_j)$ given $A_{ij}^{(n)} = 1$.  With this parametrization, it is natural to keep $w_n(u,v)$ fixed and let $\rho_n$ vary with n.  Doing so, we arrive at the following model:
\begin{align}
\label{sparsified-graphon}
h_n(u,v) = \rho_n w(u,v) \wedge 1
\end{align}
where $w(u,v) = w(v,u)$, $w(u,v) \geq 0$ and $\int_0^1 \int_0^1 w(u,v) \ du \ dv =1$. By letting $\rho_n \rightarrow 0$ at an appropriate rate, we may generate an appropriately sparse sequence of graphs. With this generalization, note that $w(u,v)$ and $\rho_n$ may lose their original interpretation for any finite $n$.  One may alternatively arrive at this model by considering $L^p$ graphons of \citet{borgs-lp-part-one}, where $w$ is an  element of $L^p([0,1]^2)$ instead of $L^\infty([0,1]^2)$. As noted in the above reference, an unbounded graphon allows power law degree distributions and allows sparse graphs to contain dense spots.   

Notable alternative frameworks for sparse network models include the graphon process or graphex \citep{veitch-roy-graphex,borgs-graphon-process}. These models are based on a different notion of exchangeability. While we focus on sparse graphons only, we will consider a subsampling procedure based on a natural sampling mechanism for graphexes in Section~\ref{p-subsampling-validity}.  See \citet{orbanz-subsampling-large-graphs} for further discussion on natural sampling mechanisms associated with various network models.

\subsection{Overview of Main Results}
 \label{results-overview}
 One of our main contributions is that we establish a general theory of subsampling for network-structured data.  At a high level, subsampling often works for many data generating processes because a size $b$ subsample can itself be viewed as a size $b$ sample from the population.  If the size $b$ subsamples are not too dependent, then they can be aggregated to form a ``good" empirical estimate of the size $b$ sampling distribution.  Furthermore, if the functional of interest converges in distribution, then the sampling distributions of the size $b$ subsamples and the size $n$ sample should be close asymptotically as $b$ grows with $n$.  Let $G_{n,b}$ denote an induced subgraph formed from $b$ vertices of $G_n$.  When $h_n$ is fixed, it is clear that $G_{n,b}$ follows the same distribution as $G_b$.  Therefore, if the induced subgraphs are weakly dependent, then intuitively, subsampling should be valid for graphons under very general conditions.  We will show that this is indeed the case in Section \ref{vertex-subsampling-validity}.

For sparse graphons, observe that $G_{n,b}$ generally does not follow the same distribution as $G_b$.  It will typically be the case that the induced subgraph is sparser, since typically $ \rho_b > \rho_n$.  However, it turns out that, for many functionals, the exact sparsity sequence does not affect the limiting distribution so long as the graph sequence is not too sparse.  Our proof hinges on a novel Hoeffding representation, which allows us to sidestep network dependence.   

In Section \ref{p-subsampling-validity}, we consider a subsampling scheme based on p-sampling. The $p$-sampling procedure, introduced by \citet{Veitch-Roy-Sampling-Graphs}, involves sampling each vertex with probability $p$ and taking the induced subgraph, with isolated nodes removed. See Definition \ref{p-sample-definition} for a formal definition.   We show that $p$-subsampling is valid for sparse graphons under slightly stronger regularity conditions \bk than those for vertex subsampling.

While the conditions in our subsampling theorems are comparable to those in the i.i.d data setting, a question remains as to the applicability of our results; namely, what network functionals converge in distribution?  Prior work by \citet{Bickel-Chen-Levina-method-of-moments} establishes asymptotic normality of certain count functionals and subsampling validity was established for these functionals by \citet{Bhattacharyya-subsample-count-features}.  In general, establishing weak convergence results for networks can be quite challenging, particularly if one considers a sufficiently large class of models.  In Section \ref{eigenvalue-limit}, we derive a limiting joint distribution for the nonzero eigenvalues of a low rank sparse graphon. Section~\ref{uniform_validity} has theoretical results on uniform validity. 

Finally, in Section \ref{simulation-study}, we complement our theoretical analysis with a simulation study. Our simulations suggest that subsampling has satisfying finite-sample properties for the maximum eigenvalue, but coverage for the non-maximal eigenvalues may be poor for sparse graphs.  We also work with a real Facebook social network dataset, where we consider the problem of two-sample testing based on subsampled eigenvalues.  Proofs for results in Sections~\ref{general-subsampling-theorems} and~\ref{eigenvalue-limit} can be found in the Appendix; the other proofs are deferred to the Supplement due to space considerations.


\subsection{Related Work on Inference for Network Data}
Aside from the aforementioned work by \citet{Bhattacharyya-subsample-count-features}, there has been some work involving subsampling/resampling networks in the statistics literature.  For example, \citet{reinert-netdis} develop a subsampling method named \textit{Netdis}, which consists of sampling nodes and forming two-step ego-networks for these nodes. In addition, we have recently become aware of the work of \citet{levin-levina-bootstrap-network-structure}, who consider the problem of bootstrapping functionals that are expressible as U-statistics for random dot product graphs.  The procedures they consider involve estimating the latent positions first and bootstrapping the associated U-statistics.  \citet{green-shalizi-network-bootstrap} also propose two bootstrap procedures for conducting inference for count functionals, one based on an ``empirical graphon'' and the other based on a sieve procedure.

In the computer science literature, various subsampling approaches have also been studied for approximate subgraph counting.  In contrast to \citet{Bhattacharyya-subsample-count-features}, who consider inference for graphon parameters, the aim here is to approximate subgraph counts in $G_n$ up to a multiplicative constant in a computationally efficient manner.  Recently, a literature on sublinear algorithms has emerged, where many of the proposed procedures are based on edge sampling; see \cite{SUBG:Feige-average-degree,SUBG:Goldreich-Ron-Approximating-Average-Parameters,SUBG:assadi-sublinear-arbitrary-subgraph,SUBG:eden-sublinear-triangles,SUBG:gonen-sublinear-stars} and references therein.  Approximate subgraph counting has also received significant attention in various streaming settings; see for example, \cite{Streaming:Reduction-streaming-algorithms, Streaming:doulion,Streaming:kane-arbitrary-subgraphs, Streaming:mcgregor-better-triangles, Streaming:bera-tighter-bounds, Kallaugher:hybrid-triangle, Kallaugher:2019:CCC:3294052.3319706}.   

For the general problem of nonparametric inference for network data, some work has started to emerge involving hypothesis testing. \citet{two-sample-test-network-statistics} develop a framework for nonparametric two-sample testing based on network statistics that satisfy a certain concentration property.  Among the network statistics considered are eigenvalues of independent edge random graphs.  \citet{tang-nonparametric-two-sample-test} also consider nonparametric two-sample testing, but in a setting where the networks are generated by a random dot product graph.  Their test involves estimating the latent positions of the network and using the kernel MMD \citep{kernel-two-sample-test} to test whether the latent positions are generated by the same distribution.


 \section{General Theorems for Subsampling}
 \label{general-subsampling-theorems}
 
 \subsection{Notation}
 Let $\hat{\theta}_n(G_n): \{0,1\}^{n \times n} \mapsto \mathbb{R}$ be an estimator of $\theta(G_\infty)$, where $\theta(G_\infty)$ is some parameter such that $\hat{\theta}_n(G_n) -\theta(G_\infty) \xrightarrow{P} 0$.  We will assume $\hat{\theta}_n(G_n)$ is invariant to permutations of the vertices. The distribution function of interest is given by:
\begin{align}
J_n(t,P_n) = P\left(\tau_n[\hat{\theta}_n(G_n) - \theta(G_\infty)] \leq t \right)
\end{align}
As discussed in Section \ref{results-overview}, for sparse graphons, it will often be the case that the induced subgraph of size $b$, denoted $G_{n,b}$, will be sparser than $G_b$. For each $n \geq 2$ and $2 \leq b_n \leq n$, let $\hat{\theta}_{n,b} : \{0,1 \}^{b \times b} \mapsto \mathbb{R}$ be an estimator defined on the induced subgraph; for reasons that will be explained shortly, $\hat{\theta}_{n,b}$ need not be equal to $\hat{\theta}_{b}$, but will be closely related.  Define the following CDF corresponding to the functional defined on an induced subgraph: 
\begin{align}
J_{n,b}(t, P_n) = P\left(\tau_b[\hat{\theta}_{n,b}(G_{n,b}) - \theta(G_\infty)] \leq t \right) 
\end{align}
We will impose the following condition: 
\begin{exe}
\label{convergence-of-CDF}
For a given sequence $\{ b_n \}_{n \in \mathbb{N}}$ satisfying $b_n \rightarrow \infty$, there exists some non-degenerate limiting distribution $J(t,P_\infty)$, such that, for all continuity points of $J(t,P_\infty)$:
\begin{align*}
|J_n(t, P_n)- J(t, P_\infty)| \rightarrow 0 \  \text{ and } \ | J_{n,b}(t, P_{n}) - J(t, P_\infty) | \rightarrow 0 
\end{align*} 
\end{exe}
%
The network statistics considered will often be normalized by some power of $\rho_n$. Assumption \ref{convergence-of-CDF} requires that the functional converges in distribution even when a size $b$ graph is normalized by $\rho_n$ instead of $\rho_b$. We will provide a concrete example below.  
\begin{example}[Distributions of interest for leading eigenvalue]
\label{eigenvalue-example}
Let $\lambda_1(w)$ denote the parameter of interest, which turns out to only depend on $w$ for graph sequences that are not too sparse; see Section \ref{eigenvalue-limit} for details.  Let $A^{(n,b)}$ denote the adjacency matrix corresponding to a subgraph induced by $b$ nodes.  The corresponding CDFs of interest may be expressed as:
\begin{align*}
\begin{split}
J_n(t, P_n) &= P( \sqrt{n}[\lambda_1(A^{(n)}/n\rho_n) - \lambda_1(w)] \leq t )
\\ J_b(t, P_b) &= P( \sqrt{b}[\lambda_1(A^{(b)}/b\rho_b) - \lambda_1(w)] \leq t )
\\ J_{n,b}(t, P_n) &= P( \sqrt{b}[\lambda_1(A^{(n,b)}/b\rho_n) - \lambda_1(w)] \leq t )
\end{split}
\end{align*}
\end{example}

 Assumption \ref{convergence-of-CDF} imposes implicit limitations on the sparsity level; for count functionals one would typically need that $n\rho_n \rightarrow \infty$, whereas for eigenvalue statistics, $\rho_n = \omega(1/\sqrt{n})$ is required. Even if subsampling is valid for certain sequences, this condition will also impose an implicit restriction on how slowly $b_n$ can grow.  Furthermore, in our definition of $J_n(t, P_n)$, note that $\rho_n$ is unknown.  For ease of exposition, we will not introduce estimation of $\rho_n$ in our theorems below, but in Section \ref{subsampling-eigenvalues} we will show that subsampling may still be used to approximate $J_n(t, P_n)$ when $\rho_n$ is estimated.

Since $J_n(t,P_n)$ is inaccessible, we will approximate it with an empirical quantity defined on subsamples of the data.  Our first subsampling scheme computes the functional of interest on each induced subgraph with $b$ vertices. \citet{Bhattacharyya-subsample-count-features}) referred to this procedure as uniform subsampling; in the present work, we use the terminology vertex subsampling to avoid confusion with the notion of uniform validity. Let $N_n = {n \choose b_n}$; we will drop the subscripts when there is no ambiguity. For a given $n$, let $ S_{b,1}, S_{b,2}, \ldots S_{b, N}$ denote subsets of size $b$ constructed from $\{1, \ldots, n\}$, arranged in any order.  Furthermore, let $G_{n,b,i}$ denote the graph induced by the nodes in $S_{b,i}$.  The resulting empirical CDF may be expressed as follows:      
\begin{align}
L_{n,b}(t,P_n) = \frac{1}{N}\sum_{i=1}^{N} \mathbbm{1}\left( \tau_b[\hat{\theta}_{n,b}(G_{n,b,i}) - \hat{\theta}_n(G_n)] \leq t \right)
\end{align}
Since $\theta(G_\infty)$ is unobservable, it is customary to replace it with its empirical counterpart estimated on $G_n$.  As long as $b_n = o(n)$, this substitution is asymptotically negligible.

%
%

 \subsection{Validity of Vertex Subsampling for Sparse Graphons}
 \label{vertex-subsampling-validity}
 In the theorem below, we establish the validity of vertex subsampling for sparse graphons under minimal conditions.  As mentioned in Section \ref{results-overview}, our proof hinges on a technique involving the independence of induced subgraphs when the node sets are disjoint.  To make the similarites to the theory for i.i.d. processes transparent, we have stated the theorem below in a manner analogous to Theorem 2.2.1 of  \citet{Politis-Romano-Wolf-subsampling}.  A proof is given in Appendix \ref{section:appendix-vertex-subsampling-proof}.       
\begin{theorem}
\label{subsampling-fixed-b}
Assume $\tau_b/\tau_n \rightarrow 0$, $b_n \rightarrow \infty$, and $ b_n = o(n)$.  Further suppose that for the sequence $\{b_n\}_{n \in \mathbb{N}}$, Assumption \ref{convergence-of-CDF} is satisfied.   Then, 
\begin{enumerate}
\item[\bf{i}.] If $t$ is a continuity point of $J(\cdot, P_\infty)$, then $L_{n,b}(t,P_n) \xrightarrow{P} J(t,P_\infty)$.
\item[\bf{ii}.] If $J(\cdot, P_\infty)$ is continuous, then,
\begin{align*}
\sup_{t \in \mathbb{R}} \left|L_{n,b}(t,P_n) - J_n(t, P_n) \right| \xrightarrow{P} 0 
\end{align*} 
\item[\bf{iii}.] Let $c_{n,b}(1-\alpha) = \inf \{ \ t \in \mathbb{R} \ | \ L_{n,b}(t,P_n) \geq 1-\alpha \}$. 

Correspondingly define:
\begin{align*}
c(1-\alpha) = \inf \{ \ t \in \mathbb{R} \ | \ J(t,P_\infty) \geq 1-\alpha \}
\end{align*}
If $J(\cdot, P_\infty)$ is continuous at $c(1-\alpha)$ then,
\begin{align*}
P\left(\tau_n[\hat{\theta}_n(G_n) - \theta(G_\infty)] \leq c_{n,b}(1-\alpha) \right) \rightarrow 1-\alpha \ \text{ as } \ n \rightarrow \infty
\end{align*}
\end{enumerate} 
\end{theorem}
\begin{remark}
The above result may be generalized in several ways.  The latent positions may instead follow any distribution.  In addition, the above theorem holds for  weighted/directed graphs as well as graphon models with  nodal covariates, so long as the pair $(\xi_i, X_i) \in \mathbb{R}^{p}$ is i.i.d.
\end{remark}
We would like to mention that several auxiliary results in \citet{Politis-Romano-Wolf-subsampling} may be proved following similar reasoning.  However, certain results, such as the validity of subsampling with a data-driven subsample size (Theorem 2.7.1 \citep{Politis-Romano-Wolf-subsampling}) require slightly stronger assumptions since $J_{n,b}(t,P_n)$ need not be equal to $J_b(t,P_n)$.  We will not pursue this here, but we would like to mention that a similar issue arises in Section \ref{p-subsampling-validity}; granting Assumption \ref{convergence-of-CDF-2} should be enough to extend the aforementioned result.

\subsection{Validity of p-subsampling for Sparse Graphons}
\label{p-subsampling-validity}
We will now consider the validity of a subsampling procedure that involves repeatedly $p$-sampling a given graph.  The notion of $p$-sampling was introduced by \citet{Veitch-Roy-Sampling-Graphs} in the context of sampling for graphex processes. This procedure is described below.    
\begin{definition}[p-sampling]
\label{p-sample-definition}
A $p$-sample of $G$, denoted $\mathrm{Smpl}_{p}(G)$, is a random subgraph obtained by including vertices independently with probability $p$, and taking the induced subgraph, with isolated vertices removed.     
\end{definition}

Given a size $n$ realization of a graphex, a $p$-sampling of the graph generates a smaller graphex; therefore, $p$-sampling may be viewed as the natural sampling mechanism for this process.  Under weak assumptions on the (expected) subsample size, we will demonstrate that a subsampling procedure based on $p$-sampling works for sparse graphons.  To our knowledge, our result is the first subsampling validity result for a procedure with a random subsample size on each iteration.  We will begin by preparing some notation. 

Since the subsample size is random, we will now denote it with $B_i$.  Let $X_{ij}^{(n)} \sim \mathrm{Bernoulli}(p_n)$; this random variable indicates whether the $j$th node is included in the subsampled graph before the deletion of isolated vertices.  We will assume the Bernoulli trials are generated independently from $G_n$. Denote the $i$th p-sample $\mathrm{Smpl}_{p,i}(G_n)$. Furthermore, let $M$ denote the number of $p$-subsamples.  In addition, we will modify Assumption \ref{convergence-of-CDF} to accommodate the random sample size. Consider the following condition:
\begin{exe}
\label{convergence-of-CDF-2}
For sequences $\{ l_n \}_{n \in \mathbb{N}}$ and $\{ u_n \}_{n \in \mathbb{N}}$ satisfying $l_n < u_n$ and $l_n \rightarrow \infty$, there exists some non-degenerate limiting distribution $J(t,P_\infty)$, such that, for all continuity points of $J(t,P_\infty)$:
\begin{align*}
|J_n(t, P_n) - J(t,P_\infty)| \rightarrow 0 \  \text{ and } \ \sup_{l_n \leq j \leq u_n}| J_{n,j}(t,P_n) - J(t,P_\infty) | \rightarrow 0
\end{align*}
\end{exe}
Now, let:
\begin{align}
\label{minimum-subsample-size}
l_n = \lfloor np_n -3\sqrt{np_n \log n} \rfloor, \ \ \ u_n = \lceil np_n + 3\sqrt{np_n \log n} \rceil
\end{align}
Recall the following Chernoff bound for binomial random variables, which yields, for $ 0 < \epsilon < np_n$:
\begin{align}
\label{chernoff-binomial-bound}
P\left(\biggl|\sum_{j=1}^n X_{ij}^{(n)} - \mathbb{E}[X_{ij}^{(n)}] \biggr| > \epsilon  \right) \leq  2\exp\left(\frac{-\epsilon^2}{3 np_n} \right) 
\end{align}
In light of Assumption \ref{convergence-of-CDF-2} and the Chernoff bound, our choice of $l_n$ and $u_n$ ensures that the sequence with a random sample size converges in distribution with high probability. 
Since the number of vertices of the input graph is random, we will have to be slightly careful in formally defining $\hat{\theta}_{n,B_i}(\cdot)$.  Let $\hat{\theta}_{n,B_i}(\cdot) : \cup_{b=0}^n \{0,1\}^{b \times b } \mapsto \mathbb{R}$ be equal to $\hat{\theta}_{n,j}(\cdot)$ when $|V(\mathrm{Smpl}_{p,i}(G_n)| = j$ for $j \geq 2$ and 0 otherwise.  Define the following functional, which denotes the empirical distribution function of the functional of interest following a $p$-subsampling procedure: 
\begin{align}
L_{n,B}^\prime(t,P_n) =  \frac{1}{M}\sum_{i=1}^{M} \mathbbm{1}\left( \tau_{B_i}[\hat{\theta}_{n,B_i}(\mathrm{Smpl}_{p,i}(G_n)) - \hat{\theta}_n(G_n)] \leq t \right)
\end{align}
     
We are now ready to state our result.  Our proof strategy involves approximating the empirical CDF of the $p$-subsampled functional with a convex combination of U-statistics, corresponding to its conditional expectation given $G_n$.  This approximation demonstrates that vertex sampling and $p$-sampling are closely linked, which is an interesting concept in its own right.  See Appendix \ref{section:appendix-p-subsampling-proof} for details. 

\begin{theorem}[Validity of $p-$subsampling for Sparse Graphons]
\label{p-subsampling-theorem}
 Assume that $\tau_n = n^{\alpha}$ for some $\alpha > 0$, $M \rightarrow \infty$, $np_n \rightarrow \infty$, and $p_n = o(1)$.  Further suppose for $l_n, u_n$ given in (\ref{minimum-subsample-size}), Assumption \ref{convergence-of-CDF-2} is satisfied and $(\rho_n, w(u,v))$ satisfy $\rho_n = \omega(\log l_n/l_n)$, and $w(\xi_i,\xi_j) \geq c$ almost surely, where $c > 0$.  Then, $\bf{i}. - \bf{iii}$. of Theorem \ref{subsampling-fixed-b} hold for $L_{n,B}^\prime(t,P_n)$.     
\end{theorem}%




 \section{Weak Convergence of Nonzero Eigenvalues of Sparse Graphons}
 \label{eigenvalue-limit}
 In this section, we state a weak convergence result for eigenvalues of adjacency matrices generated by sparse graphons of finite rank.  Limit theorems are one of the most important topics in random matrix theory, so our result here is of sufficient independent interest. The main tools we use are results on random matrix approximations of integral operators due to \citet{koltchinksii-gine-kernel-operator} and refined eigenvalue perturbation bounds developed by \citet{eldridge-beyond-davis-kahan}.  To use the latter, an upper bound on the operator norm of the centered adjacency matrix is needed; however, sharp results (e.g. \citet{Vu-Spectral-Norm}) require independent entries. To handle the graphon dependence structure, we develop a technique that involves conditioning on the latent positions and using Egorov's Theorem to control the conditional probability uniformly on a high probability set.      

For Erd\H{o}s-R\`{e}nyi graphs, limiting distributions of the eigenvalues are well-known.  After appropriate centering, the leading eigenvalue of the adjacency matrix generated by Erd\H{o}s-R\`{e}nyi graphs has been shown to converge to Normal distribution at a $\sqrt{n}$-rate by \citet{furedi-komlos-eigenvalue-dist}. However, beyond this, weak convergence results were not previously known for more general classes of models.  

We would like to note that \citet{borgs-convergent-dense-graphs} show that the scaled eigenvalues of an adjacency matrix generated by a graphon converge to a limiting quantity.  Their result can be thought of as a law of large numbers for spectra; what we show is along the lines of a central limit theorem.  We will now introduce some concepts needed to state our result.

Let $w:[0,1]^2 \mapsto \mathbb{R}$ be a symmetric element of $L^2([0,1]^2)$.  Consider the following integral operator associated with $w$, which we will denote $T_w: L^2([0,1]) \mapsto L^2([0,1])$:
\begin{align}
\label{integral-operator}
T_w f = \int_0^1 w(u,v) f(v) dv
\end{align}
By Spectral Theorem, there exists an orthonormal collection of eigenfunctions $\{ \phi_r, \ r \in J\}$, where $J$ is either finite or countable, and a sequence of real numbers $\{\lambda_r, \ r \in J \}$ satisfying $\sum_{r \in J} \lambda_r^2 < \infty$ such that:
\begin{align}
w(u,v) = \sum_{r \in J} \lambda_r \phi_r(u) \phi_r(v)
\end{align}
in the $L^2([0,1]^2)$ sense.  We will consider the ordering $\lambda_1 \geq \lambda_2 \geq \ldots > 0 > \ldots \geq \lambda_{-2} \geq \lambda_{-1}$, where negative indices correspond to negative eigenvalues.  We will not consider zero eigenvalues. We will denote the eigenvalues associated with $w$ as $\lambda_r(w)$.
%

Let $\widetilde{A}^{(n)} = A^{(n)}/n\rho_n$.  In the theorem below, we will establish conditions under which:
\begin{align}
\label{eigenvalue-convergence}
Z_{n,r} := \sqrt{n}[\lambda_r(\widetilde{A}^{(n)}) - \lambda_r(w)]
\end{align}
weakly converges to a limiting distribution $Z_{\infty,r}$.  In fact, we may show stronger statements involving convergence of joint distributions. While we require that the kernel is finite rank, this assumption is, strictly speaking, not necessary for establishing marginal convergence; see Remark \ref{remark-graphon-infinite-rank} for details.  We will now state our result below under various assumptions on the pair $(\rho_n, w(u,v))$.A proof of this theorem is given in Appendix \ref{section:appendix-eigenvalue-limit-proof}.    

\begin{theorem}[Weak Convergence of Nonzero Eigenvalues]
\label{eigenvalue-weak-convergence}
 Let $\{G_n\}_{n \in \mathbb{N}}$ be a sequence of graphs generated by the model (\ref{sparsified-graphon}). Suppose $w(u,v)$ is an element of $L^2([0,1]^2)$ with an eigendecomposition of the form:
\begin{align*}
\label{low-rank-eigendecomposition}
\begin{split}
 w(u,v) = \sum_{r=1}^k \lambda_r \phi_r(u) \phi_r(v) 
\end{split}
\end{align*}
for some $k < \infty$ where $\lambda_r \neq 0$ for all $1 \leq r \leq k$, and one of the following conditions are satisfied:
 \begin{enumerate}[label=\Alph*]
\item \label{bounded-w} (Boundedness)  $ \norm{w(u,v)}_\infty < \infty$ almost surely, $\rho_n = o(1)$ and $\rho_n = \omega(1/\sqrt{n})$.
\item \label{subweibull-w} (sub-Weibull)  There exists $K_1 < \infty$ and $\gamma >0$ such that:
\begin{align}
\mathbb{E}\left[\exp\left(\frac{w(\xi_i, \xi_j)- \mathbb{E}[w(\xi_i, \xi_j)]}{K_1} \right)^\gamma \right] \leq 2
\end{align}    
Furthermore, for $1\leq r \leq k$, the eigenfunctions satisfy:
\begin{align}
\label{eigenfunction-condition}
\int_0^1 \phi_r^4(u) du < \infty
\end{align}
Further suppose that $\rho_n = O(n^{-\delta})$ and $\rho_n = \omega(n^{-1/2+\delta})$ for some $\delta>0$. 
\item \label{general-w} (Moment) Suppose that $\mathbb{E}[w^{s}(\xi_i, \xi_j)] < \infty$ for some $s > 3 + \sqrt{5}$.  Further suppose the eigenfunctions satisfy (\ref{eigenfunction-condition}) and $\rho_n$ satisfies $n^{1+\delta}\rho_{n}^{s-2} = o(1)$ and $\rho_n = \omega\left(n^{(2-s+\delta)/2s}\right)$ for some $\delta > 0$.
 \end{enumerate}
Let $Z_n := (Z_{n,1}, \ldots, Z_{n,k})$, where $Z_{n,r}$ are given by (\ref{eigenvalue-convergence}). Then, there exists some limiting random variable $Z_\infty$ such that:
\begin{align}
Z_n \rightsquigarrow Z_\infty
\end{align}
Furthermore, if $\lambda_1, \ldots, \lambda_k$ are distinct, then $Z_\infty$ is multivariate Gaussian.  
\end{theorem}
\begin{remark}
\label{remark-graphon-infinite-rank}
Following Theorem 5.1 of \citet{koltchinksii-gine-kernel-operator}, one may impose the following condition for graphons that are not finite rank. For some $R_n \rightarrow \infty$, suppose that $w$ satisfies $\sum_{|r| > R_n} \lambda_r^2 = o(n^{-1})$. Further suppose that:
\begin{align*}
\sum_{|r| \leq R_n} \sum_{|s| \leq R_n} \int \phi_r^2 \phi_s^2 dP  \sum_{|r| \leq R_n} \sum_{|s| \leq R_n} (\lambda_r^2 +\lambda_s^2 )\int \phi_r^2 \phi_s^2 dP = o(n) 
\end{align*}
Moreover, suppose that $\sum_{r \in \mathbb{Z}} |\lambda_r| \phi_r^2 \in L^2(P)$. Then, one may show convergence of finite-dimensional distributions of the nonzero eigenvalues of interest. However, it seems that verifying these conditions is non-trivial outside of the finite-rank case.  For instance, it appears that smoothness properties by themselves do not imply the conditions above. 
\end{remark}
We will now discuss some of the conditions in our theorem.  As mentioned previously, we assume that the graphon has finite rank. Several important classes of graphon models fall within this framework.  For instance, stochastic block models and their variants, such as the mixed membership stochastic block model \citep{airoldi-mmsb} and the degree-corrected stochastic block model \citep{karrer-newman-dcsbm}, are finite rank. Subsuming these models are generalized random dot product graphs \citep{generalized-rdpg}, which are a very rich class of models, particularly when the dimension of the latent space is allowed to be any natural number.

In addition, while we require that graphon is rank $k$, the rank need not be known a priori.  For any $ k^\prime \leq k$, the theorem above provides joint convergence of the $k^\prime$ nonzero eigenvalues.  Furthermore, since we derive joint convergence, one may use the Delta Method to derive weak convergence for certain differentiable functions of the nonzero $k^\prime$ eigenvalues. We will provide examples of eigenvalue statistics that may be of interest below.

\begin{example}[Functionals Based on Disparities between Eigenvalues] For the purpose of comparing two networks, it may be of interest to consider certain functionals of two eigenvalues.  Two natural choices are spectral gaps and eigenvalue ratios, as defined below:
\begin{align}
\begin{split}
\hat{\theta}_{\mathrm{gap}}(A^{(n)}) &= \frac{\lambda_1(A^{(n)}) - \lambda_2(A^{(n)})}{n\rho_n} , \ \ \ \hat{\theta}_{\mathrm{ratio}}(G_n) = \frac{\lambda_1(A^{(n)})}{\lambda_{k^\prime}(A^{(n)})}
\end{split}
\end{align}
\end{example}  

\begin{example}[Approximate Trace] Another important functional in network analysis is $\mathrm{tr}(A^p)$.  This functional is closely related to subgraph counts; more precisely, it provides the number of closed walks of length $p$ from any vertex back to itself. It is well known that:
\begin{align*}
\begin{split}
\mathrm{tr}(A^p) &= \sum_{r=1}^n \lambda_r^p(A)
\end{split}
\end{align*}
Consider the following functional:
\begin{align}
\hat{\theta}_{\mathrm{trace},p,k^\prime}(A^{(n)}) = \ \sum_{r=1}^{k^\prime} \lambda_r^p\bigl(A^{(n)}/n\rho_n \bigr)
\end{align}    
When $k^\prime \approx k$, this functional is a suitable approximation to its population counterpart $\sum_{r=1}^k \lambda_r^p(w)$.           
\end{example}  

While our result here is sufficiently general, it should be mentioned that, even for bounded graphons, our theorem requires $\rho_n = \omega(1/\sqrt{n})$. In general, it seems difficult to improve the eigenvalue perturbation bounds to weaken conditions for concentration; \citet{orourke-perturbation-bounds} derive similar bounds under a low rank hypothesis for the mean matrix.  In addition, for sub-Weibull and $L^p$ graphons, we impose lower bounds on the rate of decay of $\rho_n$, which may seem unusual.  These lower bounds allow us to control the difference between eigenvalues of the mean matrix, with entries given by $\rho_n w(\xi_i, \xi_j) \wedge 1$, and a matrix with entries given by $\rho_n w(\xi_i, \xi_j)$.  In essence, if we allow $w$ to be unbounded, we need a sparser graph sequence to observe most of the uncensored values of $w(\xi_i,\xi_j)$.  
\subsection{On Subsampling Eigenvalues}
\label{subsampling-eigenvalues}
One practical limitation with the above result is that $\rho_n$ is assumed to be known; in practice, it will need to be estimated from data. A natural idea is to plug in the following estimator:
\begin{align}
\hat{\rho}_n = \frac{1}{n(n-1)}\sum_{i=1}^n \sum_{j=1}^n A_{ij}
\end{align}
Doing so, we arrive at the following quantity:
\begin{align}
\label{eigenvalue-with-estimated-rho}
\sqrt{n}[\lambda_r(A^{(n)}/n\hat{\rho}_n) - \lambda_r(w)]
\end{align}    

In Theorem 1 \citep{Bickel-Chen-Levina-method-of-moments}, it is shown that $\hat{\rho}_n/\rho_n$ converges to $1$ at a $\sqrt{n}$-rate; therefore, replacing $\rho_n$ with $\hat{\rho}_n$ turns out to be too large of a perturbation to be considered negligible.  One may often use the Delta Method to establish convergence with an estimated sparsity parameter, but this often leads a higher variance. The maximum eigenvalue appears to be an exception; a figure depicting this phenomenon is provided in Supplement \ref{appendix-max-eigenvalue}.

Fortunately, we may still approximate the sampling distribution of the nonzero eigenvalues through subsampling.  In essence, if we plug in $\hat{\rho}_n$ estimated on $G_n$ in each of the subsampled functionals, it will turn out that the estimate is accurate enough relative to the subsampled functionals so that it is asymptotically negligible.  For ease of exposition, we will demonstrate this result solely for vertex subsampling. Define $\hat{L}_{n,b}(t,P_n)$ as follows:
\begin{align}
\hat{L}_{n,b}(t,P_n) = \frac{1}{N}\sum_{i=1}^{N} \mathbbm{1}\left( \sqrt{b}\left[\frac{\lambda_r(A^{(n,b,i)})}{b\hat{\rho}_n} - \frac{\lambda_r(A^{(n)})}{n\hat{\rho}_n}\right]   \leq t \right) 
\end{align}
%
%
The sampling distribution of interest is given by:
\begin{align}
J_n(t, P_n) = P\left( \sqrt{n}\left[\frac{\lambda_r(A^{(n)})}{n\rho_n} - \lambda_r(w)\right]  \leq t \right)
\end{align}
We have the following result. See Appendix \ref{section:appendix-sparsity-subsampling-proof} for a proof.
 
\begin{proposition}[Subsampling Validity for Eigenvalues With Estimated Sparsity Parameter]
\label{estimated-sparsity-parameter}
 Let $\{G_n\}_{n \in \mathbb{N}}$ be a sequence of graphs generated by the model (\ref{sparsified-graphon}), with $w(u,v)$ having an eigendecomposition satisfying (\ref{low-rank-eigendecomposition}).  Moreover, suppose that one of Conditions \ref{bounded-w},\ref{subweibull-w}, \ref{general-w} of Theorem \ref{eigenvalue-weak-convergence} are satisfied for $\rho_{b_n}$, where $b_n = o(n)$ and $b_n \rightarrow \infty$. Further suppose that $t$ is a continuity point of $J(\cdot, P_\infty)$.  Then,
\begin{align*}
\hat{L}_{n,b}(t,P_n) \xrightarrow{P} J(t, P_\infty)
\end{align*}
\end{proposition}

\section{On the Uniform Validity of Subsampling for Sparse Graphons}
\label{uniform_validity}

\label{sec-uniform-coverage}
In the previous sections, we establish conditions under which subsampling is point-wise valid in the sense that, for any fixed $P$,
\begin{align}
\liminf_{n \rightarrow \infty} P\left(\tau_n[\hat{\theta}_n(G_n) - \theta(G_\infty)] \leq c_{n,b}(1-\alpha) \right) \geq 1-\alpha
\end{align}
However, such point-wise statements can be misleading in finite samples, as the rate of convergence to the desired coverage level depends on the unknown distribution $P$ and can be arbitrarily slow. We will establish sufficient conditions under which, for some collection of probability measures $\mathcal{P}_n$, we have the following uniform convergence guarantee:
\begin{align}
\label{uniform-coverage}
 \liminf_{n \rightarrow \infty} \inf_{P_n \in \mathcal{P}_n} P\left(\tau_n[\hat{\theta}_n(G_n) - \theta(G_\infty)] \leq c_{n,b}(1-\alpha) \right) \geq 1-\alpha
\end{align}
Such a statement guarantees that, for any $P_n \in \mathcal{P}_n$ and $\epsilon > 0$, there exists some $n^\prime$ such that for all $n > n^\prime$ the coverage exceeds $1 - \alpha -\epsilon$. By inversion of the confidence interval, the above statement also has implications for hypothesis testing.  Suppose $\phi(G_n)$ is an appropriate binary decision rule for $\theta(G_\infty)$ constructed from the quantiles of $J(t,P_\infty)$  and we are interested in constructing a level $\alpha$ test for a null hypothesis $P_n \in \mathcal{P}_{n,0}$ against an alternative hypothesis $P_n \in \mathcal{P}_{n,1}$.  The statement (\ref{uniform-coverage}) ensures that:
\begin{align*}
\limsup_{n \rightarrow \infty} \sup_{P_n \in \mathcal{P}_{n,0}} \mathbb{E}_{P_n} \phi(G_n) \leq \alpha
\end{align*}       
Uniform validity of confidence intervals formed by subsampling was investigated in \citet{Romano-Shaikh-Uniform} for i.i.d. data. In certain settings, uniform validity is known to fail for subsampling; see for example, \citet{andrews-guggenberger-problem-subsampling}. Therefore, uniform validity is not an immediate consequence of point-wise validity.  We will state general theorems below that establishes conditions under which uniform validity holds; see the Supplementary Material for corollaries of this result that establish uniform validity for count functionals (\ref{uniform-validity-counts}) and eigenvalues (\ref{uniform-validity-eigenvalues}). 

\begin{theorem}[Uniform Validity for Vertex Subsampling]
\label{uniform-coverage-theorem-vertex}
Let $b_n = o(n)$ and suppose that:
\begin{align}
\label{sampling-dist-convergence}
\lim_{n \rightarrow \infty} \sup_{P_n \in \mathcal{P}_n} \sup_{t \in \mathbb{R}} \left| J_{n,b}(t,P_n) -  J_{n}(t; P_n) \right|  = 0
\end{align}
Moreover, suppose that for any $\epsilon > 0$, we have that:
\begin{align}
\label{subsampling-dist-convergence}
 \lim_{n \rightarrow \infty} \sup_{P \in \mathcal{P}_n} P \left( \sup_{t \in \mathbb{R}} \left|L_{n,b}(t,P_n) - U_{n,b}(t,P_n) \right| > \epsilon \right) &= 0
\end{align}
Then, (\ref{uniform-coverage}) holds for $c_{n,b}(1-\alpha) = \{ \inf t \in \mathbb{R} \ | \  L_{n,b}(t,P_n) \geq 1-\alpha \}$.
\end{theorem}

\begin{theorem}[Uniform Validity for p-Subsampling]
\label{uniform-coverage-theorem-p}
Assume that $\tau_n = n^{\alpha}$ for some $\alpha > 0$, $M \rightarrow \infty$, $np_n \rightarrow \infty$, and $np_n = o(n)$.  Further suppose that $(\rho_n, w(u,v))$ satisfy $\rho_n = \omega(\log l_n/l_n)$, and $w(\xi_i,\xi_j) \geq c$ almost surely, where $c > 0$.  In addition, suppose for $l_n, u_n$ given in (\ref{minimum-subsample-size}), we have that:
\begin{align}
\lim_{n \rightarrow \infty} \sup_{P_n \in \mathcal{P}_n} \sup_{l_n \leq j \leq u_n} \sup_{t \in \mathbb{R}} \left| J_{n,j}(t,P_n) -  J(t, P_n) \right|  = 0
\end{align}
Moreover, suppose that for any $\epsilon > 0$, we have that:
\begin{align}
 \lim_{n \rightarrow \infty}  \sup_{P \in \mathcal{P}_n}  P \left( \sup_{t \in \mathbb{R}} \left|L_{n,B}^\prime(t,P_n) - V_{n,B}(t,P_n) \right| > \epsilon \right) &= 0
\end{align}
Then, (\ref{uniform-coverage}) holds for $c_{n,b}(1-\alpha) = \{ \inf t \in \mathbb{R} \ | \  L_{n,b}^\prime(t,P_n) \geq 1-\alpha \}$.
\end{theorem}

\section{Simulation Study}
\label{simulation-study}
We investigate the finite-sample properties of confidence intervals for various eigenvalue statistics formed by subsampling.  We consider two sparse graphon models; for these models, it will be more natural to consider the following parameterization:
\begin{align}
h_n(u,v) = P(A_{ij} = 1 \ | \ \xi_i = u, \xi_j = v ) = \nu_n h(u,v)
\end{align}
where $h(u,v)$ is a dense graphon and $\nu_n$ is a sparsity parameter. It then follows that $\rho_n = \nu_n \int h(\xi_i, \xi_j) \ dP$. We study the performance of our method for sample sizes ranging from $n=1000$ to $n=10000$ and sparsity parameters ranging from $\nu_n = n^{-0.1}$ to $\nu_n = n^{-0.45}$.  We also vary the subsample size so that $b_n$ ranges from $0.1n$ to $0.3n$ for vertex subsampling and $p_n$ from $0.1$ to $0.3$ for $p$-subsampling.  


The data generating processes considered in our simulation study are described below.  For each of these processes, we approximate the true parameter $\lambda_i(w)$ by simulating a graph of size $n=20000$ and $\nu_n = 1$. We then approximate the population parameter with $\lambda_i(A^{(n)})/n \hat{\rho}_n$.  We assess coverage by counting the number of times the parameter falls within our confidence interval; to this end, we simulate the model $500$ times and construct a confidence interval from $N=500$ subsamples for each iteration.              

\subsubsection{Stochastic Block Model (SBM)}
\label{sec:sbm}
We consider a three class stochastic block model studied in \citet{lei-graph-root}, with $ B =\bigl( \begin{smallmatrix} 1/4 & 1/2 & 1/4 \\ 1/2 & 1/4 & 1/4 \\  1/4 & 1/4 & 1/6 \end{smallmatrix}\bigr)$ and $\pi = (0.3, 0.3, 0.4)$.  The corresponding graphon is rank 2 and has one positive and one negative eigenvalue, with $\lambda_1 = 1.035$  and $\lambda_2 = -0.267$. 


\subsubsection{Gaussian Latent Space Model (GLSM)}
\label{sec:glsm}
We also investigate the properties of our procedure for a graphon model that is not low rank.  The following model is a special case of the Gaussian Latent Space Model studied in \citet{rastrelli-properties-latent-network-models}.
Let $\xi_i \sim N(0,1)$ and define:
\begin{align}
\label{eq:glsm_graphon}
h_n(u,v) = \nu_n \exp\bigl( -25(u- v)^2 \bigr) 
\end{align} 
 We study the behavior of our procedure for the top 3 positive eigenvalues; the associated parameters are: $\lambda_1 = 1.311, \lambda_2 = 1.147,$ and $\lambda_3 = 1.011$.
 \subsection{Simulation Results}

 \begin{figure}[h] 
 		\hspace{-1.5em}
\raggedright
  \begin{subfigure}[b]{0.3\linewidth}
    \raggedright
    \includegraphics[trim={0.55cm 0.5cm 0.5cm 0},clip,width=0.92\linewidth]{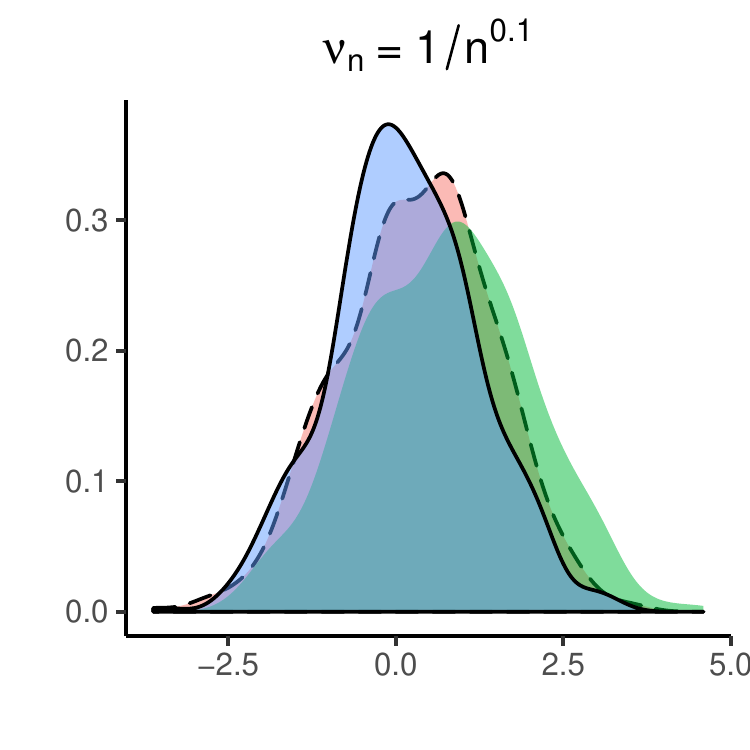} 
  \end{subfigure}
  \begin{subfigure}[b]{0.3\linewidth}
    \raggedright
    \includegraphics[trim={0.55cm 0.5cm 0.5cm 0},clip,width=0.92\linewidth]{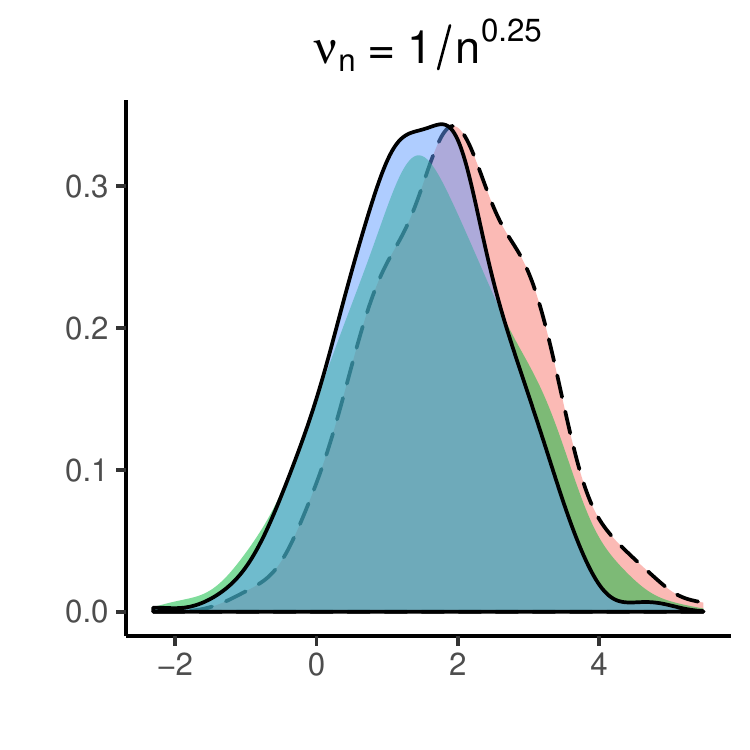} 
  \end{subfigure} 
  \begin{subfigure}[b]{0.3\linewidth}
    \raggedright
    \includegraphics[trim={0.67cm 0.5cm 0.5cm 0},clip,height =\linewidth, width=1.52\linewidth]{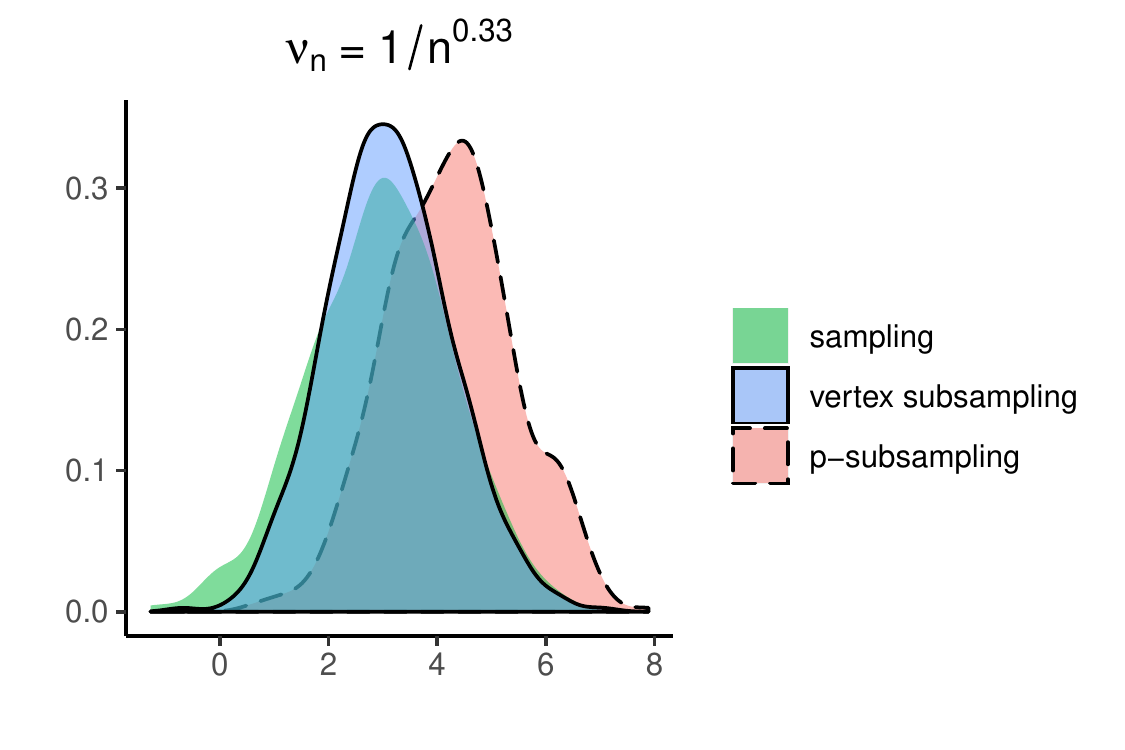} 
      \end{subfigure}  
  \caption{Sampling and subsampling distributions for inference on $\lambda_2(w)$ of the Gaussian Latent Space Model.  The sampling distribution is formed from 500 samples of graphs with $n = 7750$ vertices.  The subsampling distributions are formed from 500 subsamples of a given graph with $n = 7750$ vertices, and an (expected) subsample size of $b_n = 0.3n$.}
  \label{first-eig-plot}
\end{figure}
Our simulation results suggest that $p$-subsampling and vertex subsampling have very similar coverage properties for eigenvalues. Another key takeaway is that the coverage properties of subsampling are heavily affected by sparsity, particularly for eigenvalues with smaller magnitude.  For moderately sparse graphs, confidence intervals for the maximum eigenvalue attain the desired coverage level at moderate values of $n$; however, coverage for the other eigenvalues is more sensitive to sparsity.

We present the distributions of the subsampled $\lambda_2$ for a network with around 8000 nodes generated from a GLSM (Section~\ref{sec:glsm}), with average degree decaying from the left to the rightmost panel in Figure~\ref{first-eig-plot}. The behavior of $\lambda_1$ (see the table in Section~\ref{simulations-tables} of the Supplement) shows less degradation as one increases sparsity. We include similar plots of $\lambda_1$ and $\lambda_2$ for the SBM (Section~\ref{sec:sbm}) in Figure~\ref{second-eig-plot} and~\ref{third-eig-plot}.

While our theoretical results require $b_n = o(n)$, by and large, subsampling performs best when $b_n = 0.3n$, which is the largest subsample size we consider.  This is most likely due to the bias of our estimator, which is most visible in Figure \ref{third-eig-plot}.  As Figures \ref{first-eig-plot} and \ref{second-eig-plot} indicate, the bias appears to be less problematic for eigenvalues that are well-separated from the bulk.  In fact, for the maximal eigenvalue of the SBM and GLSM, our procedure exhibits strong coverage properties even at $n=1000$ for the challenging regime of $\nu_n = n^{-0.33}$.  For GLSM, $\lambda_2$ is also well-behaved in this sparse regime.
\begin{figure}[h] 
	\hspace{-1.5em}
\raggedright
  \begin{subfigure}[b]{0.3\linewidth}
    \raggedright
    \includegraphics[trim={0.55cm 0.5cm 0.5cm 0},clip,width=0.92\linewidth]{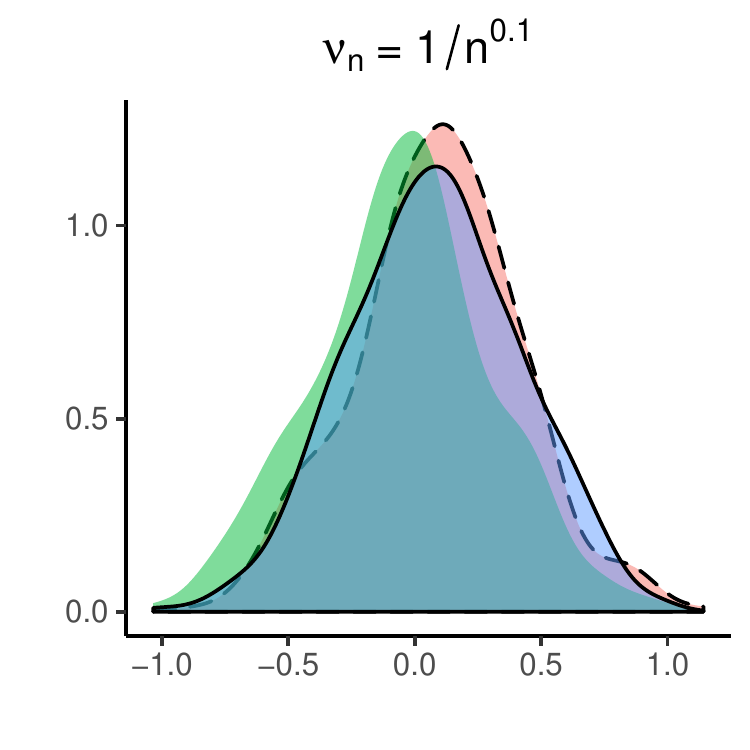} 
  \end{subfigure}
  \begin{subfigure}[b]{0.3\linewidth}
    \raggedright
    \includegraphics[trim={0.55cm 0.5cm 0.5cm 0},clip,width=0.92\linewidth]{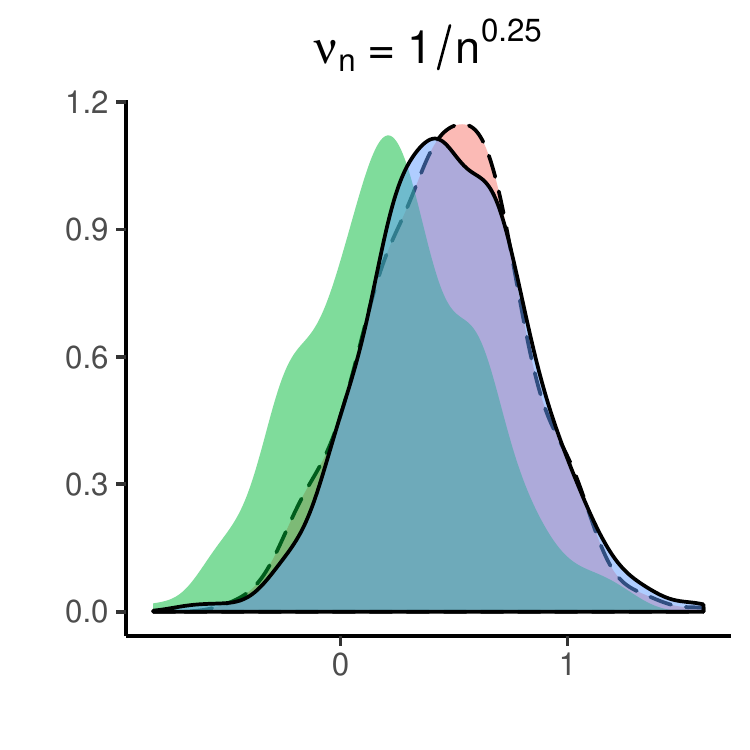} 
  \end{subfigure} 
  \begin{subfigure}[b]{0.3\linewidth}
    \raggedright
    \includegraphics[trim={0.67cm 0.5cm 0.5cm 0},clip,height =\linewidth, width=1.52\linewidth]{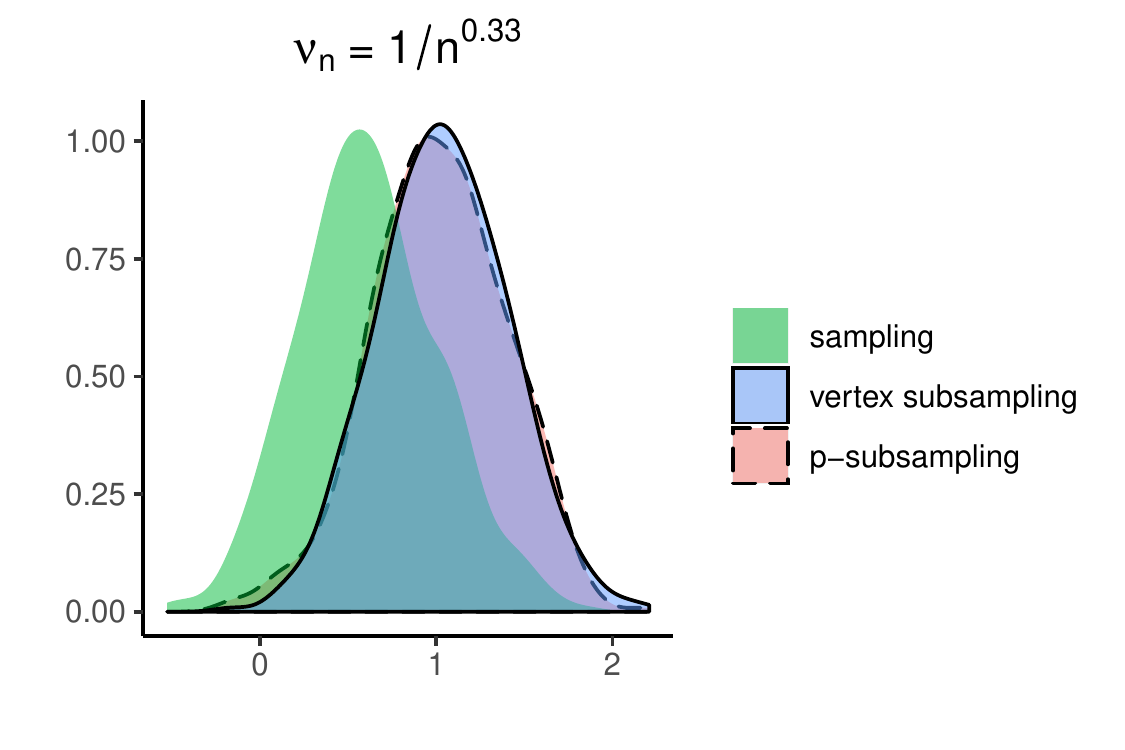} 
      \end{subfigure}  
  \caption{Sampling and subsampling distributions for inference on $\lambda_1(w)$ of the Stochastic Block Model. Simulation settings are identical to those described in Figure \ref{first-eig-plot}. }
  \label{second-eig-plot}
\end{figure}

 \begin{figure}[h] 
 		\hspace{-1.5em}
\raggedright
  \begin{subfigure}[b]{0.3\linewidth}
    \raggedright
    \includegraphics[trim={0.55cm 0.5cm 0.5cm 0},clip,width=0.92\linewidth]{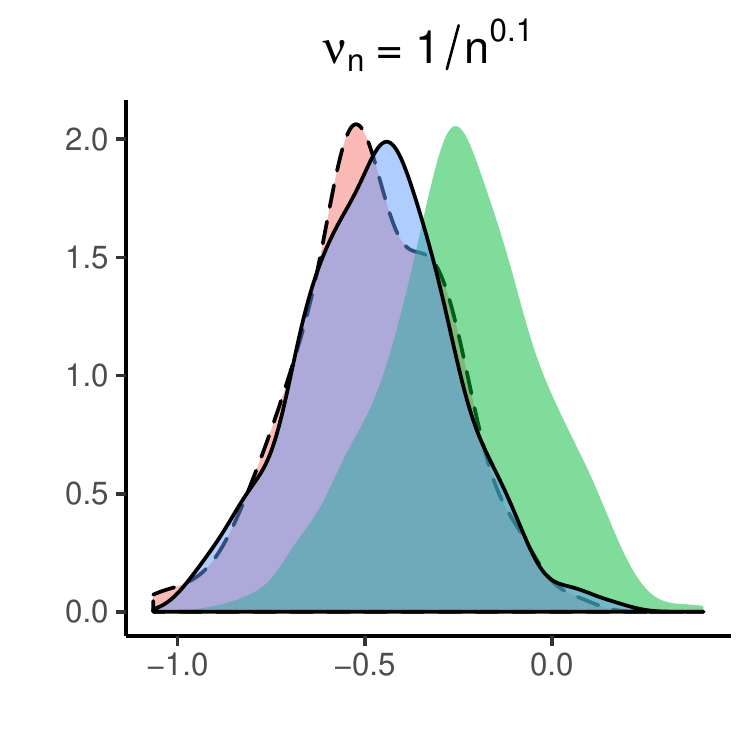} 
  \end{subfigure}
  \begin{subfigure}[b]{0.3\linewidth}
    \raggedright
    \includegraphics[trim={0.55cm 0.5cm 0.5cm 0},clip,width=0.92\linewidth]{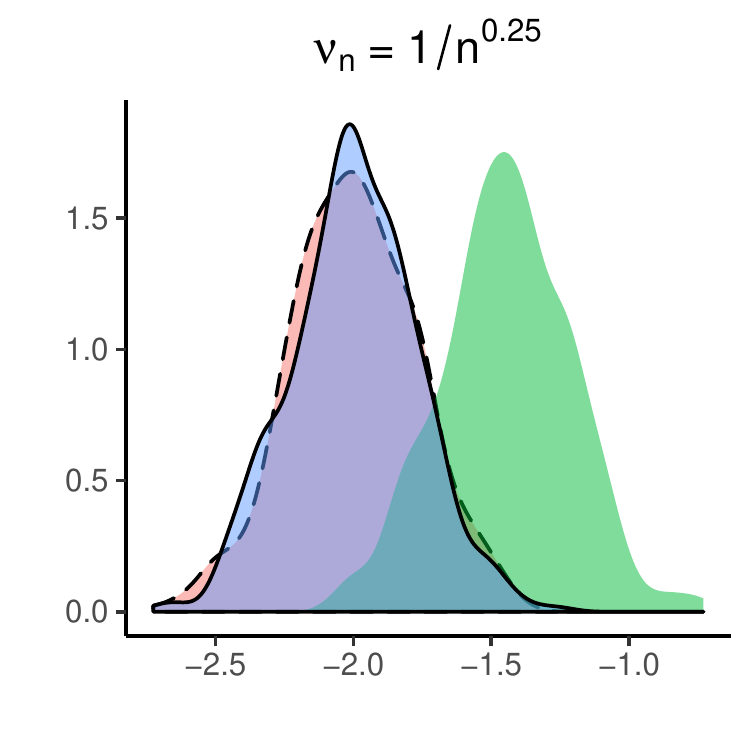} 
  \end{subfigure} 
  \begin{subfigure}[b]{0.3\linewidth}
    \raggedright
    \includegraphics[trim={0.67cm 0.5cm 0.5cm 0},clip,height =\linewidth, width=1.52\linewidth]{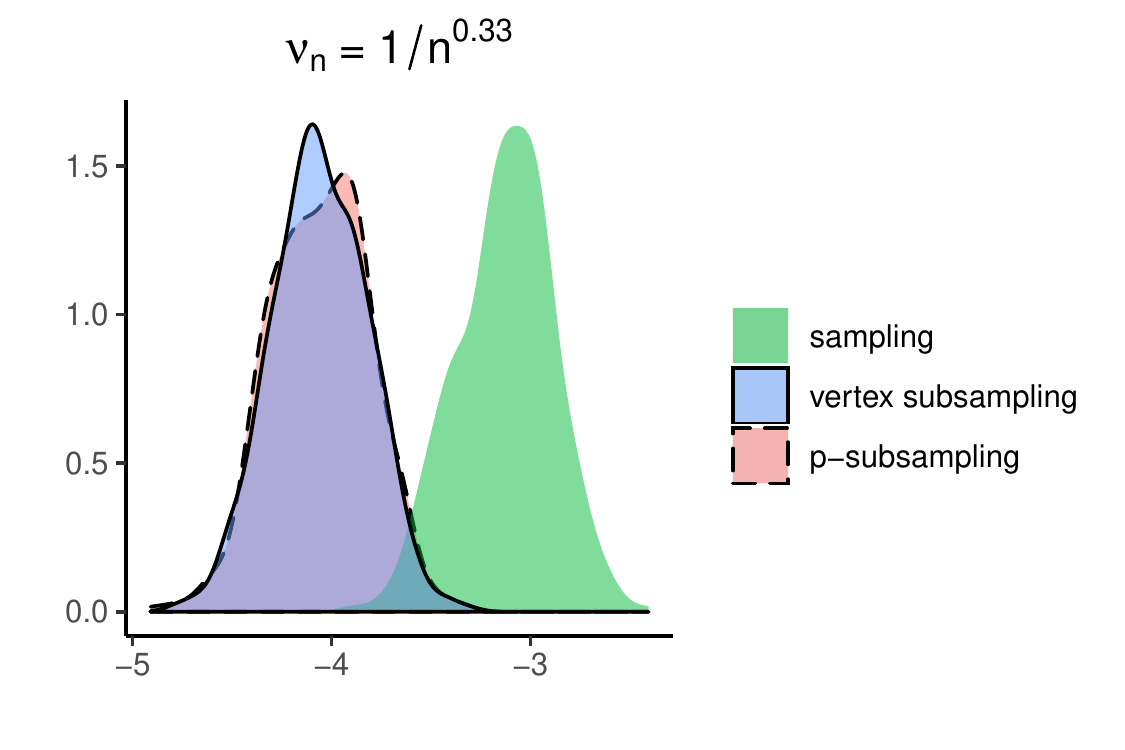} 
      \end{subfigure}  
  \caption{Sampling and subsampling distributions for inference on $\lambda_2(w)$ of the Stochastic Block Model.  Simulation settings are identical to those described in Figure \ref{first-eig-plot}.}
  \label{third-eig-plot}
\end{figure}

For $\lambda_2$  of the SBM (which is negative), coverage is close to the stipulated value only when $\nu_n = n^{-0.1}$ and $b_n = 0.3n$. For the GLSM, coverage for $\lambda_3$ is much worse for sparse graphs even for large $n$.  For tables summarizing our simulation results, see Section~\ref{simulations-tables} of the Supplement.

\section{Real Data Example: Facebook Networks}
\label{real-data}
Using Facebook networks from 100 universities in 2005~\citep{traud-social-networks-facebook}, we perform two sample tests to see whether two social networks are generated by the same sparse graphon.  Facebook networks are well-suited for our method, as they are typically large and not too sparse. They are also undirected, ensuring real-valued spectra.

We will consider a two-sample test based on subsampled eigenvalues for comparing social networks of different universities.  With the same Facebook dataset, \citet{Bhattacharyya-subsample-count-features} consider two-sample testing using subsampled count functionals; here, we examine the relative merits of using subsampled eigenvalue statistics for comparing two networks.  

Before proceeding, we would like to mention some of the strengths and weaknesses of the eigenvalue approach compared to the subgraph approach that are not directly related to the power of the test.  First, note that the entire spectrum of a symmetric matrix can be computed in $O(n^3)$ time and calculating the top few eigenvalues can be performed even faster.  For count functionals, a brute-force subgraph search with $p= |\mathcal{V}(R)|$ has complexity $O(n^p)$.  However, as noted in \citet{Bickel-Chen-Levina-method-of-moments}, for sparser graphs, certain functionals such as kl-wheels can be counted faster.  From a computational perspective, eigenvalue statistics generally may be preferable for network comparisons.  On the other hand, eigenvalue statistics require stronger conditions on the sparsity level; consequently, count functionals may be safer to use for sparser graphs.   

With eigenvalue statistics, it is also not clear a priori which statistics will be most effective at distinguishing two networks.  To limit false discoveries, we consider a sample splitting procedure in which one subnetwork is used for formulating hypotheses and the other is used for testing the hypotheses.  The most natural procedure under the model (Eq~\ref{bernoulli-model}) is node-splitting, in which nodes are randomly split into two disjoint sets and the corresponding induced subgraphs are used.  Under our model, the induced subgraphs are independent and thus, inferences based on node-splitting are valid. 

In this section, we compare the Facebook networks of University of Pennsylvania (Penn) with Columbia and Yale with Princeton. 
\begin{figure}[htb]
\centering 
  \begin{subfigure}[b]{0.46\linewidth}
    \centering
    \includegraphics[width=0.95\linewidth]{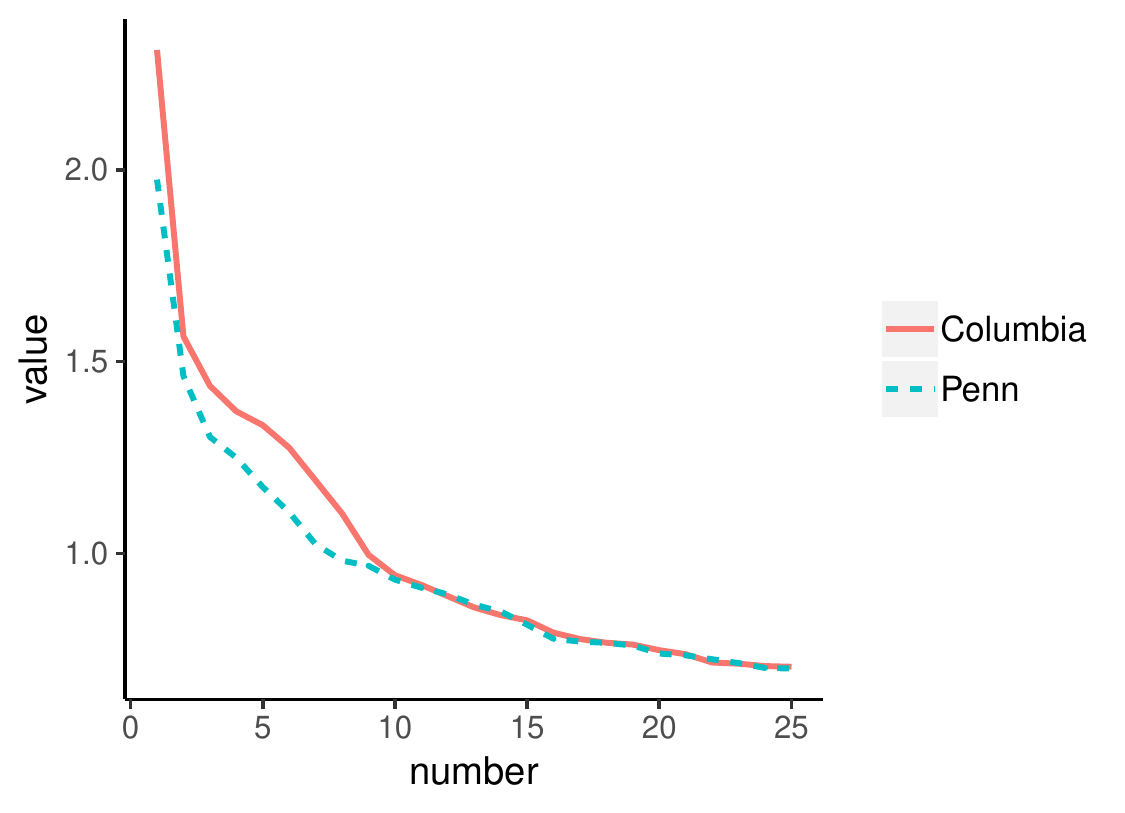} 
    \caption{\label{fig:colpenn} Columbia vs Penn}
  \end{subfigure} 
  \begin{subfigure}[b]{0.46\linewidth}
    \centering
    \includegraphics[width=0.95\linewidth]{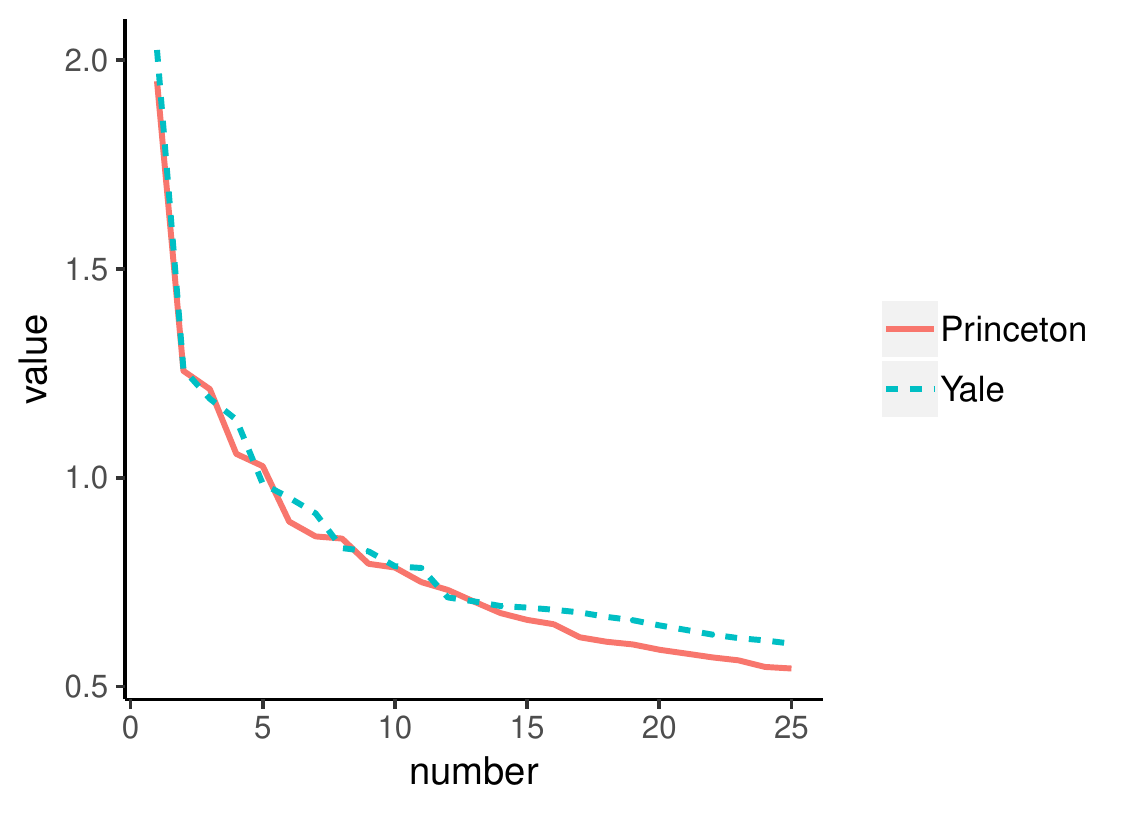} 
    \caption{\label{fig:princetonyale} Princeton vs Yale}
  \end{subfigure} 
  \caption{A comparison of the top 25 normalized eigenvalues calculated on the subnetwork used for exploration.  We see that the behavior of the top few eigenvalues and the rate of decay of the eigenvalues are a few features that vary across collegiate networks.  While differences in the spectrum for the Columbia vs Penn (Fig~\ref{fig:colpenn}) comparison are visible to the eye, it appears that the Princeton and Yale (Fig~\ref{fig:princetonyale}) spectra are very similar.} 
\end{figure} 

\subsection{Data Analysis Results}
For each comparison, we find $k$ ($k\leq 5$) such that $\lambda_k$ has the largest difference between the two exploratory subnetworks.  Then, on the held-out subnetwork, we use subsampling to perform a two-sample test. We mainly consider principal eigenvalues since inference becomes harder for eigenvalues closer to the bulk.  For each network, a $97.5\%$ confidence interval is constructed for the test statistic. If the confidence intervals are disjoint, we may reject the null hypothesis that the graphs were generated by the same graphon with a significance level of $\alpha = 0.05$. The confidence intervals are constructed via vertex subsampling, with a subsample size chosen to be about $33 \%$ of the test set.  While this may seem to be a large proportion, it appears that smaller subsample sizes result in the estimated parameters being systematically outside of the confidence interval, suggesting substantial bias.  The maximum eigenvalue, which is typically well-behaved, is an exception.  

Our results are presented in Table~\ref{table:collegesmain}.  We reject the null hypothesis for Columbia vs Penn, but fail to reject for Princeton vs Yale.  The eigenvalue plots (Fig~\ref{fig:princetonyale}) suggested that it would be difficult to distinguish between the spectra of the Yale and Princeton networks, so it is not surprising that we were not able to reject the null hypothesis here. 

We would like to emphasize that the validity of subsampling hinges on a weak convergence result, which requires a non-trivial sparsity condition.  Since our notion of sparsity pertains to a sequence of graphs and we only have one observation for each school, it is difficult to say whether this condition is satisfied.  However, we would like to note that one would expect the confidence intervals to be very wide if the graphs were too sparse, which does not appear to be the case. A more pressing issue is the bias; we leave the possibility of developing explicit bias corrections to future work.

\begin{table}[H]
\renewcommand*{\arraystretch}{1.5}
\centering
\begin{tabular}{|l|l|l|l|l|l|l}
\cline{1-6}
  Schools   &   n              & Test statistic          & Value       &   $97.5\%$ CI        & Decision       &  \\ \cline{1-6}                                                                
  Columbia  & $5885$       & $\lambda_1/n \rho_n$  &   2.522     &  $(2.124,2.769)$     & Reject $\mathrm{H}_0$         &  \\
  Penn      & $7458$       &                       &   2.033     &  $(1.853, 2.100)$    &                      &  \\ \cline{1-6}
  Princeton & $3298$      & $\lambda_4/n \rho_n$  &   1.057     &  $(0.929, 1.066)$    & Fail to reject $\mathrm{H}_0$ &  \\
  Yale      & $4289$       &                       &   1.031     &  $(0.906, 1.030)$    &                      &  \\ \cline{1-6}
\end{tabular}
\smallskip
\caption{\label{table:collegesmain}Two-sample test results} 
\end{table} 
%

\section{Discussion}
In this paper, we establish the validity of two subsampling schemes for network data generated by the sparse graphon model.  Our result demonstrates that nonparametric inference for networks is possible under conditions similar to those for other data generating processes.  We also derive a weak convergence result for the nonzero eigenvalues of a wide range of sparse graphon models.  While this result is of sufficient interest in itself, it also establishes conditions under which subsampling may be used to yield asymptotically valid inferences for eigenvalues under our model.  

While our result applies to a wide range of graphon models, the imposed conditions on $\{\rho_n \}_{n \in \mathbb{N}}$ suggest that one should only consider subsampling for eigenvalues when the graph is not too sparse.  Our simulations also suggest that, caution must be exercised when making inferences involving eigenvalues closer to the bulk of the spectrum since in these cases, the noise term in (\ref{eigenvalue_decomposition}) appears to decay more slowly.  

We believe that there are several interesting directions for future work.  It remains to be seen whether subsampling validity can be established for certain functionals of graphexes. Our work also leaves open the possibility of more computationally tractable approaches to inference for count functionals of sparse graphons.  If a weak convergence result can be shown for approximate subgraph counting methods, then subsampling validity may be established using the results in this paper.

Finally, it would be interesting to investigate the relationship between the maximum eigenvalue and the average degree for graphon models in greater detail.  For general graphs, it is well-known that the maximum eigenvalue is bounded between the average degree and maximum degree.  For sparse graphons, our simulations in Section~\ref{appendix-max-eigenvalue} of the Supplement suggest that the maximum eigenvalue and the average degree are also highly correlated under general conditions. It would be of interest to derive conditions under which an estimator of the maximum eigenvalue of the graphon operator with an estimated sparsity parameter outperforms an oracle estimator that uses the true sparsity parameter.


  

\bibliographystyle{apalike}
 \bibliography{paper.bib} 

\appendix
\addcontentsline{toc}{section}{Appendices}
\renewcommand{\thesubsection}{\Alph{section}\arabic{subsection}}
\section{Proofs for Section 2}
\subsection{Proof of Theorem \ref{subsampling-fixed-b}}
\label{section:appendix-vertex-subsampling-proof}

Define the following quantity:
\begin{align}
U_{n,b}(t,P_n) =  \frac{1}{N}\sum_{i=1}^{N} \mathbbm{1}\left( \tau_b[\hat{\theta}_{n,b,i}(G_n) - \theta(G_\infty)] \leq t \right)
\end{align}
It suffices to show that $U_{n,b}(t,P_n) \xrightarrow{P} J(t,P_\infty)$ at the continuity points of $J(\cdot, P_\infty)$.
 Let $f(j_1, j_2, \ldots, j_b) =\mathbbm{1}\left( \tau_b[\hat{\theta}_{n,b,i}(G_n) - \theta(G_\infty)] \leq t \right)$ denote the kernel corresponding to subgraph induced by the nodes $(j_1, j_2, \ldots j_b )$. Furthermore, let $k = \lceil \frac{n}{b} \rceil$ and define:
  \begin{align}
  \label{independent-sum-representation}
  F(j_1, \ldots, j_n) = \frac{f(j_1, \ldots, j_b) + f(j_{b+1}, \ldots, j_{2b}) + \ldots + f(j_{(k-1)b +1}, \ldots, j_{kb})}{k}
  \end{align}

  Letting $\sum_{\sigma}$ denote summation over all $n!$ permutations $(j_1, \ldots, j_n)$ of $\{1 ,\ldots, n \}$, we may use the Hoeffding representation to express $U_{n,b}$ as:
  \begin{align}
  \label{hoeffding-representation}
  U_{n,b} = \frac{1}{n!} \sum_\sigma F(\sigma(j_1), \ldots ,\sigma(j_n))
  \end{align}

  Now, by independence of induced subgraphs with disjoint node sets, we have that $F(\sigma(j_1), \ldots, \sigma(j_n))$ is a sum of independent random variables.  Therefore, using analogous reasoning to the proof for Hoeffding's inequality for U-Statistics (c.f. \citet{Serfling-Approximation-Theorems}, Theorem A p.201):
  \begin{align*}
  P\left( \bigl| U_{n,b}(t,P_n) -\mathbb{E}[U_{n,b}(t,P_n)] \big|   >  \epsilon \right) \leq 2 \exp\left(\frac{-C n\epsilon^2}{b} \right)
 \end{align*}
Since $\mathbb{E}[U_{n,b}(t,P_n)] = J_{n,b}(t,P_n)$ and $J_{n,b}(t,P_n) \rightarrow J(t,P_\infty)$, part $\bf{i.}$ follows. Parts $\bf{ii}. - \ \bf{iii}$ follow based on arguments analogous to \citet{Politis-Romano-Wolf-subsampling}. \qed

\subsection{Proof of Theorem \ref{p-subsampling-theorem}}
\label{section:appendix-p-subsampling-proof}
We will again show only $\mathbf{i}$.  Define the following functional:
\begin{align}
V_{n,B}(t,P_n) =  \frac{1}{M}\sum_{i=1}^{M} \mathbbm{1}\left( \tau_{B_i}[\hat{\theta}_{n,B_i}\bigl(\mathrm{Smpl}_{p,i}(G_n)\bigr) - \theta(G_\infty)] \leq t \right)
\end{align}

Note that the normalization parameter $\tau_{B_i}$ is now random.  We will show that it still suffices to control $V_{n,B}(t,P_n)$. Observe that:  
\begin{align*}
\label{subsampling-estimator-decomposition}
L_{n,B}^\prime(t,P_n) =  \frac{1}{M}\sum_{i=1}^{M} \mathbbm{1}\left( \tau_{B_i}\left[\hat{\theta}_{n,B_i}\bigl(\mathrm{Smpl}_{p,i}(G_n)\bigr) - \theta(G_\infty)\right] - \tau_{B_i}[\hat{\theta}_n(G_n)-\theta(G_\infty)] \leq t \right)
\end{align*}
For some $\epsilon > 0$, define the random variable:
\begin{align}
\chi_{n,i} = \mathbbm{1}\left( \tau_{B_i} \left|\hat{\theta}_n(G_n) - \theta(G_\infty)\right| > \epsilon \right)
\end{align}
 For any $\delta > 0$, let $A_n$ denote the corresponding event:
\begin{align}
A_n = \left \{ \frac{1}{M}\sum_{i=1}^M \chi_{n,i}> \delta    \right\}
\end{align} 
For our purposes, we will see shortly that $\epsilon$ must be chosen sufficiently small for each $\delta >0$.  We will show that $P(A_n) \rightarrow 0$ for any $\epsilon >0 $ and $\delta >0$.

Now, since $0 \leq B_i \leq \sum_{j=1}^n X_{ij}^{(n)}$ and $\frac{1}{n}\sum_{j=1}^n X_{ij}^{(n)} \xrightarrow{P} 0$, it follows that $B_i/n \xrightarrow{P} 0$.  Therefore, by continuous mapping theorem, $\tau_{B_i}/\tau_n \xrightarrow{P} 0$.  By Slutsky's theorem, we have that $\mathbb{E}[\chi_{n,i}] \rightarrow 0$ and thus, $P(A_n) \rightarrow 0$ for any $\delta >0$ and $\epsilon >0$. We will now bound $L_{n,B}^\prime(t,P_n)$ on $A_n^c$. 

Now fix a summand of $L_{n,B}^\prime(t,P_n)$.  If $\tau_{B_i} \left|\hat{\theta}_n(G_n) - \theta(G_\infty)\right| \leq \epsilon$, then:
\begin{align*}
 & \ \ \ \ \mathbbm{1}\left( \tau_{B_i}[\hat{\theta}_{n,B_i}\bigl(\mathrm{Smpl}_{p,i}(G_n)\bigr) - \hat{\theta}_n(G_n)] \leq t \right)\\
 & \leq \mathbbm{1}\left( \tau_{B_i}[\hat{\theta}_{n,B_i}\bigl(\mathrm{Smpl}_{p,i}(G_n)\bigr) - \theta(G_\infty)] - \epsilon \leq t \right) 
\end{align*}
Otherwise, since the summand is at most one, we may upper bound it by $\chi_{n,i}$ when the above condition is not satisfied. Therefore, the summand is upper-bounded by the sum of these two terms.  We may lower bound the summand in an analogous manner.  Therefore, it follows that:  
\begin{align*}
V_{n,B}(t-\epsilon,P_n) - \frac{1}{M}\sum_{i=1}^M \chi_{n,i}  \leq  L_{n,B}^\prime(t,P_n) \leq V_{n,B}(t+\epsilon,P_n) + \frac{1}{M}\sum_{i=1}^M \chi_{n,i}
\end{align*}
 For any $\delta > 0$, $\epsilon$ can be chosen so that $|J(t \pm \epsilon,P_\infty) - J(t,P_\infty)| < \delta$.  Furthermore, $\epsilon$ can be chosen so that $t \pm \epsilon$ are continuity points of $J(\cdot, P_\infty)$.  Therefore, if $V_{n,B}(t,P_n) \xrightarrow{P} J(t,P_\infty)$, with probability tending to $1$:
\begin{align*}
\label{upper-lower-bound-p-sampling}
\begin{split}
& \ \ J(t-\epsilon,P_\infty) - 2\delta \leq L_{n,B}^\prime(t,P_n) \leq J(t+\epsilon,P_\infty) + 2\delta
\\ & \implies J(t,P_\infty) - 3 \delta \leq L_{n,B}^\prime(t,P_n) \leq J(t,P_\infty) + 3 \delta
\end{split}
\end{align*}

\noindent Hence, it suffices to show $V_{n,B}(t,P_n) \xrightarrow{P} J(t,P_\infty)$ to establish $\mathbf{i}$.  We will now approximate $V_{n,B}(t, P_n)$ with a suitable linear combination of U-statistics.  Let $Y_i^{(n)} = \sum_{j=1}^n X_{ij} \sim \mathrm{Binomial}(n,p_n)$ and $q_j = P(Y_1^{(n)} = j)$.  Furthermore, let $N_j = {n \choose j}$. Let $\{G_{n,j,i}^* \}_{ 1 \leq i \leq N_j}$ represent subgraphs formed by taking induced subgraphs with $j$ vertices, arranged in any order, with isolated vertices removed.  Further let $j^* = |\mathcal{V}(G_{n,j,i}^*)|$; note that $\hat{\theta}_{n,j^*,i}(\cdot)$ is itself now a random function through $j^*$. Observe that:
\begin{align*}
\mathbb{E}[V_{n,B}(t,P_n) \ | \ G_n ] = \sum_{j=0}^n q_j \left( \frac{1}{N_j} \sum_{i=1}^{N_j} \mathbbm{1}\left( \tau_{j^*}[\hat{\theta}_{n,j^*,i}( G_{n,j,i}^*) - \theta(G_\infty)] \leq t \right) \right)
\end{align*}
Now conditional on $G_n$, $V_{n,b}(t,P_n)$ is a sum of independent random variables bounded by 1.  By Hoeffding's inequality, since $M \rightarrow \infty$:
\begin{align*}
P\left( \left|V_{n,b}(t,P_n) - \mathbb{E}[V_{n,b}(t,P_n) \ | \ G_n] \right| > \epsilon \ | \ G_n  \right) \leq 2 \exp( -2 M \epsilon^2 ) \rightarrow 0
\end{align*}
Therefore,  $ V_{n,B}(t,P_n) -\mathbb{E}[V_{n,B}(t,P_n) \ | \ G_n ] \xrightarrow{P} 0$ unconditionally.

Now, we will further approximate $\mathbb{E}[V_{n,B}(t,P_n) \ | \ G_n ]$ with another linear combination of U-statistics that excludes p-subsample sizes that are unlikely to be chosen.  This will allow us to show that $\mathbb{E}[V_{n,B}(t,P_n) \ | \ G_n ]$ concentrates around its expectation under the assumed conditions.

Recall that $C_n = \{ l_n, l_n +1, \ldots, u_n \}$, where $c_n = 3 \sqrt{n p_n \log n}$,  $l_n = \lfloor np_n - c_n \rfloor$ and $u_n = \lceil np_n + c_n \rceil$. By the Chernoff bound in (\ref{chernoff-binomial-bound}), we have that $P(Y_1^{(n)} \not\in C_n) \rightarrow 0$. Now define:
\begin{align}
\overline{U}_{n}(t,P_n) :=  \sum_{j \in C_n } q_j \left( \frac{1}{N_j} \sum_{i=1}^{N_j} \mathbbm{1}\left( \tau_{j^*}[\hat{\theta}_{n,j^*,i}( G_{n,j,i}^*) - \theta(G_\infty)] \leq t \right) \right) 
\end{align}    

Since each U-statistic is bounded by $1$ and $ \sum_{j \not \in C_n} q_j \rightarrow 0$, we have that $\overline{U}_{n}(t,P_n) - \mathbb{E}[V_{n,B}(t,P_n) \ | \ G_n ] \rightarrow 0$. Now let:
\begin{align}
U_{n,j}(t,P_n) := \frac{1}{N_j} \sum_{i=1}^{N_j} \mathbbm{1}\left( \tau_{j^*}[\hat{\theta}_{n,j^*,i}( G_{n,j,i}^*) - \theta(G_\infty)] \leq t \right)
\end{align}
Let $\tilde{q}_j=q_j/\sum_{k\in C_n}q_k$ and $\tilde{a}=a/\sum_{k\in C_n}q_k$. 
Applying Markov's inequality on the moment generating function, for any $s,a >0$, and some constant $C>0$, we have that:
\begin{align*}
P\left( \overline{U}_{n} -  \mathbb{E}[\overline{U}_{n}]  > a \right) 
& \leq \exp\left(-s\tilde{a}\right) \cdot \mathbb{E}\left[ \exp\left(  s\sum_{j \in C_n} \tilde{q}_j (U_{n,j}-\mathbb{E}[U_{n,j}]) \right)  \right] \\
&\stackrel{(i)}{\leq} \exp\left(-sa\right) \sum_{j\in C_n}\tilde{q}_j \mathbb{E}\left[ \exp\left(  s (U_{n,j}-\mathbb{E}[U_{n,j}]) \right)  \right] \\ &\stackrel{(ii)}{\leq}  \exp\left(\frac{-Cna^2}{u_n} \right)
\stackrel{(iii)}{=} o(1)
\end{align*}
Step $(i)$ follows from Jensen's inequality since $\sum_{j\in C_n}\tilde{q}_j=1$. Step $(ii)$ follows from Hoeffding's inequality for each $U_{n,j}$. Step $(iii)$ holds since $u_n = o(n)$. The argument for the lower tail is analogous.

%
Now, the result will follow if we can show that $\mathbb{E}[\overline{U}_{n}(t,P_n)] - J(t, P_\infty) \rightarrow 0$. Let $\mathcal{E}_j$ denote the event that none of the nodes of a random graph with $j$ vertices are isolated.  Since $w(\xi_i,\xi_j) \geq c \ a.s.$,  standard arguments for Erd\H{o}s-R\`{e}nyi graphs can be used to conclude that the probability that the graph contains an isolated node, i.e. $P(\mathcal{E}_j^c)$, goes to $0$ so long as $\rho_n = \omega(\log l_n / l_n)$.  

We will now derive an upper bound.  For each $j$, by law of total probability partitioning on $\{\mathcal{E}_j,\mathcal{E}_j^c \}$, observe that:
\begin{align*}
& \ \ \ \sum_{j \in C_n} q_j P( \tau_{j^*}[\hat{\theta}_{n,j^*,i}( G_{n,j,i}^*) - \theta(G_\infty)]  \leq t) - J (t, P_\infty)  
\\ & \leq \sum_{j \in C_n} q_j \left[J_{n,j}(t,P_n) - J(t,P_\infty) \right]  +  \sum_{j \in C_n} q_j P(\mathcal{E}_j^c) - \bigl(1- \sum_{j \in C_n} q_j\bigr) \cdot J(t,P_\infty)  
\end{align*}
Since $\sup_{j \in C_n}\left| J_{n,j}(t,P_n) - J(t,P_\infty) \right| \rightarrow 0$ and $\sum_{j \in C_n} q_j < 1$, the first term may be upper bounded by $\epsilon/2$ for any $\epsilon >0$.  By similar reasoning, the second term is bounded by $\epsilon/2$.  The last term may be upper bounded by zero. For the lower bound, observe that, by inclusion-exclusion principle:
\begin{align*}
& \ \ \ \sum_{j \in C_n} q_j P( \tau_{B_i}[\hat{\theta}_{n,j^*,i}( G_{n,j^*,i}^*) - \theta(G_\infty)]  \leq t) - J (t, P_\infty)  
\\ & \geq \sum_{j \in C_n} q_j \left[ J_{n,j}(t,P_n) - J(t,P_\infty) \right]+  \sum_{j \in C_n} q_j [P(\mathcal{E}_j) -1]- \bigl(1- \sum_{j \in C_n} q_j\bigr) \cdot J(t,P_\infty)    
\end{align*} 
Now, the lower bound may be shown by an analogous argument to the upper bound, and the result follows. \qed
 \section{Proofs for Section 3}
\subsection{Proof of Theorem \ref{eigenvalue-weak-convergence}}
\label{section:appendix-eigenvalue-limit-proof}
Consider the following decomposition:
\begin{align}
\begin{split}
\label{eigenvalue_decomposition}
Z_{n,r} &= \lambda_{r}(A^{(n)}/\sqrt{n}\rho_n) - \lambda_{r}(P^{(n)}/\sqrt{n}\rho_n) 
\\ &  \ \ \ + \lambda_{r}(P^{(n)}/\sqrt{n}\rho_n)- \lambda_{r}(W^{(n)})/\sqrt{n})
\\ & \ \ \ + \sqrt{n}[ \lambda_{r}(W^{(n)}/n)- \lambda_{r}(w)]
\end{split}
\end{align}
where $P_{ij}^{(n)} = \mathbb{E}[A_{ij}^{(n)} \ | \ \boldsymbol{\xi}_n] = \rho_n w(\xi_i, \xi_j) \wedge 1$ for $i \neq j$ and $P_{ii}^{(n)} = 0$.  Similarly, $W_{ij}^{(n)} = w(\xi_i, \xi_j)$ and for $i \neq j$ and $W_{ii}^{(n)} = 0$. 
By Slutsky's Theorem, it suffices to show that one of the terms of $Z_n$ weakly converges while the other terms converge in probability to a point mass.  The first term is a noise term that arises from randomly perturbing the mean matrix.  The second term is also a nuisance term, which arises from truncating $W^{(n)}/\sqrt{n}$.  The third term, which involves an approximation of the eigenvalues of the graphon operator with the eigenvalues of $W^{(n)}/\sqrt{n}$, will determine the limiting distribution.        

 For the third term of (\ref{eigenvalue_decomposition}), under the assumptions in the theorem, we may apply Corollary 5.8 of \citet{koltchinksii-gine-kernel-operator} to establish the existence of a limiting distribution. For A, note that $\norm{w(u,v)}_\infty < \infty \ a.s.$ together with $\lambda_r \neq 0$ implies (\ref{eigenfunction-condition}); see Proposition 7.17 of \citet{Lovacz-Graph-Limits}. 

Now, for the second term of (\ref{eigenvalue_decomposition}), observe that case A is trivial since if $\rho_n \rightarrow 0$, for $n$ large enough, we have that $W_{ij}^{(n)} = P_{ij}^{(n)}/\rho_n \ a.s.$  We will now deal with cases B and C. By Weyl's inequality, we have that:
\begin{align*}
& \ \ \ P\left(\max_{1 \leq r \leq k} \left|\lambda_{r}(P^{(n)}/\sqrt{n}\rho_n)- \lambda_{r}(W^{(n)})/\sqrt{n}) \right| > \epsilon \right)
\\ &\leq  P\left( \norm{P^{(n)}/\sqrt{n}\rho_n- W^{(n)}/\sqrt{n}}_{op} > \epsilon \right)
\end{align*}
Now, let $\norm{\cdot}_{F}$ denote the Frobenius norm of matrix. It is well-known that $\norm{A}_{op} \leq \norm{A}_{F}$.  Now, we have that:
\begin{align}
\label{bounding-censored-eigenvalue-perturbation}
\begin{split}
& \ \ \ P\left( \norm{P^{(n)}/\sqrt{n}\rho_n- W^{(n)}/\sqrt{n}}_{op} > \epsilon \right)
\\ &\leq P\biggl(\sqrt{\sum_{i=1}^n \sum_{j=1}^n \frac{w^2(\xi_i, \xi_j)}{n} \mathbbm{1}(\rho_n w(\xi_i, \xi_j) > 1) } > \epsilon \biggr)
\\ &\leq n\epsilon^{-2} \ \mathbb{E}\left[ w^2(\xi_i, \xi_j)  \mathbbm{1}\left( w(\xi_i, \xi_j) > \rho_n^{-1} \right) \right]
\end{split}
\end{align}
Now for case C, for $s > 2$ consider the bound:
\begin{align*}
n\epsilon^{-2}\ \mathbb{E}\left[ w^2(\xi_i, \xi_j)  \mathbbm{1}\left( w(\xi_i, \xi_j) > \rho_n^{-1} \right) \right] &\leq n\epsilon^{-2} \ \mathbb{E}\left[ w^2(\xi_i, \xi_j) \left(\frac{w(\xi_i, \xi_j)}{\rho_n^{-1}}\right)^{s-2} \right]
\\ &=\epsilon^{-2} \mathbb{E}\left[ w^s(\xi_i, \xi_j)\right] n \rho_n^{s-2} 
\end{align*}
For case B, we will exploit the fact that sub-Weibull distributions have exponential tails. If $w(\xi_i, \xi_j) - \mathbb{E}[w(\xi_i, \xi_j)]$ is sub-Weibull$(K_1, \lambda)$,then there exists $K_2$ differing from $K_1$ by at most a constant factor~\citep{chong-lasso-heavy-tailed} such that:
\begin{align}
P\left( \left|w(\xi_i, \xi_j) - \mathbb{E}[w(\xi_i, \xi_j)] \right| > t \right) \leq \exp\left( -(t/K_2)^\lambda \right)
\end{align}
Now for case B, we may bound (\ref{bounding-censored-eigenvalue-perturbation}) using the Cauchy-Schwartz inequality:
\begin{align*}
n \epsilon^{-2} \ \mathbb{E}\left[ w^2(\xi_i, \xi_j)  \mathbbm{1}\left( w(\xi_i, \xi_j) > \rho_n^{-1} \right) \right] &\leq n \epsilon^{-2} \sqrt{\mathbb{E}\left[ w^4(\xi_i, \xi_j) \right] P( w(\xi_i, \xi_j) > \rho_n^{-1}) } 
\end{align*} 
The sub-Weibull condition also ensures existence of moments; therefore, for both cases B and C the second term of (\ref{eigenvalue_decomposition}) converges in probability to 0 under the assumed conditions.

What remains is showing that the first term converges in probability to $0$.  It turns out that Weyl's inequality is too weak to yield the desired concentration.  We will therefore use tools developed by \citet{eldridge-beyond-davis-kahan} for establishing concentration. Now observe that:
\begin{align*}
& \ \ \ \  \ \ \ P\left(\sqrt{n} \bigl|\lambda_{r}(A^{(n)}/n\rho_n) - \lambda_{r}(P^{(n)}/n\rho_n)\bigr| > \epsilon \right)
\\ & \ \ \ \leq \sup_{\omega \in E_\delta} P\left(\sqrt{n} \bigl|\lambda_{r}(A^{(n)}/n\rho_n) - \lambda_{r}(P^{(n)}/n\rho_n)\bigr| > \epsilon \ \biggr\rvert \ \boldsymbol{\xi}_n(\omega) \ \right) + P(E_{\delta}^c)
\end{align*}
where $E_\delta$ satisfies $P(E_\delta) \geq 1 - \delta$.  To establish convergence in probability, it now suffices to show that the first term converges to zero for any $ 0 < \delta < 1$.

 We will break this part of the proof into several steps. \\ \\ 
\noindent\textbf{Step 1: Derive a bound on the spectral norm of} $(A-P)/\rho_n$.

\noindent For notational convenience, let $H = (A-P)/\rho_n$. Conditioning on $\boldsymbol{\xi}_n$, $H$ is a mean $0$ random variable with independent entries.  We also have that:
\begin{align*}
\begin{split}
\label{crude-variance-bound}
\text{Var}\left( H_{ij} \ | \ \boldsymbol{\xi}_n \right) &= \rho_n^{-1} w(\xi_i, \xi_j) \wedge \rho_n^{-2}  - w^2(\xi_i, \xi_j) \wedge \rho_n^{-2}
\end{split}
\end{align*}
We may now derive bounds for the variance under the assumptions A,B,C. We will start with A; for this case, observe that there exists a constant $D <\infty$ such that $\text{Var}( H_{ij}) \leq D\rho_n^{-1}$. For cases B and C, define the event:
\begin{align}
\mathcal{B}_n = \left\{\max_{1 \leq i < j \leq n} w(\xi_i,\xi_j) \leq M_n \right\}
\end{align} 
For both cases, we will choose $M_n$ so that $P(\mathcal{B}_n) \rightarrow 1$.  For case B, by Markov's inequality, the choice $M_n = n^{(2+\delta)/s}$ yields $P(\mathcal{B}_n) \geq 1 - 1/n^{\delta}$ for some $\delta > 0$.  Note that this bound is worse than the trivial bound $\rho_n^{-2}$ when $s<4$ and $\rho_n = \omega(1/\sqrt{n})$.  For case C, the choice $M_n =  (3K_2 \log n)^{1/\lambda}  + \mathbb{E}[w(\xi_i, \xi_j)]$ yields $P(\mathcal{B}_n) \geq 1 - 1/n$.  It then follows that:
\begin{align}
\sup_{\omega \in \mathcal{B}_n} \max_{1 \leq i < j \leq n} \mathrm{Var}\left( H_{ij} \ | \ \boldsymbol{\xi}_n(\omega) \right) \leq  M_n \rho_n^{-1}
\end{align}   

In all three cases, notice that $|H_{ij}| \leq \rho_n^{-1}$. Now, we may derive a on $\norm{H}_{op}$ using a result from random matrix theory.  In what follows, we say that a sequence of events $\{E_n\}_{n \in \mathbb{N}}$ holds w.h.p. uniformly over some set $\mathcal{B}_n$ if $\inf_{\omega \in \mathcal{B}_n } P( E_n \ | \ \boldsymbol{\xi}_n(\omega)) \rightarrow 1$.  Applying Theorem 1.4 of \citet{Vu-Spectral-Norm}, we have that, conditional on $\boldsymbol{\xi}_n$, there exists some $C < \infty$ such that:
\begin{itemize}
\item Case A: $\norm{H}_{op} \leq C\sqrt{n/\rho_n}$ w.h.p. uniformly on $\Omega$. 
\item Case B: $\norm{H}_{op} \leq C n^{1/2+\delta}/\sqrt{\rho_n}$ w.h.p. uniformly on $\mathcal{B}_n$.  
\item Case C: $\norm{H}_{op} \leq Cn^{(2+s+\delta)/2s}/\sqrt{\rho_n}$ w.h.p. uniformly on $\mathcal{B}_n$.
\end{itemize}    
Uniformity follows from the fact the above result is based on Markov's inequality, with constants depending only on $K$ and $\sigma$. While $K$ and $\sigma$ are changing with $n$, the requirement that $\sigma \geq C^\prime n^{-1/2} K \ln^2 n$ from the above reference holds for $n$ large enough for any $ 0 < C^\prime < \infty$ with $\rho_n = \omega(1/\sqrt{n})$. \\ \\
\noindent \textbf{Step 2: For a fixed unit vector $x \in \mathbb{R}^n$, derive a bound on} $x^THx$

\noindent 
Let $\mathcal{A}_n$ denote the event:
\begin{align}
\mathcal{A}_n = \left\{x^T H x \leq \epsilon_n  \right\}  
\end{align}
for some $\epsilon_n$ to be chosen shortly. By Lemma 2 \citep{eldridge-beyond-davis-kahan}, we have that, for some universal constant $K_1 > 0$:
\begin{align*}
P(x^THx > a \ | \ \boldsymbol{\xi}_n) \leq \exp\left(\frac{-K_1 a^2}{ \max_{i,j}||H_{ij}||_{\psi_2}^2} \right) 
\end{align*}  
where $\norm{\cdot}_{\psi_2}$ is the sub-Gaussian norm of a random variable $X$, given by:
\begin{align}
\norm{X}_{\psi_2} :=  \inf \left\{ t > 0 \ \bigr\rvert \ \mathbb{E}\left[\exp\left(X^2/t^2\right)\right] \leq 2 \right\} 
\end{align}

For cases B and C, $P(\mathcal{A}_n) \geq 1 - \frac{1}{n}$ for the choice $\epsilon_n = C \sqrt{\log n}/\rho_n$ for some $C < \infty$ since $\norm{H_{ij}}_{\psi_2}^2 \leq \frac{1}{4 \rho_n^2}$ by Hoeffding's Lemma.  For case A, we will use a sharper bound for the sub-Gaussian norm, which improves our result by a log factor compared to an analysis with a Hoeffding bound. Following for example \citep{buldygin-sub-gaussian-binary}, we have that:
\begin{align*}
\norm{H_{ij}}_{\psi_2} = \rho_n^{-1}\norm{A_{ij}-P_{ij}}_{\psi_2} = \Theta\left(\frac{1}{ \rho_n \sqrt{|\log \rho_n|}} \right)
\end{align*}
For case A, we have $P(\mathcal{A}_n) \geq 1 - \frac{1}{n}$ for $\epsilon_n = C\rho_n^{-1}\sqrt{\log n / |\log \rho_n|}$ for some $C < \infty$. Observe that $\epsilon_n = o(\sqrt{n})$ for $\rho_n = \omega(1/\sqrt{n})$. \\ \\
\noindent \textbf{Step 3: Use eigenvalue perturbation bounds for concentration}

\noindent Suppose that $T_w$ has $d$ positive eigenvalues; at least one such eigenvalue exists.  We will now establish concentration for $(\lambda_1(A^{(n)}), \ldots , \lambda_d(A^{(n)}))$; the argument for the negative eigenvalues will be analogous.     
We will start with an upper bound.  Let $\mathcal{S}_n$ denote the span generated by the eigenvectors $\{v_1(\omega), \ldots, v_k(\omega) \}$ associated with $\{\lambda_1(P/\rho_n), \ldots,  \lambda_k(P/\rho_n) \}$ and let $h:= \max_{v \in \mathcal{S}_n} v^THv $. By Lemma 5 \citep{eldridge-beyond-davis-kahan}, conditional on each $\boldsymbol{\xi}_n$, $h$ may be bounded by $k$ times the bound on $x^THx$ derived in Step 2. 

Now, Theorem 7 \citep{eldridge-beyond-davis-kahan} yields, for any $1 \leq r \leq p$, if $\lambda_r(P/\rho_n) - \lambda_{d+1} (P/\rho_n) > 2 \norm{H}_{op} + h  $, then: 
\begin{align}
\lambda_r(A/\rho_n) - \lambda_r(P/\rho_n) \leq  h + \frac{\norm{H}_{op}^2}{\lambda_r(P/\rho_n) - \lambda_{d+1}(P/\rho_n) + h - \norm{H}_{op}}  
\end{align} 
We will start by verifying that $\lambda_r(P/\rho_n) - \lambda_{d+1} (P/\rho_n) > 2 \norm{H}_{op} + h$ holds with high probability, conditional on $\boldsymbol{\xi}_n$, uniformly over an appropriate set.   
First, note that since $\mathbb{E}[w^2(u,v)] < \infty$,  Theorem 3.1 of \citet{koltchinksii-gine-kernel-operator} implies that, for all $1 \leq r \leq d$:
\begin{align*}
\begin{split}
\lambda_r(W/n)  \xrightarrow{a.s.} \lambda_r(w), \ \ \ \lambda_{d+1}(W/n)  \xrightarrow{a.s.} 0
\end{split}
\end{align*}
Now, since $n^{1+\delta}\rho_n^{s-2} = o(1)$, the bound from (\ref{bounding-censored-eigenvalue-perturbation}) yields, for some $C < \infty$:
\begin{align*}
P\left(\max_{1 \leq r \leq d+1} \left|\lambda_r(W/n) -  \lambda_r(P/n \rho_n)\right| > \epsilon \right) &\leq \frac{C \ \mathbb{E}[w^{s}(\xi_i,\xi_j)]}{n^{1+\delta}} 
\end{align*}
Now by Borel-Cantelli Lemma, $\lambda_r(W/n) -  \lambda_r(P/n\rho_n) \xrightarrow{a.s.} 0$. By Egorov's theorem (c.f. \citep{Stein-Shakarchi-Analysis} Theorem 4.4 p.33), we may choose a set $\mathcal{C}$ such that $P(\mathcal{C}) \geq 1-\delta/2$ and the convergence of $\lambda_r(P/n \rho_n)$ is uniform over $\mathcal{C}$. Therefore, uniformly on $\mathcal{C}$, $\lambda_r(P/\rho_n)- \lambda_{d+1}(P/\rho_n) = \Theta(n)$ and therefore dominates the other terms.  For case A, let $E_\delta = \mathcal{C}$. For cases B and C, let $E_\delta = \mathcal{B}_n \cap \mathcal{C}$; for $n$ large enough, $P(E_\delta) > 1- \delta$.
Now, plugging in bounds from Step 1 and Step 2, we may argue that for some $K <\infty$:
\begin{align*}
\inf_{\omega \in E_\delta} P\left( \lambda_r(A / \rho_n) -  \lambda_r(P /\rho_n) \leq  h +  \frac{K\norm{H}_{op}^2}{n} \ \biggr\rvert \ \boldsymbol{\xi}(\omega) \right) \rightarrow 1
\end{align*} 
Now for the lower bound, we may use Theorem 6 \citep{eldridge-beyond-davis-kahan} along with the bound from Step 2 to conclude that:
\begin{align*}
\inf_{\omega \in E_\delta} P\left( \lambda_r(A / \rho_n) -  \lambda_r(P / \rho_n) \geq h \ \biggr\rvert \ \boldsymbol{\xi}(\omega) \right) \rightarrow 1  
\end{align*} 

If the spectrum of $T_w$ contains negative eigenvalues, notice that we may argue in analogous fashion due to Theorem 11 \citep{eldridge-beyond-davis-kahan}; however for the negative eigenvalues, the terms in the upper and lower bounds are flipped.      
Now putting together these two bounds we have that $ \lambda_r(A /\rho_n) -  \lambda_r(P /\rho_n) = o_P(\sqrt{n})$ under the stated conditions and the result follows. \qed 

\subsection{Proof of Proposition \ref{estimated-sparsity-parameter}}
\label{section:appendix-sparsity-subsampling-proof}
Consider the following decomposition:
\begin{align*}
\begin{split}
 & \ \ \ \ \mathbbm{1}\left(\sqrt{b}\left[\frac{\lambda_r(A^{(n,b,i)})}{b\hat{\rho}_n} - \frac{\lambda_r(A^{(n)})}{n\hat{\rho}_n}\right] \leq t \right) 
 \\ &= \mathbbm{1}\ \left( \sqrt{b}\left[\frac{\lambda_r(A^{(n,b,i)})}{b\rho_n} - \lambda_r(w)\right] \cdot \frac{\rho_n}{\hat{\rho}_n} -  \sqrt{b}\left[ \frac{\lambda_r(A^{(n)})}{n \rho _n} - \lambda_r(w) \right] \cdot \frac{\rho_n}{\hat{\rho}_n} \ \leq t \right)
 \end{split}
\end{align*}
Define the events:
\begin{align*}
E_n^{(1)} = \left\{ \left| \frac{\rho_n}{\hat{\rho}_n} - 1 \right| \leq \epsilon_1   \right\} , \ \ \   E_n^{(2)} = \left\{   \sqrt{b}\left| \frac{\lambda_r(A^{(n)})}{n \rho _n} - \lambda_r(w) \right| \cdot \frac{\rho_n}{\hat{\rho}_n} \leq \epsilon_2   \right\}
\end{align*}
Now we will treat the cases $t>0$, $t<0$, and $t = 0$ separately.  For $t>0$, we may choose $\epsilon_2$ so that $t-\epsilon_2 > 0$.  Moreover, for any $\epsilon_2 >0$, $P(E_n^{(2)}) \rightarrow 1$. Furthermore, since $t > 0$, we have that, for any $0 < \epsilon_1 < 1$:
\begin{align*}
& \ \ \ \ \mathbbm{1}\ \left( \sqrt{b}\left[\frac{\lambda_r(A^{(n,b,i)})}{b\rho_n} - \lambda_r(w)\right] \cdot \frac{\rho_n}{\hat{\rho}_n} \leq t + \epsilon_2 \right) \mathbbm{1}(E_n^{(1)}) 
\\ &\leq \mathbbm{1}\ \left( \sqrt{b}\left[\frac{\lambda_r(A^{(n,b,i)})}{b\rho_n} - \lambda_r(w)\right] \leq \frac{t + \epsilon_2}{1-\epsilon_1} \right)  
\end{align*}
We can make a similar argument for the lower bound.  Let $\alpha = \frac{t - \epsilon_2}{1+\epsilon_1}$ and $\beta=\frac{t + \epsilon_2}{1-\epsilon_1}$.  For $t>0$, we have the bound: 
\begin{align*}
U_{n,b}\left(\alpha,P_n \right) \leq \hat{L}_{n,b}(t,P_n) \mathbbm{1}(E_n^{(1)} \cap E_n^{(2)}) \leq U_{n,b}\left(\beta,P_n \right)  
\end{align*}
For $t<0$, we have a similar bound, but with $ \alpha = \frac{t - \epsilon_2}{1-\epsilon_1}$ and  $\beta= \frac{t + \epsilon_2}{1+\epsilon_1}$.  For $t=0$, we have $ \alpha= \frac{t - \epsilon_2}{1-\epsilon_1}$ and $ \beta = \frac{t + \epsilon_2}{1-\epsilon_1}$.  In all three cases, we may choose $\epsilon_1$ and $\epsilon_2$ so that $\alpha$ and $\beta$ are themselves continuity points of $J(\cdot,P_\infty)$.  Moreover, they can be chosen so that, for any $\delta >0$, $|J(\cdot,P_\infty) - J(t,P_\infty)| < \delta$ at $\alpha$ and $\beta$. From the proof of Theorem \ref{subsampling-fixed-b}, $U_{n,b}(t,P_n) \xrightarrow{P} J(t,P_\infty)$ under our assumptions.  Therefore,  with probability tending to $1$:
\begin{align*}
J(t,P_\infty) - 2 \delta \leq \hat{L}_{n,b}(t,P_n) \leq  J(t,P_\infty) + 2 \delta
\end{align*}
The result follows. \qed

\begin{center}
\large{Supplement to ``Subsampling Sparse Graphons Under Minimal Assumptions''}
\end{center}






\setcounter{section}{0}
\renewcommand{\thesection}{S\arabic{section}}
\renewcommand{\thesubsection}{S\arabic{section}.\arabic{subsection}}     
\renewcommand{\thetable}{S\arabic{table}}   
\renewcommand{\thefigure}{S\arabic{figure}}
\renewcommand{\theequation}{S.\arabic{equation}}

\section{Proofs for Section \ref{uniform_validity}}
\label{appendix-uniform-validity}

\subsection{Proof of Theorem \ref{uniform-coverage-theorem-vertex}} 
By Lemma A.1 (vii) \citep{Romano-Shaikh-Uniform}, if
\begin{align}
\label{uniform-subsampling-deviation}
P\left(\sup_{t \in \mathbb{R}} | L_{n,b}(t , P_n) - J_n(t, G_n) | \leq \epsilon \right) > 1-\delta
\end{align} 
then it follows that:
\begin{align}
P\left(\tau_n[\hat{\theta}_{n}(G_n) - \theta(G_\infty)] \leq c_{n,b}(1-\alpha) \right) \geq 1 - (\alpha + \epsilon + \delta)
\end{align} 
Therefore, it suffices to show that for any $\epsilon >0$, there exists $N$ such that for all $n > N$ and for any $\delta >0$, the relation (\ref{uniform-subsampling-deviation}) holds uniformly in $\mathcal{P}_n$.  To this end, notice that we may bound the complement of (\ref{uniform-subsampling-deviation}) by:
\begin{align*}
& 1 - \left[ \sup_{P \in \mathcal{P}_n} P\left( \sup_{t \in \mathbb{R}} \left| L_{n,b}(t , P_n) - U_{n,b}(t, P_n)\right| > \epsilon/3 \right) \right.
\\ &\left. \ \ \ \ \ \ \ + \sup_{P \in \mathcal{P}_n} P\left( \sup_{t \in \mathbb{R}} \left| U_{n,b}(t , P_n) - J_{n,b}(t, P_n)\right| > \epsilon/3 \right) \right.
 \\ &\left. \ \ \ \ \ \ \ + \sup_{P \in \mathcal{P}_n} \mathbbm{1}\left(\sup_{t \in \mathbb{R}} \left| J_{n,b}(t,, P_n) -  J_{n}(t , P_n) \right| > \epsilon/3 \right) \  \right]
\end{align*}
It suffices to upper bound the second term; the other terms converge to $0$ by assumption.  Using analogous reasoning to Theorem 2.1 \citep{Romano-Shaikh-Uniform} with the Hoeffding representation used in (\ref{hoeffding-representation}), we have the following bound:
\begin{align*} 
& \ \ \ \ \ \ \sup_{P \in \mathcal{P}_n} P\left( \sup_{t \in \mathbb{R}} \left| U_{n,b}(t , P_n) - J_{n,b}(t, P_n)\right| > \epsilon \right) 
\\ & \leq \frac{1}{\epsilon}\sup_{P \in \mathcal{P}_n} \mathbb{E} \left[\frac{1}{n!} \sum_\sigma \sup_{t \in \mathbb{R}} \left| F(j_{\sigma(1)}, \ldots j_{\sigma(n)};t) - J_{n,b}(t, P_n)\right| \right]
\\ & \leq \frac{1}{\epsilon} \sup_{P \in \mathcal{P}_n} \int_0^1 P\left( \sup_{t \in \mathbb{R}} \left| F(j_{1}, \ldots j_{n};t) - J_{n,b}(t, P_n)\right| > u \right) du \rightarrow 0  
\end{align*}
where the last line follows from the Dvoretzky-Kiefer-Wolfowitz inequality.   \qed

\subsection{Proof of Theorem \ref{uniform-coverage-theorem-p}} 
It suffices to bound:
\begin{align*}
& 1 - \left[ \sup_{P \in \mathcal{P}_n} P\left( \sup_{t \in \mathbb{R}} \left| L_{n,b}^\prime(t , P_n) - V_{n,b}(t, P_n)\right| > \epsilon/3 \right) \right.
\\ &\left. \ \ \ \ \ \ \ + \sup_{P \in \mathcal{P}_n} P\left( \sup_{t \in \mathbb{R}} \left| V_{n,B}(t , P_n) - J(t, P_n)\right| > \epsilon/3 \right) \right.
 \\ &\left. \ \ \ \ \ \ \ + \sup_{P \in \mathcal{P}_n} \mathbbm{1}\left(\sup_{t \in \mathbb{R}} \left| J_{n}(t, P_n) -  J(t , P_n) \right| > \epsilon/3 \right) \  \right]
\end{align*}
What remains is bounding the second term above. Conditional on $G_n$, observe that $V_{n,B}( \cdot , P_n)$ may be interpreted as an empirical CDF of i.i.d. random variables with CDF $\mathbb{E}[V_{n,B}( \cdot , P_n) \ | \ G_n]$.  Therefore, by Dvoretzky-Kiefer-Wolfowitz inequality:
\begin{align*}
P\left( \sup_{t \in \mathbb{R}} \ \bigl| V_{n,B}( t , P_n) -  \mathbb{E}[V_{n,B}( t, P_n) | G_n] \bigr| > \epsilon   \right) \leq \exp(-2M\epsilon^2)
\end{align*}   
Recall the representation:
\begin{align*}
\mathbb{E}[V_{n,B}(t,P_n) \ | \ G_n ] = \sum_{j=0}^n q_j \left(\frac{1}{N_j} \sum_{i=1}^{N_j} \mathbbm{1}\left( \tau_{j^*}[\hat{\theta}_{n,j^*,i}( G_{n,j,i}^*) - \theta(G_\infty)] \leq t \right) \right)
\end{align*}
Let:
\begin{align*}
U_{n,j}(t,P_n) := \frac{1}{N_j} \sum_{i=1}^{N_j} \mathbbm{1}\left( \tau_{j^*}[\hat{\theta}_{n,j^*,i}( G_{n,j,i}^*) - \theta(G_\infty)] \leq t \right)
\end{align*}
Now observe that:
\begin{align*}
& \ \ \ \sup_{P \in \mathcal{P}_n} P\left( \sup_{t \in \mathbb{R}} \left| \mathbb{E}[V_{n,B}(t , P_n) \ | \ G_n ] - J(t, P_n)\right| > \epsilon \right)
\\ &  \ \leq \frac{1}{\epsilon} \sup_{P \in \mathcal{P}_n}  \sum_{j=0}^n q_j\mathbb{E}\left[  \sup_{t \in \mathbb{R}} \ \biggl| U_{n,j}(t,P_n)  - \mathbb{E}[U_{n,j}(t,P_n)] \biggr|\right] 
\\ & \ \ + \frac{1}{\epsilon} \sup_{P \in \mathcal{P}_n}\sum_{j=0}^n q_j \sup_{t \in \mathbb{R}} \left| P\left( \tau_{j^*}[\hat{\theta}_{n,j^*,i}( G_{n,j,i}^*) - \theta(G_\infty)] \leq t \right)- J_{n,j}(t,P_n) \right|
\\ & \ \ +  \frac{1}{\epsilon} \sup_{P \in \mathcal{P}_n}  \sum_{j=0}^n q_j \sup_{t \in \mathbb{R}} \left| J_{n,j}(t,P_n)- J(t, P_\infty)\right|
\\ & \ \ \ = \mathbf{I} +   \mathbf{II} +  \mathbf{III}
\end{align*}
Similar to the proof of Theorem \ref{p-subsampling-theorem}, we will restrict our attention to subsample sizes that are chosen with high probability.  Recall that $C_n = \{ l_n, l_n +1, \ldots, u_n \}$, where $c_n = 3 \sqrt{n p_n \log n}$,  $l_n = \lfloor np_n - c_n \rfloor$ and $u_n = \lceil np_n + c_n \rceil$. Furthermore, let $F^{(j)}(j_1, \ldots j_n)$ be a sum of independent elements analogous to (\ref{independent-sum-representation}), with a kernel of size $j$ corresponding to $\mathbbm{1}\left( \tau_{j^*}[\hat{\theta}_{n,j^*,i}( G_{n,j,i}^*) - \theta(G_\infty)] \leq t \right)$.  For $\mathbf{I}$, we have that:
\begin{align}
\frac{1}{\epsilon} \sum_{j \in C_n} q_j \int_0^1 P\left(\sup_{t \in \mathbb{R}} \left| F^{(j)}(j_1, \ldots j_n ; t) - \mathbb{E}[F^{(j)}(j_1, \ldots j_n ; t)] \right|  > u \right) \ du \rightarrow 0
\end{align}
Since $(1-\sum_{j \in C_n} q_j) \rightarrow 0$, $\mathbf{I} \rightarrow 0$. Now, since deletion of isolated nodes is uniformly negligible, following similar reasoning to the proof of Theorem \ref{p-subsampling-theorem}, $\mathbf{II} \rightarrow 0$.  $\mathbf{III} \rightarrow 0$ by assumption and the fact that  $(1-\sum_{j \in C_n} q_j) \rightarrow 0$.  \qed

\section{Uniform Validity for Counts and Eigenvalues}
 In this section, we will establish conditions under which uniform validity holds for two important classes of functionals: counts and eigenvalues. We will begin with the former. 

\subsection{Uniform Validity for Count Functionals}
\label{uniform-validity-counts}
We will now introduce the count functionals under consideration, which were introduced in  \citet{Bickel-Chen-Levina-method-of-moments}. Let $R$ be a subset of $\{(i,j) \ | \ 1 \leq i \leq j \leq n \}$ and let $G_n(R)$ denote the subgraph induced by the vertices in $R$, which we will denote $\mathcal{V}(R)$.  Furthermore, let $\overline{R} = \{(i,j) \not\in R, i \in \mathcal{V}(R), j \in \mathcal{V}(R) \}$. Let $p= |\mathcal{V}(R)|$ and $e = | \mathcal{E}(R)|$.  The probability that an induced subgraph is exactly equal to $R$ is given by:
\begin{align}
\begin{split}
P(R) &= P(G_n(R) = R)
\\ &= \mathbb{E}\left[\prod_{(i,j) \in R} h_n(\xi_i, \xi_j)  \prod_{(i,j) \in \overline{R}} (1-h_n(\xi_i, \xi_j))\right]
\end{split}
\end{align}
When the graph sequence is sparse, $P(R)$ is uninformative, as $P(R) \rightarrow 0$. Instead,  consider the normalized functional below:
\begin{align}
\tilde{P}(R) = \rho_n^{-e} \ P(R)
\end{align} 
Our estimator of $\tilde{P}(R)$ is given by:
\begin{align}
\hat{P}(R) = \rho_n^{-e} \frac{1}{{n \choose p} \ |\mathrm{Iso}(R)| } \sum_{S \sim R} \mathbbm{1}(S \subseteq G_n)
\end{align}
Below, we will characterize the class of distributions in terms of the pair $(\rho_n, w(u,v))$.  We will use $\mathcal{W}$ to denote a class of graphons, and $\mathcal{S}$ to denote a class of sparsity sequences, with $\mathcal{S}_n$ denoting the set of real numbers formed by evaluating the sequences in $\mathcal{S}$ at $n$.

It turns out that the behavior of $\tilde{P}(R)$ is driven by a U-statistic involving the edge structure of the subgraph.  Our conditions on $\mathcal{W}$ will be defined in terms of this kernel. Let:  
\begin{align}
g(\xi_1, \ldots , \xi_p \ ;  w)  =  \frac{1}{|\mathrm{Iso}(R)|}\sum_{\mathcal{V}(S) = \{1, \ldots,  p \}, \ S \sim R} \  \prod_{(i,j) \in S} w(\xi_i,\xi_j)
\end{align}
For example, in the case of a two-star, for some $S \sim R$ we have that:
\begin{align*}
\prod_{(i,j) \in S} w(\xi_i,\xi_j) = w(\xi_1, \xi_2) w(\xi_1, \xi_3)
\end{align*}
The parameter of interest is given by $\theta(G_\infty) = \lim_{n \rightarrow \infty} \tilde{P}(R) =  \mathbb{E}[g(\xi_1, \ldots, \xi_p)]$ when $\rho_n \rightarrow 0$.  Now, let $g_1(x \ ;w)$ denote the following functional:  
\begin{align}
g_1(x \  ;   w) &= \mathbb{E}[ g(\xi_1, \ldots ,\xi_p \ ;  w)  \ | \ \xi_1=x \ ]   
\end{align}
Suppose that functions in $\mathcal{W}$ satisfy the following conditions:
\begin{align}
\label{lyapunov-condition}
\sup_{w \in \mathcal{W}} \mathbb{E}\left[ w^{2(e+1)}(\xi_i, \xi_j) \right] < \infty,  \ \ \inf_{w \in \mathcal{W}} \mathrm{Var}(g_1(\xi_1  \ ; w) ) > 0
\end{align}

Now suppose that $\mathcal{S}$ is defined so that the following properties hold:
\begin{align}
\label{sparsity-condition}
\lim_{n \rightarrow \infty} \sup_{\rho_n \in \mathcal{S}_n} \rho_n \rightarrow 0, \ \ \ \lim_{n \rightarrow \infty} \inf_{\rho_n \in \mathcal{S}_n} b_n \rho_n  = \infty
\end{align}

We have the following result for vertex subsampling; an analogous theorem holds for $p$-subsampling under some additional assumptions stated in Theorem \ref{uniform-coverage-theorem-p}. While uniform validity holds under mild conditions on $\mathcal{W}$, we would like note that convergence rates depend mostly on $\{ \rho_n \}_{n \in \mathbb{N}}$.   

\begin{proposition}[Uniform Validity of Vertex Subsampling for Counts]
\label{uniform-count-theorem}
Suppose that $R$ is acyclic or a $p$-cycle.  Furthermore, suppose that $\mathcal{W}$ satisfies (\ref{lyapunov-condition}), $\mathcal{S}$ satisfies (\ref{sparsity-condition}), and $b_n  =o(n)$.  Then, conditions (\ref{sampling-dist-convergence}) and (\ref{subsampling-dist-convergence}) are satisfied.  
\end{proposition}
\begin{proof}

For the sequence $\{b_n\}_{n \in \mathbb{N}}$, we will begin by establishing:
\begin{align*}
\sup_{P_n \in \mathcal{P}_n} \sup_{t \in \mathbb{R}} \left| J_{n,b}(t, P_n) -  J_{n}(t, P_n) \right| \rightarrow 0
\end{align*}

By triangle inequality, we may further reduce this problem to showing:
\begin{align*}
\label{kolmogrov-convergence}
\sup_{P_n \in \mathcal{P}_n} \sup_{t \in \mathbb{R}} \left| J_{n,b}(t, P_n) - J(t,P_\infty) \right| \rightarrow 0, \ \ \  
\sup_{P_n \in \mathcal{P}_n} \sup_{t \in \mathbb{R}} \left| J_{n}(t, P_{n}) - J(t,P_\infty) \right| \rightarrow 0 
\end{align*}

Since $b_n \leq n$, it will turn out that establishing convergence for the first term above implies convergence of the second term.  Furthermore, since the posited limit is continuous, it suffices to show convergence in distribution due to Polya's Theorem. In essence, we will use a triangular array central limit theorem to establish convergence in distribution of a dominant term for any sequence $\{ P_n \}_{n \in \mathbb{N}}$ and show that the remainder terms converge in probability uniformly over $\mathcal{P}_n$.  Weak convergence would then follow from Slutsky's Theorem. 

The proof strategy largely follows the proof of Theorem 1 of \citet{Bickel-Chen-Levina-method-of-moments}.  However, the estimation of $\rho_n$ turns out to be a non-negligible perturbation, so we do not include it in our analysis below.  We also show that the censoring related to the sparsity sequence is negligible for count statistics under more general conditions than those for eigenvalue statistics.  Finally, we verify that the bounds indeed hold uniformly over a large class of distributions.  Now, let:
\begin{align}
\begin{split}
U_1 &= \rho_n^{-e} \sum_{S \subset G_{n,b}, \ S \sim R} T(S) - \mathbb{E}[T(S) \ | \  \boldsymbol{\xi}_n ]
\\ U_2 &= \mathbb{E}[\hat{P}(R) \ | \ \boldsymbol{\xi}_n]\rho_n^{-e} - \tilde{P}(R)
\end{split}
\end{align}
where $T(S)$ is given by:
\begin{align}
T(S) = \frac{1}{{b_n \choose p} |Iso(R)|} \prod_{(i,j) \in \mathcal{E}(S)} A_{ij}
\end{align}

We will start by showing that $\sup_{P_n \in \mathcal{P}_n} \sqrt{b_n} U_1 = o_P(1)$. We have that:
\begin{align*}
\mathrm{Var}(U_1) = \sum_{S_1 \sim R, S_2 \sim R} \mathrm{Cov}(T(S_1), T(S_2)) \rho_n^{-2q}
\end{align*}
Using analogous reasoning to the above reference, observe that the covariance is zero for $\mathcal{E}(S_1) \cap \mathcal{E}(S_2) = \emptyset$; therefore it suffices to consider $S_1, S_2$ such that $S_1 \cap S_2$ has $c$ vertices and $d$ edges in common.
Since $A_{ij}^2 = A_{ij}$ and $|\mathcal{E}(S_1 \cup S_2)| = 2q-d$, we have that:
\begin{align*} 
 \mathbb{E}[T(S_1)T(S_2)] &= \mathbb{E}\left[\prod_{(i,j) \in \overline{S_1 \cap S_2}} A_{ij} \prod_{(i,j) \in S_1 \cap S_2} A_{ij}^2 \right] 
 \\ &= \mathbb{E}\left[\prod_{(i,j) \in S_1 \cup S_2} A_{ij} \right]
 \\ &\leq \rho_n^{2q-d} \ \mathbb{E}\left[\left(w(\xi_i, \xi_j) \wedge \rho_n^{-1}\right)^{2q-d} \right]
 \\ & \leq  \rho_n^{2q-d} \  \mathbb{E}\left[w^{2q-d}(\xi_i, \xi_j)\right]
\end{align*} 
Observe that for $\mathbb{E}[T(S_1)] \cdot \mathbb{E}[T(S_2)]$ we have the bound $\rho_n^{2q} \ \mathbb{E}[w^{q}(\xi_i, \xi_j)]^2$. Therefore, $\mathrm{Cov}(T(S_1), T(S_2) \leq \rho_n^{2q-d}\{ \mathbb{E}[w^{2q}(\xi_i, \xi_j)] \wedge 1 \}$.  As observed in \citep{Bickel-Chen-Levina-method-of-moments}, there are $O(n^{2p-c})$ terms with $c$ vertices in common, so that the total contribution of these terms to $\mathrm{Var}(U_1)$ is, uniformly of $P_n \in \mathcal{P}_n$:
\begin{align*}
O\left(n^{-c} \rho_n^{-d} \int w^{2q}(u,v) \ du \ dv \right)
\end{align*}  
Therefore, under the assumption that $\inf_{P_n \in \mathcal{P}_n} b_n \rho_n \rightarrow \infty$, we have that, for any acyclic graph or $p$-cycle, by law of total variance:
\begin{align*}
\label{order-variance}
\sup_{P_n \in \mathcal{P}_n}\text{Var}(U_1 )/b_n = \sup_{P_n \in \mathcal{P}_n}\mathbb{E}[\text{Var}(U_1  \ | \ \boldsymbol{\xi}_n)]/b_n \rightarrow 0
\end{align*}
Therefore, by Chebychev's inequality $\sqrt{b_n}U_1 \xrightarrow {P} 0$, uniformly in $\mathcal{P}_n$. Now, we will bound $U_2$. We will further approximate $U_2$ and show that the remainder is uniformly negligible.  Let:

\begin{align}
\begin{split}
V_1 &= \frac{1}{{b_n \choose p}|\mathrm{Iso}(R)|}\sum_{S \subset G_{n,b}, \ S \sim R } \ \prod_{(i,j) \in S} w(\xi_i, \xi_j)
\\ V_2 &= \frac{1}{{b_n \choose p}|\mathrm{Iso}(R)|}\sum_{S \subset G_{n,b}, \ S \sim R } \ \prod_{(i,j) \in S} w(\xi_i, \xi_j) \wedge \rho_n^{-1}
\\ V_3 &= \frac{1}{{b_n \choose p}|\mathrm{Iso}(R)|}\sum_{S \subset G_{n,b}, \ S \sim R} \ \prod_{(i,j) \in S}  w(\xi_i, \xi_j) \wedge \rho_n^{-1} \prod_{(i,j) \in \bar{S}} (1-h_n (\xi_i, \xi_j))
\end{split}
\end{align}
Further let $U_4 = V_1 - \mathbb{E}[V_1]$ and observe that $U_2 = V_3 - \mathbb{E}[V_3]$; we will show that $\sqrt{b_n} U_2$ may be well approximated by $\sqrt{b_n}U_4$, uniformly over $\mathcal{P}_n$.  By Chebychev's inequality again,
we have that:
\begin{align*}
P\left(\left|\sqrt{b_n}(V_2 - \mathbb{E}[V_2]) - \sqrt{b_n}(V_3 - \mathbb{E}[V_3]) \right|  > \epsilon \right) \leq \frac{b_n \mathrm{Var}(V_2-V_3)}{\epsilon^2}
\end{align*}

Now observe that $V_1 - V_2$ is a U-statistic of order $p$. Therefore, by Lemma A of \citet{Serfling-Approximation-Theorems}, page 183 and the $c_r$ inequality, for some $(k,l) \in \bar{S}$:
\begin{align*}
\begin{split} 
& \ \ \ \ \ \mathrm{Var}(V_2-V_3) 
\\ & \leq \frac{p}{b_n} \ \mathbb{E}\left[\left( \frac{1}{|\mathrm{Iso}(R)|}\sum_{\mathcal{V}(S) = \{1,\ldots,  p \}, \ S \sim R} \ \prod_{(i,j) \in S} w(\xi_i, \xi_j) \wedge \rho_n^{-1} \left[1- \prod_{(i,j) \in \bar{S}}(1- \rho_n  w(\xi_i, \xi_j) \wedge 1) \right]\right)^2\right]
\\ & \leq \frac{p}{b_n |\mathrm{Iso}(R)|} \mathbb{E}\left[ \left( \prod_{(i,j) \in S}  w(\xi_i, \xi_j) \wedge \rho_n^{-1} \times \rho_n w(\xi_{k}, \xi_{l}) \wedge 1 \right)^2 \right]
\\ &\leq \rho_n^2 \frac{p\int_{[0,1]^2} w^{2(e+1)} (u,v) du \ dv }{b_n |\mathrm{Iso}(R)|}
\end{split}
\end{align*} 
Under our assumptions, $b_n\mathrm{Var}(V_2-V_3) \rightarrow 0$ uniformly in $\mathcal{P}_n$.  Now we will show $\sqrt{b_n} (V_1 - V_2)$ is asymptotically negligible. Observe that the function $g(x) = \prod_{i=1}^e x_i$ satisfies the relation:
\begin{align*}
|g(x)-g(y)| &\leq (\norm{x}_\infty^{e-1} + \norm{y}_\infty^{e-1}) \norm{x-y}_\infty
\\ \implies |g(x)-g(y)|^2  &\leq (\norm{x}_\infty^{d-1} + \norm{y}_\infty^{e-1})^2 \norm{x-y}_\infty^2
\\  \implies |g(x)-g(y)|^2  &\leq (\norm{x}_\infty^{2(e-1)} + 2\norm{x}_\infty^{e-1} \norm{y}_\infty^{e-1} + \norm{y}_\infty^{2(e-1)}) \norm{x-y}_\infty^2
\end{align*}
Therefore, it follows that $\mathrm{Var}(\sqrt{b_n}(V_1-V_2))$ is upper bounded by:
 \begin{align*}
& \ \ \ \   \frac{p}{|\mathrm{Iso}(R)|}\mathbb{E}\left[ \left(\prod_{(i,j) \in S} w(\xi_i, \xi_j) \wedge \rho_n^{-1} - \prod_{(i,j) \in S}w(\xi_i, \xi_j)\right)^2  \right]
\\ & \leq \frac{p}{|\mathrm{Iso}(R)|} \mathbb{E}\left[ \ \left(\max_{(i,j) \in S} w(\xi_i, \xi_j) \wedge \rho_n^{-1}\right)^{2(e-1)} \right.
\\ & \left. \ \ \ \ \ \ \ \ \ \ \ \ \ \ \ \ \ \  + \ 2 \left(\max_{(i,j) \in S} w(\xi_i, \xi_j) \wedge \rho_n^{-1}\right)^{e-1} \left(\max_{(i,j) \in S} w(\xi_i, \xi_j)\right)^{e-1} \right.
\\ & \left. \left. \ \ \ \ \ \ \ \ \ \ \ \ \ \ \ \ \ \ + \left(\max_{(i,j) \in S} w(\xi_i, \xi_j)\right)^{2(e-1)} \right\} \times \max_{(i,j) \in S} \left| w(\xi_i, \xi_j) \wedge  \rho_n^{-1} - w(\xi_i, \xi_j) \right|^2    \right]
 \\ & \leq  \frac{4p}{|\mathrm{Iso}(R)|} \mathbb{E}\left[ \max_{(i,j) \in S} w^{2e}(\xi_i, \xi_j) \mathbbm{1}\left(\max_{(i,j) \in S} w(\xi_i, \xi_j) > \rho_n^{-1} \right) \right]
\end{align*}

Since $\sup_{w \in \mathcal{W}} \mathbb{E}[w^{2e}(\xi_i,\xi_j)] < \infty$, it follows that:
\begin{align*}
\sup_{w \in \mathcal{W}} \mathbb{E}\left[ \max_{(i,j) \in S} w^{2e}(\xi_i,\xi_j) \right] & \leq \sup_{w \in \mathcal{W}} e \times \mathbb{E}\left[w^{2e}(\xi_i,\xi_j) \right] < \infty 
\end{align*}
Therefore, it follows that the collection $\{\max_{(i,j) \in S} w^{2e}(\xi_j,\xi_j)\}_{w \in \mathcal{W}}$ is uniformly integrable, and furthermore, since $\inf_{P_n \in \mathcal{P}_n} \rho_n^{-2e} \rightarrow \infty$, we have that the above term converges to 0 uniformly in $\mathcal{P}_n$. Now, we have that  $\sqrt{b}U_4$ converges to a Normal distribution under our assumptions for any sequence $\{ P_n \}_{n \in \mathbb{N}}$ by a Lindeberg Central Limit Theorem applied to the dominant term of the Hoeffding projection.  
Finally it suffices to verify that: 
\begin{align*}
 \lim_{n \rightarrow \infty} \sup_{P \in \mathcal{P}_n} P \left( \sup_{t \in \mathbb{R}} \left|L_{n,b}(t,P_n) - U_{n,b}(t,P_n) \right| > \epsilon \right) &= 0
\end{align*}
By triangle inequality, the quantity above is bounded by:
\begin{align*}
\sup_{P \in \mathcal{P}_n} P \left( \sup_{t \in \mathbb{R}} \left|L_{n,b}(t,P_n) - J_{n,b}(t,P_n) \right| > \epsilon \right) + \sup_{P \in \mathcal{P}_n} P \left( \sup_{t \in \mathbb{R}} \left|U_{n,b}(t,P_n) - J_{n,b}(t,P_n) \right| > \epsilon \right) 
\end{align*}
The second term was bounded in the proof of Theorem \ref{uniform-coverage-theorem-vertex}; it suffices to control the first term.  For each $\omega \in \Omega$, we may write:
\begin{align*}
& \sup_{t \in \mathbb{R}} \left|L_{n,b}(t,P_n) - J_{n,b}(t,P_n) \right| 
\\ &= \max \biggl(\sup_{t^+(\omega) \in \mathbb{R}} \left\{ L_{n,b}(t,P_n) - J_{n,b}(t,P_n) \biggr\} \ ,   \sup_{t^-(\omega) \in \mathbb{R}} \biggl\{   J_{n,b}(t,P_n) - L_{n,b}(t,P_n)   \biggr\} \right)
\end{align*}
where $t^+(\omega)$ and $t^-(\omega)$ are defined as $t$ such that the differences above are nonnegative.  Now by union bound, we have that:
\begin{align*}
\begin{split}
&  \ \ \ P \left( \sup_{t \in \mathbb{R}} \left|L_{n,b}(t,P_n) - J_{n,b}(t,P_n) \right| > \epsilon \right)
\\ &  \leq P\left(\sup_{t^+(\omega) \in \mathbb{R}} \biggl\{ L_{n,b}(t,P_n) - J_{n,b}(t,P_n) \biggr\} > \epsilon  \right) + P\left (\sup_{t^-(\omega) \in \mathbb{R}} \biggl\{   J_{n,b}(t,P_n) - L_{n,b}(t,P_n)   \biggr\} > \epsilon \right)
\end{split}
\end{align*}
We will bound the first term; the bound for the second term is analogous.  Define the event:
\begin{align}
\label{set-for-estimated-parameter}
A_n = \left\{ \sqrt{b}\left| \hat{\theta}_n(G_n) - \theta(G_\infty)\right| \leq a_n \right\}
\end{align}
We may choose $a_n \rightarrow 0$ such that $P(A_n) \rightarrow 1$.  Now by law of total probability:
\begin{align*}
\begin{split}
& \ \ \ \sup_{P_n \in \mathcal{P}_n} P\left(\sup_{t^+(\omega) \in \mathbb{R}} \biggl\{ L_{n,b}(t,P_n) - J_{n,b}(t,P_n) \biggr\} > \epsilon  \right) 
\\ &\leq \sup_{P_n \in \mathcal{P}_n} P\left(\sup_{t^+(\omega) \in \mathbb{R}} \biggl\{ U_{n,b}(t+a_n,P_n) - J_{n,b}(t,P_n) \biggr\} > \epsilon \right) + \sup_{P_n \in \mathcal{P}_n} P(A_n^c) 
\\ & \leq \sup_{P_n \in \mathcal{P}_n} P\left(\sup_{t \in \mathbb{R}} \left| U_{n,b}(t+a_n,P_n) - J_{n,b}(t+a_n ,P_n) \right|  > \epsilon/3 \right) 
\\ &  \ \ \ + \sup_{P_n \in \mathcal{P}_n} \mathbbm{1}\left( \sup_{t \in \mathbb{R}} |J_{n,b}(t+a_n ,P_n) - J(t+a_n, P_\infty)| > \epsilon/3 \right) 
\\ &  \ \ \ + \sup_{P_n \in \mathcal{P}_n} \mathbbm{1}\left(\sup_{t \in \mathbb{R}} |J(t+a_n ,P_\infty)- J(t, P_\infty)| > \epsilon/3 \right) + \sup_{P_n \in \mathcal{P}_n} P_n(A_n^c)
\\ &= \mathbf{I} + \mathbf{II} + \mathbf{III} + \mathbf{IV} 
\end{split}
\end{align*}
By similar reasoning to Theorem \ref{uniform-coverage-theorem-vertex}, $\mathbf{I} \rightarrow 0$.  By Polya's Theorem, $\mathbf{II} \rightarrow 0$.  For $\mathbf{III}$, since $\sigma^2(P_\infty) \geq c^2$, a bound based on the Levy concentration function of a Gaussian random variable gives:
\begin{align*}
\mathbf{III} \leq \frac{a_n}{\sqrt{2\pi} c} \rightarrow 0
\end{align*} 
Finally, we have that $\mathbf{IV} \rightarrow 0$ since $P(A_n^c) \rightarrow 0$.  The negative case is analogous. \qed

\begin{remark}
Using a Berry-Esseen Theorem for U-Statistics due to \citet{Van-zwet-berry-esseen-symmetric}, one may also derive finite sample bounds.  However, the rates heavily depend on $\{\rho_n \}_{n \in \mathbb{N}}$ and stronger integrability conditions are required.  
\end{remark}
\begin{remark}
The above argument may be easily generalized to yield uniform validity for subsampled estimators with sparsity parameter $\rho_n$ estimated on the size $n$ sample. Formally, it suffices to show: $\lim_{n \rightarrow \infty} \sup_{P \in \mathcal{P}_n} P( \sup_{t \in \mathbb{R}} |\hat{L}_{n,b}(t,P_n) - U_{n,b}(t,P_n) | > \epsilon ) = 0$.   
\end{remark}
\end{proof}

\subsection{Uniform Validity for Eigenvalues}
\label{uniform-validity-eigenvalues}
In this section, we will provide some simple sufficient conditions under which uniform validity holds for eigenvalue statistics.  As one might expect, the nature of the function class is much different from the class considered for count functionals in Section \ref{uniform-validity-counts}.  Let $\mathcal{W}$ be a collection of functions with an integral operator satisfying:
\begin{align}
\label{rank-condition}
r \leq \mathrm{rank}(T_w) \leq K
\end{align}
for some $0 < r \leq K < \infty$. Suppose that $ w \in \mathcal{W}$ further satisfy:
\begin{align}
\label{eigenvalue-condition}
\begin{split}
\norm{w(u,v)}_\infty &\leq  D \ a.s.
\\ |\lambda_i| &\geq d \ \ \ \forall \ 1 \leq i \leq k 
\\ |\lambda_i - \lambda_{i-1}| &\geq \epsilon  \ \ \ \forall \ 1 \leq i \leq k
\\ \int \phi_r^4 \ dP &\geq c
\end{split}
\end{align}
for $D < \infty$, $c >1 $, $d > 0$, and $\epsilon > 0$. Further suppose $\{\rho_n\}_{n \in \mathbb{N}} \in \mathcal{S}$ satisfy:
\begin{align}
\label{sparsity-eigenvalue-condition}
\lim_{n \rightarrow \infty}\sup_{\rho_n \in \mathcal{S}_n} \rho_n = 0, \ \ \ \inf_{\rho_n \in \mathcal{S}_n} \rho_n  = \omega( 1/\sqrt{b_n})
\end{align}

We have the following result; a proof is omitted as our result may be verified by tracing steps in the proof of Theorem 5.6 of \citet{koltchinksii-gine-kernel-operator}.  Also note that we may use a triangular array strong law of large numbers, which holds under a uniform bound on the fourth moment, to derive a uniform variant of Theorem 3.1 of the above reference.  This result is needed for generalizing Step 3 of the proof of Theorem \ref{subsampling-eigenvalues}.   
\begin{proposition}[Uniform Validity of Vertex Subsampling for Eigenvalues]
	\label{uniform-eigenvalue-theorem}
	Let $\mathcal{W}$ be a collection of functions such that all $w \in \mathcal{W}$ satisfy (\ref{rank-condition}) for some $K < \infty$.  Further suppose that (\ref{eigenvalue-condition}) is satisfied for appropriate constants and $\mathcal{S}$ satisfies (\ref{sparsity-eigenvalue-condition}) for some $b_n = o(n)$.  Let  $\tau_n[\hat{\theta}(G_n) - \theta(G_\infty)] = \sqrt{n}[\lambda_r(A^{(n)})/n \rho_n - \lambda_r(w)]$.
	
	Then, conditions (\ref{sampling-dist-convergence}) and (\ref{subsampling-dist-convergence}) are satisfied.  
\end{proposition}
\begin{remark}
	For certain problems such as hypothesis testing, it may be of interest to consider function classes changing with $n$. The bounds in \citep{koltchinksii-gine-kernel-operator} are actually sharper than needed for finite-rank graphons, so one may asymptotically relax the conditions in (\ref{eigenvalue-condition}) at an appropriate rate.  
\end{remark}

\begin{remark}
	Uniform validity may be established for smooth functionals of eigenvalues or counts using uniform variants of the Delta Method or the Continuous Mapping Theorem.  See for example, \citet{kasy-uniform-delta-method}.
\end{remark}

\section{An Interesting Phenomenon Involving $\lambda_1(A^{(n)})/n\rho_n$}
\label{appendix-max-eigenvalue}
In Section \ref{subsampling-eigenvalues}, we argue that replacing $\rho_n$ with $\hat{\rho}_n$ leads a non-negligble perturbation. Intuitively, it would seem that a functional with an estimated sparsity parameter would have a larger variance. While simulations suggest that this intuition is correct for most eigenvalues and count functionals, to our surprise, they suggest the opposite is true for the maximum eigenvalue.  While intriguing, a theoretical investigation of this phenomenon is beyond the scope of the present paper.  
\begin{figure}[H] 
\centering
  \begin{subfigure}[b]{0.45\linewidth}
    \centering
    \includegraphics[width=0.95\linewidth]{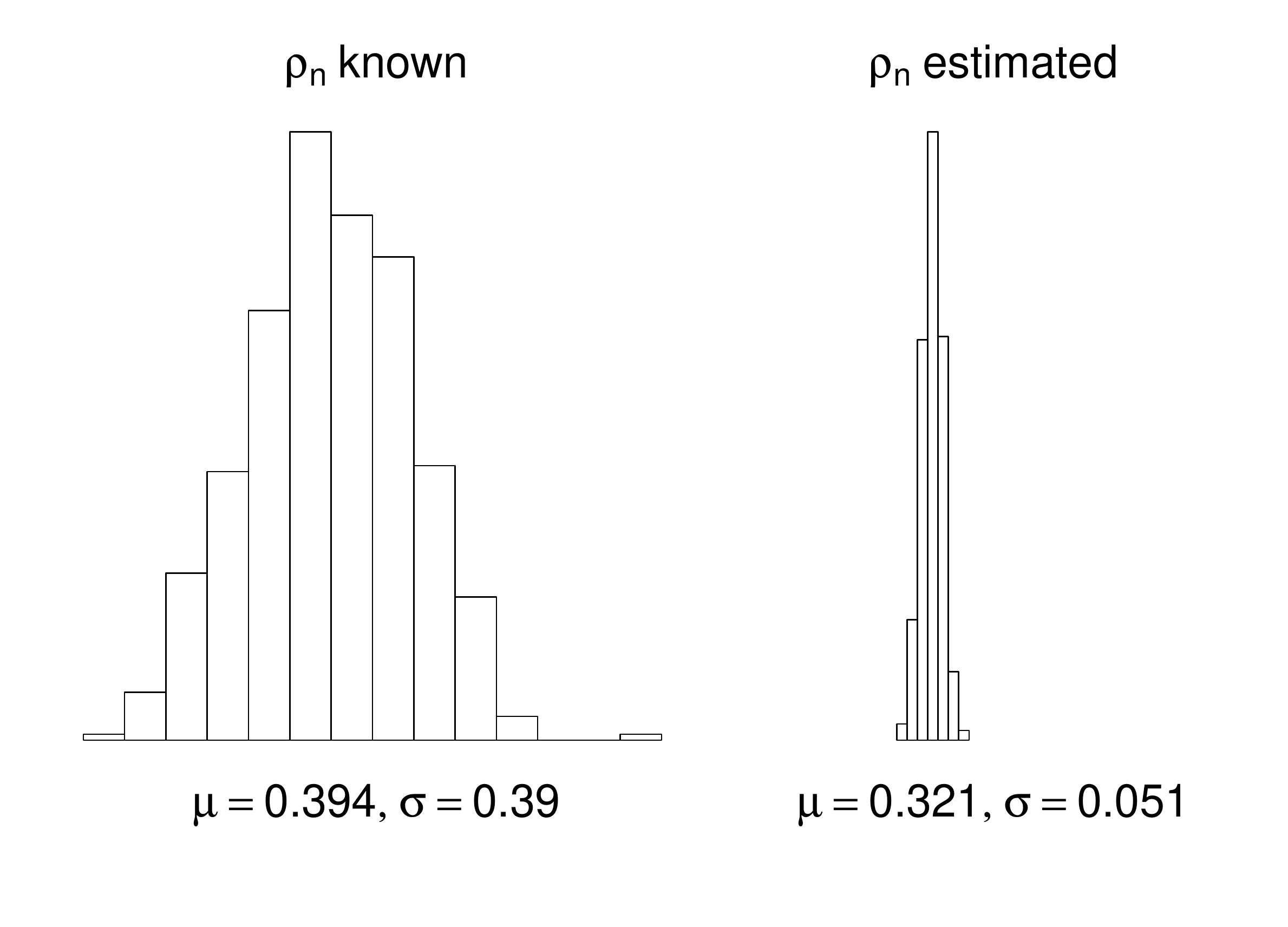} 
    \caption{2 Class Stochastic Block Model} 
    \vspace{4ex}
  \end{subfigure}
  \begin{subfigure}[b]{0.45\linewidth}
    \centering
    \includegraphics[width=0.95\linewidth]{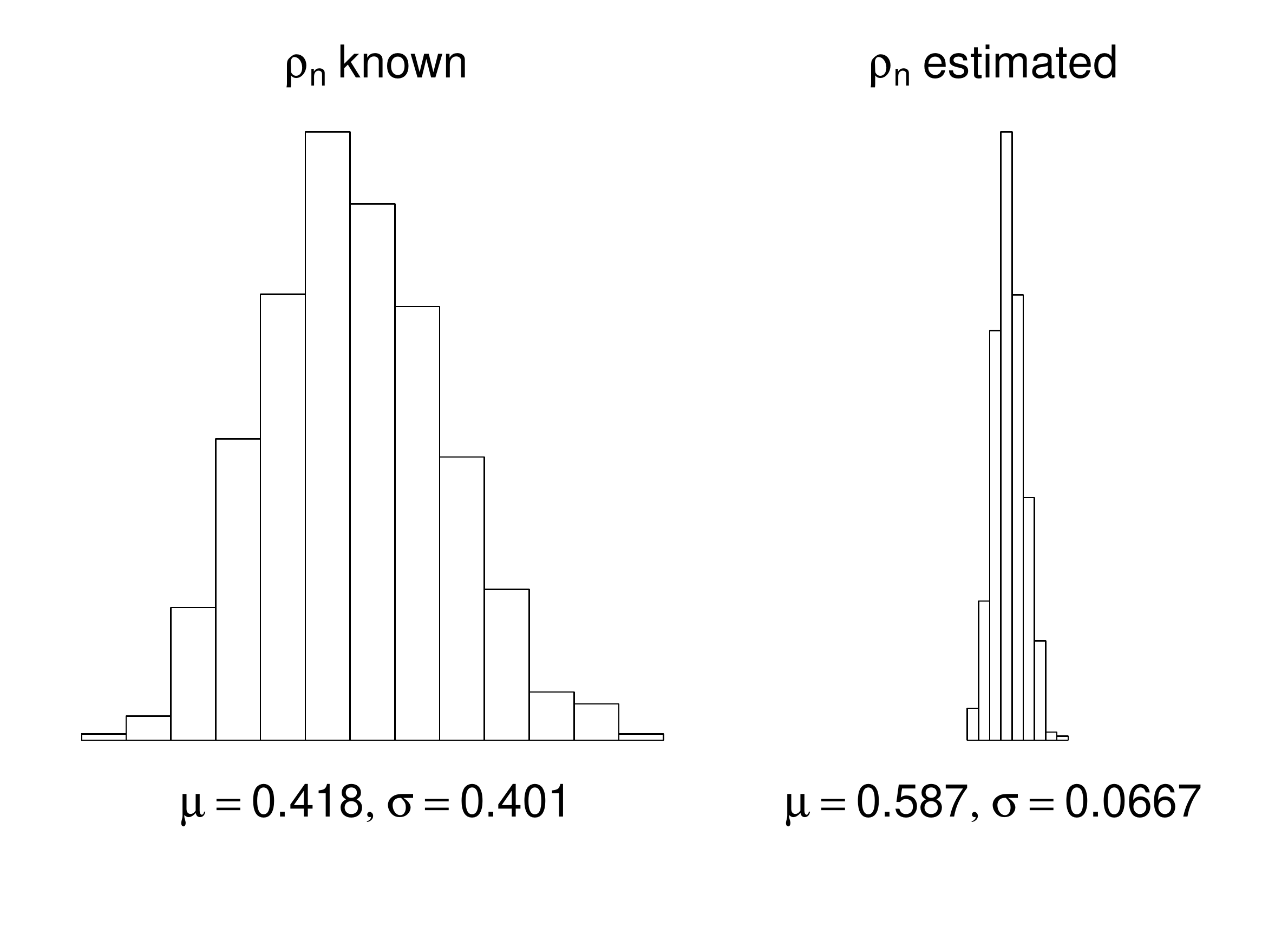} 
    \caption{3 Class Stochastic Block Model} 
    \vspace{4ex}
  \end{subfigure} 
  \begin{subfigure}[b]{0.45\linewidth}
    \centering
    \includegraphics[width=0.95\linewidth]{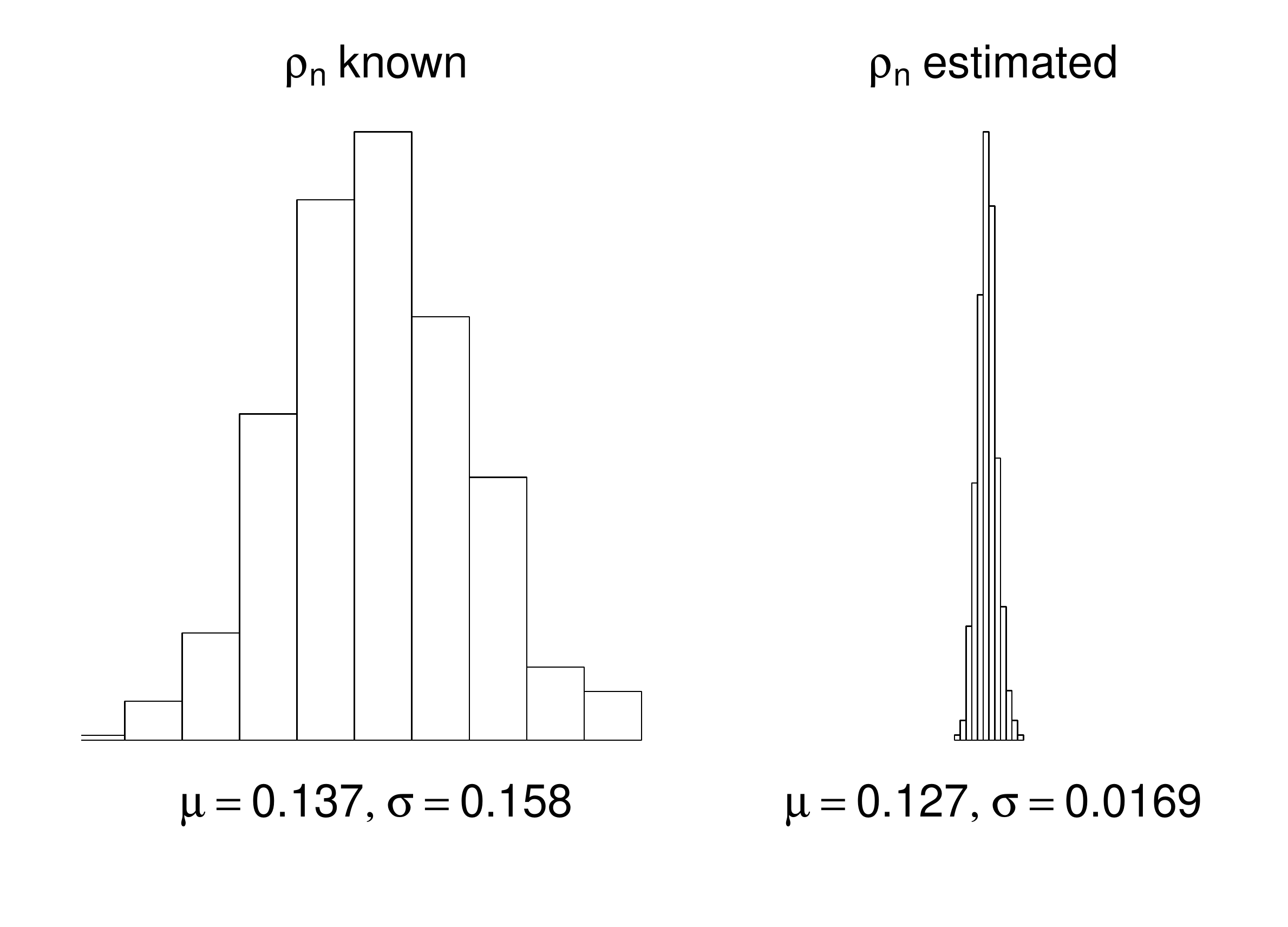} 
    \caption{Low-rank Graphon} 
  \end{subfigure}
  \begin{subfigure}[b]{0.45\linewidth}
    \centering
    \includegraphics[width=0.95\linewidth]{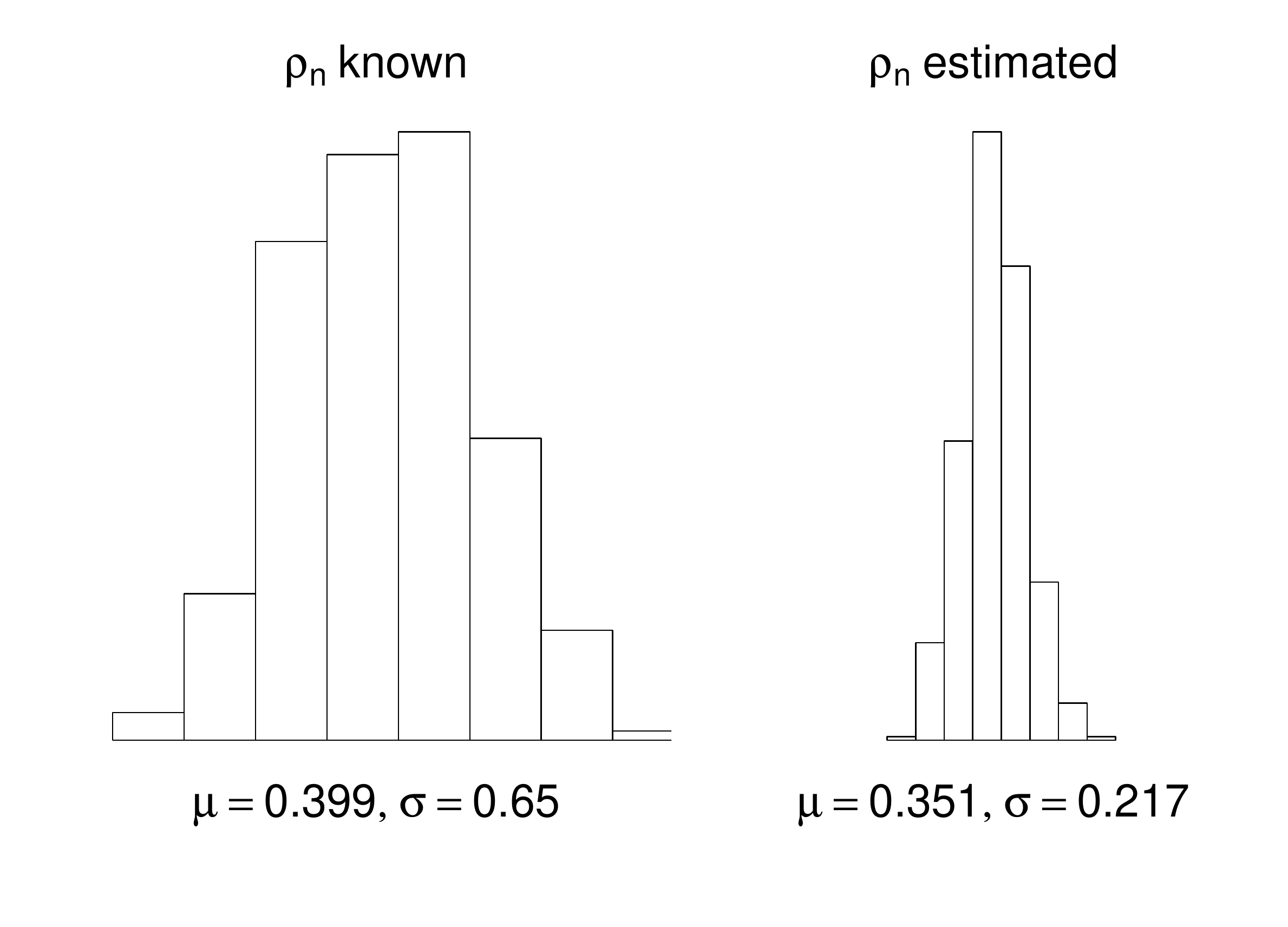} 
    \caption{Gaussian Latent Space Model} 
  \end{subfigure} 
  \caption{A comparison of the sampling distribution of the maximum eigenvalue normalized by known $\rho_n$ with  estimated $\rho_n$ for various graphon models.  The comparisons involve 500 samples of graphs with $n=1000$ vertices, with a sparsity parameter $\rho_n \asymp n^{-1/4}$.}
  \label{fig7} 
\end{figure}

\newpage

\section{Tables for Simulation Study}
\label{simulations-tables}

Here we provide detailed coverage results for subsampling and p-subsampling for the SBM (table~\ref{table:sbmall}) and the GLSM  (table~\ref{table:glsmall}) .

\begin{table}[!ht]
\centering
\resizebox{1.25\textwidth}{!}{\begin{tabular}{llllllllllllll}
                                          &                                     &                                                        & \multicolumn{5}{c}{Coverage for Vertex Subsampling}         &                                                      & \multicolumn{5}{c}{Coverage for $p$-subsampling}           \\ \hline
\multicolumn{1}{|l|}{Sparsity Level}      & \multicolumn{1}{l|}{Parameter}      & \multicolumn{1}{l|}{\diagbox[width=3.5em]{$b_n$}{$n$}} & $1000$ & 3250  & 5500  & 7750  & \multicolumn{1}{l|}{10000} & \multicolumn{1}{l|}{\diagbox[width=3em]{$p_n$}{$n$}} & 1000  & 3250  & 5500  & 7750  & \multicolumn{1}{l|}{10000} \\ \hline
\multicolumn{1}{|l|}{}                    & \multicolumn{1}{l|}{}               & \multicolumn{1}{l|}{$0.10n$}                           & 1.000  & 1.000 & 1.000 & 1.000 & \multicolumn{1}{l|}{1.000} & \multicolumn{1}{l|}{0.10}                            & 1.000 & 1.000 & 1.000 & 1.000 & \multicolumn{1}{l|}{1.000} \\
\multicolumn{1}{|l|}{}                    & \multicolumn{1}{l|}{}               & \multicolumn{1}{l|}{$0.15n$}                           & 1.000  & 1.000 & 1.000 & 1.000 & \multicolumn{1}{l|}{1.000} & \multicolumn{1}{l|}{0.15}                            & 1.000 & 1.000 & 1.000 & 1.000 & \multicolumn{1}{l|}{1.000} \\
\multicolumn{1}{|l|}{$\nu_n = n^{-0.1}$}  & \multicolumn{1}{l|}{$\lambda_1(w)$} & \multicolumn{1}{l|}{$0.20n$}                           & 1.000  & 1.000 & 1.000 & 1.000 & \multicolumn{1}{l|}{1.000} & \multicolumn{1}{l|}{0.20}                            & 1.000 & 1.000 & 1.000 & 1.000 & \multicolumn{1}{l|}{1.000} \\
\multicolumn{1}{|l|}{}                    & \multicolumn{1}{l|}{}               & \multicolumn{1}{l|}{$0.25n$}                           & 1.000  & 1.000 & 1.000 & 1.000 & \multicolumn{1}{l|}{1.000} & \multicolumn{1}{l|}{0.25}                            & 1.000 & 1.000 & 1.000 & 1.000 & \multicolumn{1}{l|}{1.000} \\
\multicolumn{1}{|l|}{}                    & \multicolumn{1}{l|}{}               & \multicolumn{1}{l|}{$0.30n$}                           & 1.000  & 1.000 & 1.000 & 1.000 & \multicolumn{1}{l|}{1.000} & \multicolumn{1}{l|}{0.30}                            & 1.000 & 1.000 & 1.000 & 1.000 & \multicolumn{1}{l|}{1.000} \\ \cline{2-14} 
\multicolumn{1}{|l|}{}                    & \multicolumn{1}{l|}{}               & \multicolumn{1}{l|}{$0.10n$}                           & 0.000  & 0.002 & 0.012 & 0.004 & \multicolumn{1}{l|}{0.014} & \multicolumn{1}{l|}{0.10}                            & 0.000 & 0.002 & 0.008 & 0.004 & \multicolumn{1}{l|}{0.010} \\
\multicolumn{1}{|l|}{}                    & \multicolumn{1}{l|}{}               & \multicolumn{1}{l|}{$0.15n$}                           & 0.002  & 0.068 & 0.140 & 0.192 & \multicolumn{1}{l|}{0.252} & \multicolumn{1}{l|}{0.15}                            & 0.000 & 0.072 & 0.152 & 0.206 & \multicolumn{1}{l|}{0.252} \\
\multicolumn{1}{|l|}{}                    & \multicolumn{1}{l|}{$\lambda_2(w)$} & \multicolumn{1}{l|}{$0.20n$}                           & 0.212  & 0.432 & 0.510 & 0.564 & \multicolumn{1}{l|}{0.560} & \multicolumn{1}{l|}{0.20}                            & 0.224 & 0.450 & 0.488 & 0.580 & \multicolumn{1}{l|}{0.584} \\
\multicolumn{1}{|l|}{}                    & \multicolumn{1}{l|}{}               & \multicolumn{1}{l|}{$0.25n$}                           & 0.578  & 0.754 & 0.798 & 0.832 & \multicolumn{1}{l|}{0.796} & \multicolumn{1}{l|}{0.25}                            & 0.582 & 0.748 & 0.822 & 0.814 & \multicolumn{1}{l|}{0.810} \\
\multicolumn{1}{|l|}{}                    & \multicolumn{1}{l|}{}               & \multicolumn{1}{l|}{$0.30n$}                           & 0.878  & 0.906 & 0.932 & 0.920 & \multicolumn{1}{l|}{0.898} & \multicolumn{1}{l|}{0.30}                            & 0.878 & 0.902 & 0.932 & 0.912 & \multicolumn{1}{l|}{0.904} \\ \hline
\multicolumn{1}{|l|}{}                    & \multicolumn{1}{l|}{}               & \multicolumn{1}{l|}{$0.10n$}                           & 0.968  & 0.976 & 0.998 & 0.998 & \multicolumn{1}{l|}{1.000} & \multicolumn{1}{l|}{0.10}                            & 0.738 & 0.988 & 0.996 & 0.996 & \multicolumn{1}{l|}{0.998} \\
\multicolumn{1}{|l|}{}                    & \multicolumn{1}{l|}{}               & \multicolumn{1}{l|}{$0.15n$}                           & 1.000  & 1.000 & 1.000 & 1.000 & \multicolumn{1}{l|}{1.000} & \multicolumn{1}{l|}{0.15}                            & 1.000 & 1.000 & 1.000 & 1.000 & \multicolumn{1}{l|}{1.000} \\
\multicolumn{1}{|l|}{$\nu_n = n^{-0.25}$} & \multicolumn{1}{l|}{$\lambda_1(w)$} & \multicolumn{1}{l|}{$0.20n$}                           & 1.000  & 1.000 & 1.000 & 1.000 & \multicolumn{1}{l|}{1.000} & \multicolumn{1}{l|}{0.20}                            & 1.000 & 1.000 & 1.000 & 1.000 & \multicolumn{1}{l|}{1.000} \\
\multicolumn{1}{|l|}{}                    & \multicolumn{1}{l|}{}               & \multicolumn{1}{l|}{$0.25n$}                           & 1.000  & 1.000 & 1.000 & 1.000 & \multicolumn{1}{l|}{1.000} & \multicolumn{1}{l|}{0.25}                            & 1.000 & 1.000 & 1.000 & 1.000 & \multicolumn{1}{l|}{1.000} \\
\multicolumn{1}{|l|}{}                    & \multicolumn{1}{l|}{}               & \multicolumn{1}{l|}{$0.30n$}                           & 1.000  & 1.000 & 1.000 & 1.000 & \multicolumn{1}{l|}{1.000} & \multicolumn{1}{l|}{0.30}                            & 1.000 & 1.000 & 1.000 & 1.000 & \multicolumn{1}{l|}{1.000} \\ \cline{2-14} 
\multicolumn{1}{|l|}{}                    & \multicolumn{1}{l|}{}               & \multicolumn{1}{l|}{$0.10n$}                           & 0.000  & 0.000 & 0.000 & 0.000 & \multicolumn{1}{l|}{0.000} & \multicolumn{1}{l|}{0.10}                            & 0.000 & 0.000 & 0.000 & 0.000 & \multicolumn{1}{l|}{0.000} \\
\multicolumn{1}{|l|}{}                    & \multicolumn{1}{l|}{}               & \multicolumn{1}{l|}{$0.15n$}                           & 0.000  & 0.000 & 0.000 & 0.000 & \multicolumn{1}{l|}{0.000} & \multicolumn{1}{l|}{0.15}                            & 0.000 & 0.000 & 0.000 & 0.000 & \multicolumn{1}{l|}{0.000} \\
\multicolumn{1}{|l|}{}                    & \multicolumn{1}{l|}{$\lambda_2(w)$} & \multicolumn{1}{l|}{$0.20n$}                           & 0.000  & 0.000 & 0.000 & 0.000 & \multicolumn{1}{l|}{0.000} & \multicolumn{1}{l|}{0.20}                            & 0.000 & 0.000 & 0.000 & 0.000 & \multicolumn{1}{l|}{0.000} \\
\multicolumn{1}{|l|}{}                    & \multicolumn{1}{l|}{}               & \multicolumn{1}{l|}{$0.25n$}                           & 0.012  & 0.012 & 0.006 & 0.006 & \multicolumn{1}{l|}{0.008} & \multicolumn{1}{l|}{0.25}                            & 0.022 & 0.018 & 0.004 & 0.000 & \multicolumn{1}{l|}{0.010} \\
\multicolumn{1}{|l|}{}                    & \multicolumn{1}{l|}{}               & \multicolumn{1}{l|}{$0.30n$}                           & 0.244  & 0.324 & 0.298 & 0.300 & \multicolumn{1}{l|}{0.290} & \multicolumn{1}{l|}{0.30}                            & 0.330 & 0.346 & 0.316 & 0.328 & \multicolumn{1}{l|}{0.286} \\ \hline
\multicolumn{1}{|l|}{}                    & \multicolumn{1}{l|}{}               & \multicolumn{1}{l|}{$0.10n$}                           & 0.000  & 0.000 & 0.000 & 0.000 & \multicolumn{1}{l|}{0.000} & \multicolumn{1}{l|}{0.10}                            & 0.000 & 0.000 & 0.000 & 0.000 & \multicolumn{1}{l|}{0.000} \\
\multicolumn{1}{|l|}{}                    & \multicolumn{1}{l|}{}               & \multicolumn{1}{l|}{$0.15n$}                           & 0.346  & 0.196 & 0.218 & 0.244 & \multicolumn{1}{l|}{0.244} & \multicolumn{1}{l|}{0.15}                            & 0.008 & 0.154 & 0.226 & 0.276 & \multicolumn{1}{l|}{0.256} \\
\multicolumn{1}{|l|}{$\nu_n = n^{-0.33}$} & \multicolumn{1}{l|}{$\lambda_1(w)$} & \multicolumn{1}{l|}{$0.20n$}                           & 1.000  & 0.998 & 0.996 & 0.998 & \multicolumn{1}{l|}{1.000} & \multicolumn{1}{l|}{0.20}                            & 0.974 & 0.990 & 1.000 & 0.996 & \multicolumn{1}{l|}{1.000} \\
\multicolumn{1}{|l|}{}                    & \multicolumn{1}{l|}{}               & \multicolumn{1}{l|}{$0.25n$}                           & 1.000  & 1.000 & 1.000 & 1.000 & \multicolumn{1}{l|}{1.000} & \multicolumn{1}{l|}{0.25}                            & 1.000 & 1.000 & 1.000 & 1.000 & \multicolumn{1}{l|}{1.000} \\
\multicolumn{1}{|l|}{}                    & \multicolumn{1}{l|}{}               & \multicolumn{1}{l|}{$0.30n$}                           & 1.000  & 1.000 & 1.000 & 1.000 & \multicolumn{1}{l|}{1.000} & \multicolumn{1}{l|}{0.30}                            & 1.000 & 1.000 & 1.000 & 1.000 & \multicolumn{1}{l|}{1.000} \\ \cline{2-14} 
\multicolumn{1}{|l|}{}                    & \multicolumn{1}{l|}{}               & \multicolumn{1}{l|}{$0.10n$}                           & 0.000  & 0.000 & 0.000 & 0.000 & \multicolumn{1}{l|}{0.000} & \multicolumn{1}{l|}{0.10}                            & 0.000 & 0.000 & 0.000 & 0.000 & \multicolumn{1}{l|}{0.000} \\
\multicolumn{1}{|l|}{}                    & \multicolumn{1}{l|}{}               & \multicolumn{1}{l|}{$0.15n$}                           & 0.000  & 0.000 & 0.000 & 0.000 & \multicolumn{1}{l|}{0.000} & \multicolumn{1}{l|}{0.15}                            & 0.000 & 0.000 & 0.000 & 0.000 & \multicolumn{1}{l|}{0.000} \\
\multicolumn{1}{|l|}{}                    & \multicolumn{1}{l|}{$\lambda_2(w)$} & \multicolumn{1}{l|}{$0.20n$}                           & 0.000  & 0.000 & 0.000 & 0.000 & \multicolumn{1}{l|}{0.000} & \multicolumn{1}{l|}{0.20}                            & 0.000 & 0.000 & 0.000 & 0.000 & \multicolumn{1}{l|}{0.000} \\
\multicolumn{1}{|l|}{}                    & \multicolumn{1}{l|}{}               & \multicolumn{1}{l|}{$0.25n$}                           & 0.000  & 0.000 & 0.000 & 0.000 & \multicolumn{1}{l|}{0.000} & \multicolumn{1}{l|}{0.25}                            & 0.016 & 0.000 & 0.000 & 0.000 & \multicolumn{1}{l|}{0.000} \\
\multicolumn{1}{|l|}{}                    & \multicolumn{1}{l|}{}               & \multicolumn{1}{l|}{$0.30n$}                           & 0.302  & 0.016 & 0.014 & 0.010 & \multicolumn{1}{l|}{0.008} & \multicolumn{1}{l|}{0.30}                            & 0.402 & 0.030  & 0.020  & 0.014 & \multicolumn{1}{l|}{0.008} \\ \hline
\multicolumn{1}{|l|}{}                    & \multicolumn{1}{l|}{}               & \multicolumn{1}{l|}{$0.10n$}                           & 0.000  & 0.000 & 0.000 & 0.000 & \multicolumn{1}{l|}{0.000} & \multicolumn{1}{l|}{0.10}                            & 0.000 & 0.000 & 0.000 & 0.000 & \multicolumn{1}{l|}{0.000} \\
\multicolumn{1}{|l|}{}                    & \multicolumn{1}{l|}{}               & \multicolumn{1}{l|}{$0.15n$}                           & 0.000  & 0.000 & 0.000 & 0.000 & \multicolumn{1}{l|}{0.000} & \multicolumn{1}{l|}{0.15}                            & 0.000 & 0.000 & 0.000 & 0.000 & \multicolumn{1}{l|}{0.000} \\
\multicolumn{1}{|l|}{$\nu_n = n^{-0.45}$} & \multicolumn{1}{l|}{$\lambda_1(w)$} & \multicolumn{1}{l|}{$0.20n$}                           & 0.000  & 0.000 & 0.000 & 0.000 & \multicolumn{1}{l|}{0.000} & \multicolumn{1}{l|}{0.20}                            & 0.000 & 0.000 & 0.000 & 0.000 & \multicolumn{1}{l|}{0.000} \\
\multicolumn{1}{|l|}{}                    & \multicolumn{1}{l|}{}               & \multicolumn{1}{l|}{$0.25n$}                           & 0.086  & 0.006 & 0.002 & 0.000 & \multicolumn{1}{l|}{0.000} & \multicolumn{1}{l|}{0.25}                            & 0.000 & 0.000 & 0.000 & 0.000 & \multicolumn{1}{l|}{0.000} \\
\multicolumn{1}{|l|}{}                    & \multicolumn{1}{l|}{}               & \multicolumn{1}{l|}{$0.30n$}                           & 0.936  & 0.906 & 0.930 & 0.940 & \multicolumn{1}{l|}{0.932} & \multicolumn{1}{l|}{0.30}                            & 0.108 & 0.844 & 0.928 & 0.934 & \multicolumn{1}{l|}{0.938} \\ \cline{2-14} 
\multicolumn{1}{|l|}{}                    & \multicolumn{1}{l|}{}               & \multicolumn{1}{l|}{$0.10n$}                           & 0.000  & 0.000 & 0.000 & 0.000 & \multicolumn{1}{l|}{0.000} & \multicolumn{1}{l|}{0.10}                            & 0.000 & 0.000 & 0.000 & 0.000 & \multicolumn{1}{l|}{0.000} \\
\multicolumn{1}{|l|}{}                    & \multicolumn{1}{l|}{}               & \multicolumn{1}{l|}{$0.15n$}                           & 0.042  & 0.000 & 0.000 & 0.000 & \multicolumn{1}{l|}{0.000} & \multicolumn{1}{l|}{0.15}                            & 0.000 & 0.000 & 0.000 & 0.000 & \multicolumn{1}{l|}{0.000} \\
\multicolumn{1}{|l|}{}                    & \multicolumn{1}{l|}{$\lambda_2(w)$} & \multicolumn{1}{l|}{$0.20n$}                           & 0.936  & 0.000 & 0.000 & 0.000 & \multicolumn{1}{l|}{0.000} & \multicolumn{1}{l|}{0.20}                            & 0.016 & 0.000 & 0.000 & 0.000 & \multicolumn{1}{l|}{0.000} \\
\multicolumn{1}{|l|}{}                    & \multicolumn{1}{l|}{}               & \multicolumn{1}{l|}{$0.25n$}                           & 1.000  & 0.000 & 0.000 & 0.000 & \multicolumn{1}{l|}{0.000} & \multicolumn{1}{l|}{0.25}                            & 1.000 & 0.000 & 0.000 & 0.000 & \multicolumn{1}{l|}{0.000} \\
\multicolumn{1}{|l|}{}                    & \multicolumn{1}{l|}{}               & \multicolumn{1}{l|}{$0.30n$}                           & 0.840  & 0.012 & 0.000 & 0.000 & \multicolumn{1}{l|}{0.000} & \multicolumn{1}{l|}{0.30}                            & 1.000 & 0.000 & 0.000 & 0.000 & \multicolumn{1}{l|}{0.000} \\ \hline
\end{tabular}}
\smallskip
\caption{\label{table:sbmall}Simulation results for 3-class SBM}
\end{table}

\begin{table}[htbp]
\centering
\resizebox{1.25\textwidth}{!}{\begin{tabular}{llllllllllllll}
                                          &                                     &                                                        & \multicolumn{5}{c}{Coverage for Vertex Subsampling}         &                                                      & \multicolumn{5}{c}{Coverage for $p$-subsampling}           \\ \hline
\multicolumn{1}{|l|}{Sparsity Level}      & \multicolumn{1}{l|}{Parameter}      & \multicolumn{1}{l|}{\diagbox[width=3.5em]{$b_n$}{$n$}} & $1000$ & 3250  & 5500  & 7750  & \multicolumn{1}{l|}{10000} & \multicolumn{1}{l|}{\diagbox[width=3em]{$p_n$}{$n$}} & 1000  & 3250  & 5500  & 7750  & \multicolumn{1}{l|}{10000} \\ \hline
\multicolumn{1}{|l|}{}                    & \multicolumn{1}{l|}{}               & \multicolumn{1}{l|}{$0.10n$}                           & 0.938  & 0.970 & 0.990 & 0.976 & \multicolumn{1}{l|}{0.982} & \multicolumn{1}{l|}{0.10}                            & 0.888 & 0.962 & 0.992 & 0.984 & \multicolumn{1}{l|}{0.986} \\
\multicolumn{1}{|l|}{}                    & \multicolumn{1}{l|}{}               & \multicolumn{1}{l|}{$0.15n$}                           & 0.956  & 0.988 & 0.978 & 0.980 & \multicolumn{1}{l|}{0.984} & \multicolumn{1}{l|}{0.15}                            & 0.910 & 0.984 & 0.976 & 0.984 & \multicolumn{1}{l|}{0.984} \\
\multicolumn{1}{|l|}{$\nu_n = n^{-0.1}$}  & \multicolumn{1}{l|}{$\lambda_1(w)$} & \multicolumn{1}{l|}{$0.20n$}                           & 0.962  & 0.986 & 0.982 & 0.966 & \multicolumn{1}{l|}{0.954} & \multicolumn{1}{l|}{0.20}                            & 0.920 & 0.978 & 0.982 & 0.972 & \multicolumn{1}{l|}{0.966} \\
\multicolumn{1}{|l|}{}                    & \multicolumn{1}{l|}{}               & \multicolumn{1}{l|}{$0.25n$}                           & 0.964  & 0.972 & 0.976 & 0.954 & \multicolumn{1}{l|}{0.948} & \multicolumn{1}{l|}{0.25}                            & 0.944 & 0.966 & 0.978 & 0.958 & \multicolumn{1}{l|}{0.948} \\
\multicolumn{1}{|l|}{}                    & \multicolumn{1}{l|}{}               & \multicolumn{1}{l|}{$0.30n$}                           & 0.972  & 0.968 & 0.982 & 0.962 & \multicolumn{1}{l|}{0.918} & \multicolumn{1}{l|}{0.30}                            & 0.954 & 0.970 & 0.982 & 0.966 & \multicolumn{1}{l|}{0.932} \\ \cline{2-14} 
\multicolumn{1}{|l|}{}                    & \multicolumn{1}{l|}{}               & \multicolumn{1}{l|}{$0.10n$}                           & 0.894  & 0.972 & 0.962 & 0.960 & \multicolumn{1}{l|}{0.974} & \multicolumn{1}{l|}{0.10}                            & 0.818 & 0.928 & 0.946 & 0.956 & \multicolumn{1}{l|}{0.964} \\
\multicolumn{1}{|l|}{}                    & \multicolumn{1}{l|}{}               & \multicolumn{1}{l|}{$0.15n$}                           & 0.940  & 0.972 & 0.966 & 0.972 & \multicolumn{1}{l|}{0.976} & \multicolumn{1}{l|}{0.15}                            & 0.894 & 0.970 & 0.962 & 0.966 & \multicolumn{1}{l|}{0.978} \\
\multicolumn{1}{|l|}{}                    & \multicolumn{1}{l|}{$\lambda_2(w)$} & \multicolumn{1}{l|}{$0.20n$}                           & 0.908  & 0.958 & 0.972 & 0.966 & \multicolumn{1}{l|}{0.960} & \multicolumn{1}{l|}{0.20}                            & 0.880 & 0.956 & 0.974 & 0.974 & \multicolumn{1}{l|}{0.958} \\
\multicolumn{1}{|l|}{}                    & \multicolumn{1}{l|}{}               & \multicolumn{1}{l|}{$0.25n$}                           & 0.938  & 0.976 & 0.966 & 0.958 & \multicolumn{1}{l|}{0.964} & \multicolumn{1}{l|}{0.25}                            & 0.906 & 0.966 & 0.966 & 0.956 & \multicolumn{1}{l|}{0.952} \\
\multicolumn{1}{|l|}{}                    & \multicolumn{1}{l|}{}               & \multicolumn{1}{l|}{$0.30n$}                           & 0.938  & 0.952 & 0.958 & 0.956 & \multicolumn{1}{l|}{0.962} & \multicolumn{1}{l|}{0.30}                            & 0.914 & 0.948 & 0.958 & 0.950 & \multicolumn{1}{l|}{0.962} \\ \cline{2-14} 
\multicolumn{1}{|l|}{}                    & \multicolumn{1}{l|}{}               & \multicolumn{1}{l|}{$0.10n$}                           & 0.000  & 0.034 & 0.946 & 0.782 & \multicolumn{1}{l|}{0.622} & \multicolumn{1}{l|}{0.10}                            & 0.678 & 0.686 & 0.668 & 0.622 & \multicolumn{1}{l|}{0.566} \\
\multicolumn{1}{|l|}{}                    & \multicolumn{1}{l|}{}               & \multicolumn{1}{l|}{$0.15n$}                           & 0.000  & 0.928 & 0.784 & 0.770 & \multicolumn{1}{l|}{0.676} & \multicolumn{1}{l|}{0.15}                            & 0.756 & 0.746 & 0.702 & 0.658 & \multicolumn{1}{l|}{0.624} \\
\multicolumn{1}{|l|}{}                    & \multicolumn{1}{l|}{$\lambda_3(w)$} & \multicolumn{1}{l|}{$0.20n$}                           & 0.000  & 0.884 & 0.744 & 0.724 & \multicolumn{1}{l|}{0.648} & \multicolumn{1}{l|}{0.20}                            & 0.774 & 0.724 & 0.716 & 0.702 & \multicolumn{1}{l|}{0.620} \\
\multicolumn{1}{|l|}{}                    & \multicolumn{1}{l|}{}               & \multicolumn{1}{l|}{$0.25n$}                           & 0.000  & 0.744 & 0.706 & 0.698 & \multicolumn{1}{l|}{0.618} & \multicolumn{1}{l|}{0.25}                            & 0.804 & 0.762 & 0.752 & 0.676 & \multicolumn{1}{l|}{0.598} \\
\multicolumn{1}{|l|}{}                    & \multicolumn{1}{l|}{}               & \multicolumn{1}{l|}{$0.30n$}                           & 0.158  & 0.714 & 0.728 & 0.624 & \multicolumn{1}{l|}{0.606} & \multicolumn{1}{l|}{0.30}                            & 0.800 & 0.774 & 0.750 & 0.678 & \multicolumn{1}{l|}{0.576} \\ \hline
\multicolumn{1}{|l|}{}                    & \multicolumn{1}{l|}{}               & \multicolumn{1}{l|}{$0.10n$}                           & 0.570  & 0.754 & 0.906 & 0.928 & \multicolumn{1}{l|}{0.962} & \multicolumn{1}{l|}{0.10}                            & 0.018 & 0.352 & 0.678 & 0.766 & \multicolumn{1}{l|}{0.878} \\
\multicolumn{1}{|l|}{}                    & \multicolumn{1}{l|}{}               & \multicolumn{1}{l|}{$0.15n$}                           & 0.806  & 0.916 & 0.954 & 0.984 & \multicolumn{1}{l|}{0.996} & \multicolumn{1}{l|}{0.15}                            & 0.252 & 0.764 & 0.862 & 0.932 & \multicolumn{1}{l|}{0.974} \\
\multicolumn{1}{|l|}{$\nu_n = n^{-0.25}$} & \multicolumn{1}{l|}{$\lambda_1(w)$} & \multicolumn{1}{l|}{$0.20n$}                           & 0.868  & 0.948 & 0.984 & 0.998 & \multicolumn{1}{l|}{0.994} & \multicolumn{1}{l|}{0.20}                            & 0.558 & 0.860 & 0.960 & 0.976 & \multicolumn{1}{l|}{0.988} \\
\multicolumn{1}{|l|}{}                    & \multicolumn{1}{l|}{}               & \multicolumn{1}{l|}{$0.25n$}                           & 0.954  & 0.974 & 0.980 & 0.982 & \multicolumn{1}{l|}{0.980} & \multicolumn{1}{l|}{0.25}                            & 0.746 & 0.920 & 0.972 & 0.978 & \multicolumn{1}{l|}{0.986} \\
\multicolumn{1}{|l|}{}                    & \multicolumn{1}{l|}{}               & \multicolumn{1}{l|}{$0.30n$}                           & 0.962  & 0.980 & 0.984 & 0.976 & \multicolumn{1}{l|}{0.946} & \multicolumn{1}{l|}{0.30}                            & 0.844 & 0.970 & 0.982 & 0.988 & \multicolumn{1}{l|}{0.978} \\ \cline{2-14} 
\multicolumn{1}{|l|}{}                    & \multicolumn{1}{l|}{}               & \multicolumn{1}{l|}{$0.10n$}                           & 0.444  & 0.546 & 0.638 & 0.724 & \multicolumn{1}{l|}{0.752} & \multicolumn{1}{l|}{0.10}                            & 0.004 & 0.112 & 0.278 & 0.404 & \multicolumn{1}{l|}{0.492} \\
\multicolumn{1}{|l|}{}                    & \multicolumn{1}{l|}{}               & \multicolumn{1}{l|}{$0.15n$}                           & 0.642  & 0.832 & 0.828 & 0.884 & \multicolumn{1}{l|}{0.912} & \multicolumn{1}{l|}{0.15}                            & 0.120 & 0.486 & 0.632 & 0.702 & \multicolumn{1}{l|}{0.772} \\
\multicolumn{1}{|l|}{}                    & \multicolumn{1}{l|}{$\lambda_2(w)$} & \multicolumn{1}{l|}{$0.20n$}                           & 0.804  & 0.890 & 0.918 & 0.932 & \multicolumn{1}{l|}{0.960} & \multicolumn{1}{l|}{0.20}                            & 0.410 & 0.728 & 0.776 & 0.842 & \multicolumn{1}{l|}{0.870} \\
\multicolumn{1}{|l|}{}                    & \multicolumn{1}{l|}{}               & \multicolumn{1}{l|}{$0.25n$}                           & 0.904  & 0.938 & 0.948 & 0.960 & \multicolumn{1}{l|}{0.968} & \multicolumn{1}{l|}{0.25}                            & 0.656 & 0.816 & 0.872 & 0.916 & \multicolumn{1}{l|}{0.924} \\
\multicolumn{1}{|l|}{}                    & \multicolumn{1}{l|}{}               & \multicolumn{1}{l|}{$0.30n$}                           & 0.904  & 0.960 & 0.966 & 0.966 & \multicolumn{1}{l|}{0.964} & \multicolumn{1}{l|}{0.30}                            & 0.746 & 0.902 & 0.926 & 0.942 & \multicolumn{1}{l|}{0.950} \\ \cline{2-14} 
\multicolumn{1}{|l|}{}                    & \multicolumn{1}{l|}{}               & \multicolumn{1}{l|}{$0.10n$}                           & 0.168  & 0.104 & 0.134 & 0.106 & \multicolumn{1}{l|}{0.098} & \multicolumn{1}{l|}{0.10}                            & 0.000 & 0.002 & 0.022 & 0.022 & \multicolumn{1}{l|}{0.000} \\
\multicolumn{1}{|l|}{}                    & \multicolumn{1}{l|}{}               & \multicolumn{1}{l|}{$0.15n$}                           & 0.362  & 0.350 & 0.348 & 0.296 & \multicolumn{1}{l|}{0.264} & \multicolumn{1}{l|}{0.15}                            & 0.016 & 0.114 & 0.126 & 0.130 & \multicolumn{1}{l|}{0.100} \\
\multicolumn{1}{|l|}{}                    & \multicolumn{1}{l|}{$\lambda_3(w)$} & \multicolumn{1}{l|}{$0.20n$}                           & 0.524  & 0.532 & 0.502 & 0.456 & \multicolumn{1}{l|}{0.418} & \multicolumn{1}{l|}{0.20}                            & 0.174 & 0.306 & 0.302 & 0.264 & \multicolumn{1}{l|}{0.280} \\
\multicolumn{1}{|l|}{}                    & \multicolumn{1}{l|}{}               & \multicolumn{1}{l|}{$0.25n$}                           & 0.716  & 0.618 & 0.620 & 0.540 & \multicolumn{1}{l|}{0.484} & \multicolumn{1}{l|}{0.25}                            & 0.384 & 0.446 & 0.408 & 0.374 & \multicolumn{1}{l|}{0.346} \\
\multicolumn{1}{|l|}{}                    & \multicolumn{1}{l|}{}               & \multicolumn{1}{l|}{$0.30n$}                           & 0.796  & 0.738 & 0.710 & 0.642 & \multicolumn{1}{l|}{0.570} & \multicolumn{1}{l|}{0.30}                            & 0.536 & 0.564 & 0.560 & 0.484 & \multicolumn{1}{l|}{0.438} \\ \hline
\multicolumn{1}{|l|}{}                    & \multicolumn{1}{l|}{}               & \multicolumn{1}{l|}{$0.10n$}                           & 0.124  & 0.122 & 0.238 & 0.364 & \multicolumn{1}{l|}{0.500} & \multicolumn{1}{l|}{0.10}                            & 0.000 & 0.000 & 0.002 & 0.010 & \multicolumn{1}{l|}{0.046} \\
\multicolumn{1}{|l|}{}                    & \multicolumn{1}{l|}{}               & \multicolumn{1}{l|}{$0.15n$}                           & 0.370  & 0.564 & 0.746 & 0.846 & \multicolumn{1}{l|}{0.874} & \multicolumn{1}{l|}{0.15}                            & 0.000 & 0.044 & 0.176 & 0.328 & \multicolumn{1}{l|}{0.450} \\
\multicolumn{1}{|l|}{$\nu_n = n^{-0.33}$} & \multicolumn{1}{l|}{$\lambda_1(w)$} & \multicolumn{1}{l|}{$0.20n$}                           & 0.686  & 0.850 & 0.908 & 0.956 & \multicolumn{1}{l|}{0.970} & \multicolumn{1}{l|}{0.20}                            & 0.030 & 0.324 & 0.534 & 0.696 & \multicolumn{1}{l|}{0.786} \\
\multicolumn{1}{|l|}{}                    & \multicolumn{1}{l|}{}               & \multicolumn{1}{l|}{$0.25n$}                           & 0.822  & 0.940 & 0.970 & 0.984 & \multicolumn{1}{l|}{0.982} & \multicolumn{1}{l|}{0.25}                            & 0.164 & 0.624 & 0.786 & 0.842 & \multicolumn{1}{l|}{0.904} \\
\multicolumn{1}{|l|}{}                    & \multicolumn{1}{l|}{}               & \multicolumn{1}{l|}{$0.30n$}                           & 0.916  & 0.958 & 0.996 & 0.986 & \multicolumn{1}{l|}{0.978} & \multicolumn{1}{l|}{0.30}                            & 0.432 & 0.780 & 0.908 & 0.950 & \multicolumn{1}{l|}{0.966} \\ \cline{2-14} 
\multicolumn{1}{|l|}{}                    & \multicolumn{1}{l|}{}               & \multicolumn{1}{l|}{$0.10n$}                           & 0.058  & 0.010 & 0.020 & 0.024 & \multicolumn{1}{l|}{0.004} & \multicolumn{1}{l|}{0.10}                            & 0.000 & 0.000 & 0.000 & 0.000 & \multicolumn{1}{l|}{0.000} \\
\multicolumn{1}{|l|}{}                    & \multicolumn{1}{l|}{}               & \multicolumn{1}{l|}{$0.15n$}                           & 0.150  & 0.178 & 0.270 & 0.324 & \multicolumn{1}{l|}{0.466} & \multicolumn{1}{l|}{0.15}                            & 0.000 & 0.002 & 0.008 & 0.034 & \multicolumn{1}{l|}{0.046} \\
\multicolumn{1}{|l|}{}                    & \multicolumn{1}{l|}{$\lambda_2(w)$} & \multicolumn{1}{l|}{$0.20n$}                           & 0.408  & 0.512 & 0.678 & 0.736 & \multicolumn{1}{l|}{0.740} & \multicolumn{1}{l|}{0.20}                            & 0.000 & 0.068 & 0.174 & 0.254 & \multicolumn{1}{l|}{0.304} \\
\multicolumn{1}{|l|}{}                    & \multicolumn{1}{l|}{}               & \multicolumn{1}{l|}{$0.25n$}                           & 0.668  & 0.814 & 0.848 & 0.860 & \multicolumn{1}{l|}{0.882} & \multicolumn{1}{l|}{0.25}                            & 0.026 & 0.316 & 0.436 & 0.554 & \multicolumn{1}{l|}{0.594} \\
\multicolumn{1}{|l|}{}                    & \multicolumn{1}{l|}{}               & \multicolumn{1}{l|}{$0.30n$}                           & 0.816  & 0.932 & 0.940 & 0.960 & \multicolumn{1}{l|}{0.942} & \multicolumn{1}{l|}{0.30}                            & 0.214 & 0.572 & 0.682 & 0.776 & \multicolumn{1}{l|}{0.788} \\ \cline{2-14} 
\multicolumn{1}{|l|}{}                    & \multicolumn{1}{l|}{}               & \multicolumn{1}{l|}{$0.10n$}                           & 0.026  & 0.000 & 0.000 & 0.000 & \multicolumn{1}{l|}{0.000} & \multicolumn{1}{l|}{0.10}                            & 0.000 & 0.000 & 0.000 & 0.000 & \multicolumn{1}{l|}{0.000} \\
\multicolumn{1}{|l|}{}                    & \multicolumn{1}{l|}{}               & \multicolumn{1}{l|}{$0.15n$}                           & 0.038  & 0.002 & 0.004 & 0.006 & \multicolumn{1}{l|}{0.006} & \multicolumn{1}{l|}{0.15}                            & 0.000 & 0.000 & 0.000 & 0.000 & \multicolumn{1}{l|}{0.000} \\
\multicolumn{1}{|l|}{}                    & \multicolumn{1}{l|}{$\lambda_3(w)$} & \multicolumn{1}{l|}{$0.20n$}                           & 0.152  & 0.100 & 0.104 & 0.094 & \multicolumn{1}{l|}{0.084} & \multicolumn{1}{l|}{0.20}                            & 0.000 & 0.000 & 0.000 & 0.006 & \multicolumn{1}{l|}{0.002} \\
\multicolumn{1}{|l|}{}                    & \multicolumn{1}{l|}{}               & \multicolumn{1}{l|}{$0.25n$}                           & 0.412  & 0.330 & 0.340 & 0.274 & \multicolumn{1}{l|}{0.252} & \multicolumn{1}{l|}{0.25}                            & 0.004 & 0.060 & 0.054 & 0.046 & \multicolumn{1}{l|}{0.058} \\
\multicolumn{1}{|l|}{}                    & \multicolumn{1}{l|}{}               & \multicolumn{1}{l|}{$0.30n$}                           & 0.566  & 0.634 & 0.494 & 0.512 & \multicolumn{1}{l|}{0.440} & \multicolumn{1}{l|}{0.30}                            & 0.064 & 0.232 & 0.154 & 0.200 & \multicolumn{1}{l|}{0.146} \\ \hline
\multicolumn{1}{|l|}{}                    & \multicolumn{1}{l|}{}               & \multicolumn{1}{l|}{$0.10n$}                           & 0.000  & 0.000 & 0.000 & 0.000 & \multicolumn{1}{l|}{0.000} & \multicolumn{1}{l|}{0.10}                            & 0.000 & 0.000 & 0.000 & 0.000 & \multicolumn{1}{l|}{0.000} \\
\multicolumn{1}{|l|}{}                    & \multicolumn{1}{l|}{}               & \multicolumn{1}{l|}{$0.15n$}                           & 0.020  & 0.000 & 0.000 & 0.000 & \multicolumn{1}{l|}{0.000} & \multicolumn{1}{l|}{0.15}                            & 0.000 & 0.000 & 0.000 & 0.000 & \multicolumn{1}{l|}{0.000} \\
\multicolumn{1}{|l|}{$\nu_n = n^{-0.45}$} & \multicolumn{1}{l|}{$\lambda_1(w)$} & \multicolumn{1}{l|}{$0.20n$}                           & 0.094  & 0.012 & 0.010 & 0.020 & \multicolumn{1}{l|}{0.040} & \multicolumn{1}{l|}{0.20}                            & 0.000 & 0.000 & 0.000 & 0.000 & \multicolumn{1}{l|}{0.000} \\
\multicolumn{1}{|l|}{}                    & \multicolumn{1}{l|}{}               & \multicolumn{1}{l|}{$0.25n$}                           & 0.282  & 0.214 & 0.284 & 0.394 & \multicolumn{1}{l|}{0.508} & \multicolumn{1}{l|}{0.25}                            & 0.000 & 0.000 & 0.000 & 0.000 & \multicolumn{1}{l|}{0.000} \\
\multicolumn{1}{|l|}{}                    & \multicolumn{1}{l|}{}               & \multicolumn{1}{l|}{$0.30n$}                           & 0.588  & 0.618 & 0.776 & 0.812 & \multicolumn{1}{l|}{0.866} & \multicolumn{1}{l|}{0.30}                            & 0.000 & 0.000 & 0.002 & 0.008 & \multicolumn{1}{l|}{0.020} \\ \cline{2-14} 
\multicolumn{1}{|l|}{}                    & \multicolumn{1}{l|}{}               & \multicolumn{1}{l|}{$0.10n$}                           & 0.986  & 0.000 & 0.000 & 0.000 & \multicolumn{1}{l|}{0.000} & \multicolumn{1}{l|}{0.10}                            & 0.830 & 0.000 & 0.000 & 0.000 & \multicolumn{1}{l|}{0.000} \\
\multicolumn{1}{|l|}{}                    & \multicolumn{1}{l|}{}               & \multicolumn{1}{l|}{$0.15n$}                           & 0.004  & 0.000 & 0.000 & 0.000 & \multicolumn{1}{l|}{0.000} & \multicolumn{1}{l|}{0.15}                            & 0.000 & 0.000 & 0.000 & 0.000 & \multicolumn{1}{l|}{0.000} \\
\multicolumn{1}{|l|}{}                    & \multicolumn{1}{l|}{$\lambda_2(w)$} & \multicolumn{1}{l|}{$0.20n$}                           & 0.028  & 0.000 & 0.000 & 0.000 & \multicolumn{1}{l|}{0.000} & \multicolumn{1}{l|}{0.20}                            & 0.000 & 0.000 & 0.000 & 0.000 & \multicolumn{1}{l|}{0.000} \\
\multicolumn{1}{|l|}{}                    & \multicolumn{1}{l|}{}               & \multicolumn{1}{l|}{$0.25n$}                           & 0.116  & 0.010 & 0.018 & 0.014 & \multicolumn{1}{l|}{0.036} & \multicolumn{1}{l|}{0.25}                            & 0.000 & 0.000 & 0.000 & 0.000 & \multicolumn{1}{l|}{0.000} \\
\multicolumn{1}{|l|}{}                    & \multicolumn{1}{l|}{}               & \multicolumn{1}{l|}{$0.30n$}                           & 0.412  & 0.212 & 0.282 & 0.366 & \multicolumn{1}{l|}{0.424} & \multicolumn{1}{l|}{0.30}                            & 0.000 & 0.000 & 0.000 & 0.000 & \multicolumn{1}{l|}{0.000} \\ \cline{2-14} 
\multicolumn{1}{|l|}{}                    & \multicolumn{1}{l|}{}               & \multicolumn{1}{l|}{$0.10n$}                           & 1.000  & 0.000 & 0.000 & 0.000 & \multicolumn{1}{l|}{0.000} & \multicolumn{1}{l|}{0.10}                            & 1.000 & 0.000 & 0.000 & 0.000 & \multicolumn{1}{l|}{0.000} \\
\multicolumn{1}{|l|}{}                    & \multicolumn{1}{l|}{}               & \multicolumn{1}{l|}{$0.15n$}                           & 0.018  & 0.000 & 0.000 & 0.000 & \multicolumn{1}{l|}{0.000} & \multicolumn{1}{l|}{0.15}                            & 0.000 & 0.000 & 0.000 & 0.000 & \multicolumn{1}{l|}{0.000} \\
\multicolumn{1}{|l|}{}                    & \multicolumn{1}{l|}{$\lambda_3(w)$} & \multicolumn{1}{l|}{$0.20n$}                           & 0.020  & 0.000 & 0.000 & 0.000 & \multicolumn{1}{l|}{0.000} & \multicolumn{1}{l|}{0.20}                            & 0.000 & 0.000 & 0.000 & 0.000 & \multicolumn{1}{l|}{0.000} \\
\multicolumn{1}{|l|}{}                    & \multicolumn{1}{l|}{}               & \multicolumn{1}{l|}{$0.25n$}                           & 0.098  & 0.000 & 0.000 & 0.000 & \multicolumn{1}{l|}{0.000} & \multicolumn{1}{l|}{0.25}                            & 0.000 & 0.000 & 0.000 & 0.000 & \multicolumn{1}{l|}{0.000} \\
\multicolumn{1}{|l|}{}                    & \multicolumn{1}{l|}{}               & \multicolumn{1}{l|}{$0.30n$}                           & 0.248  & 0.008 & 0.002 & 0.004 & \multicolumn{1}{l|}{0.004} & \multicolumn{1}{l|}{0.30}                            & 0.000 & 0.000 & 0.000 & 0.000 & \multicolumn{1}{l|}{0.000} \\ \hline
\end{tabular}}
\caption{\label{table:glsmall} Simulation results for Gaussian latent space model}
\end{table}

\end{document}